\documentclass{article}

\usepackage{todonotes}

\usepackage[utf8]{inputenc}
\usepackage{amsmath, amsthm, amsfonts, mathrsfs, amsfonts, amssymb}
\usepackage[british]{babel}
\usepackage{caption}
\usepackage[nocompress]{cite}
\usepackage{graphicx}
\usepackage[export]{adjustbox}
\usepackage{listings}
\usepackage{mathtools}
\usepackage{subcaption}
\usepackage{tikz}
\usetikzlibrary{calligraphy,hobby}
\usepackage{titling}
\usepackage{url}

\usepackage[hidelinks]{hyperref}

\numberwithin{equation}{section}

\DeclarePairedDelimiter\abs{\lvert}{\rvert}

\DeclarePairedDelimiter{\norm}\lVert\rVert

\usepackage[margin=2.5cm]{geometry}

\def\Re{\mathop\mathrm{Re}\nolimits}			% real part
\newcommand{\Wholes}{\mathbb{Z}}							% integers
\newtheorem{thm}{Theorem}[section]
\newtheorem{cor}[thm]{Corollary}
\newtheorem{lem}[thm]{Lemma}
\newtheorem{prop}[thm]{Proposition}

\newtheorem{remark}{Remark}
\newtheorem*{defn*}{Definition}

\usepackage{authblk}
\title{\sc Conditional Speed and Shape Corrections for Travelling 
Wave Solutions to
Stochastically Perturbed Reaction-Diffusion Systems}

\author[1]{Mark van den Bosch \thanks{\tt vandenboschm@math.leidenuniv.nl}}
\author[2,3]{Christian H.S. Hamster \thanks{\tt c.h.s.hamster@uva.nl}}
\author[1]{Hermen Jan Hupkes \thanks{\tt hhupkes@math.leidenuniv.nl}}
\affil[1]{\small Mathematisch Instituut, Universiteit Leiden, P.O. Box 9512, 2300 RA Leiden, The Netherlands}
\affil[2]{\small KdV Institute, Universiteit van Amsterdam, P.O. Box 94248, 1090 GE Amsterdam, The Netherlands}
\affil[3]{\small Dutch Institute for Emergent Phenomena, Universiteit van Amsterdam, P.O. Box 94248, 1090 GE Amsterdam, The Netherlands}

\begin{document}
\newpage
\maketitle

\begin{abstract}
    In this work we perform rigorous small noise expansions to
    study the impact of stochastic forcing on the behaviour of planar travelling wave solutions to reaction-diffusion equations on cylindrical domains. In particular, we use a stochastic freezing approach that allows effective limiting information
    to be extracted concerning the behaviour of the
    stochastic perturbations from the deterministic wave. As an application, this allows us to provide a rigorous definition for the stochastic corrections to the wave speed. In addition, our approach allows their size to be computed to any desired order in the noise strength, provided that sufficient smoothness is available.
\end{abstract}

\,



\,

\section{Introduction}

The main purpose of this paper is to provide a rigorous underpinning
to the estimates obtained in the series \cite{hamster2020transinv,bosch2024multidimensional} for the speed and shape corrections that arise
when stochastically perturbing travelling wave solutions to
reaction-diffusion systems.
In particular, we consider SPDEs such as
\begin{equation}
\label{eq:int:main:spde}
   \mathrm dU = [\Delta U + f(U)]  \mathrm dt + \sigma g(U)  \mathrm  d W^Q_t
\end{equation}
on cylindrical domains of the form $\mathcal{D} = \mathbb{R} \times \mathbb T^{d-1}$ for some integer $d \ge 1$, driven by a cylindrical $Q$-Wiener process $(W^Q_t)_{t \ge 0}$ that is white in time,  coloured in space and translationally invariant.
We assume that in the deterministic case $\sigma = 0$
\eqref{eq:int:main:spde} features a planar travelling wave
solution $u(x,x_\perp,t) = \Phi_0(x - c_0 t)$.

\paragraph{Stochastic freezing}
The main contribution in the earlier works \cite{hamster2020transinv,bosch2024multidimensional} is that we constructed a stochastic phase function $\Gamma: [0,T] \to \mathbb R$ so that the perturbation
\begin{equation}
\label{eq:int:def:V}
    V(x,x_\perp,t) = U( x + \Gamma(t), x_\perp, t) - \Phi_0(x)
\end{equation}
remains bounded with high probability over time intervals that are exponentially long with respect to $\sigma^{-1}$
and admits a natural Taylor expansion
\begin{equation}
\label{eq:int:tay:v}
    V = \sigma V^{(1)} + \sigma^2 V^{(2)} + \ldots
\end{equation}
with respect to the variable $\sigma$.
In addition, we explicitly characterized the processes $V^{(1)}$ and $V^{(2)}$
and showed that the resulting expressions converged in expectation as $t \to \infty$ to a well-defined limit.
This leads formally to a Taylor expansion of the form
\begin{equation}
    C_{\mathrm{obs}} = c_0 + \sigma^2 c^{(2)} + \sigma^3 c^{(3)} + \ldots
\end{equation}
for the expected observed wave speed $C_{\mathrm{obs}}$, which
intuitively should satisfy a relation of the form
\begin{equation}
\label{eq:int:intuitive:def:c:obs}
    C_{\mathrm{obs}} \sim \frac{2}{T} \mathbb E \big[ \Gamma(T) - \Gamma(\tfrac{1}{2}T) \big] \hbox{ for large } T.
\end{equation}
We provided an explicit expression for $c^{(2)}$ and a numerical procedure to compute $c^{(3)}$, and validated our findings for several example systems by analyzing large samples of numerical simulations.
In essence, this approach can be seen as a stochastic version of the freezing approach developed by
Beyn \cite{beyn2004freezing}, which allows us to adopt the spirit behind the modern machinery for deterministic stability
issues initiated by Howard and Zumbrun \cite{ZUMHOW1998}.
However, a key missing feature in these earlier works is the ability to provide a rigorous quantification for statements of the form \eqref{eq:int:intuitive:def:c:obs}.

\paragraph{Rigorous expansion}
In this paper we provide such a quantification, not only
for the wave speed but for all sufficiently smooth functionals acting on $U$. In particular, assuming throughout this introduction that $V(0) = 0$, we formalize the expansion procedure behind  \eqref{eq:int:tay:v} by writing
\begin{equation}
    V = \sigma V^{(1)} + \sigma^2 V^{(2)} + \ldots + \sigma^{r-1} V^{(r-1)} + Z
\end{equation}
and establish bounds for the expansion terms $V^{(i)}$ and the residual term $Z$. This allows us to quantify the timescale over which the residual remains small, which in our case will mean $Z = O(\sigma^{r-\frac{1}{2}})$.

Writing $(S(t))_{t \ge 0}$ for the deterministic semigroup that
governs the linearized behaviour of \eqref{eq:int:main:spde} with $\sigma = 0$ near the travelling wave $(\Phi_0,c_0)$,
we have the representation 
\begin{equation}
    V^{(1)}(t) = \int_0^t S(t-s) \varrho \, \mathrm d W^Q_s
\end{equation}
for some fixed Hilbert-Schmidt operator $\varrho$.
Using techniques from \cite{hamster2020expstability},
which have recently been extended to a Banach-space setting in \cite{cox2024sharp},
it is possible to obtain the scaling behaviour
\begin{equation}
\mathbb E \sup_{0 \le t \le T} \norm{ V^{(1)}(t) }^{2p}
\sim p^p +  [ \ln T ]^p
\end{equation}
with respect to the exponent $p \ge 1$ and the timescale $T \ge 1$.
These computations exploit the fact that the constant function $s \mapsto \varrho$ naturally admits a global pathwise bound. However, for $V^{(2)}$ and higher one needs to understand expressions of the form
\begin{equation}
    \int_0^t S(t-s) \Lambda[ V^{(1)}(s)] \, \mathrm d W^Q_s,
\end{equation}
which involve (multi)-linear maps $\Lambda$.
In particular, we no longer have global pathwise bounds on the integrand, requiring us to generalize the techniques from \cite{hamster2020expstability}.
In fact, we also need to consider similar stochastic convolutions
where the semigroup $S$ is replaced by a random evolution family.
The resulting expressions are referred to as `forward integrals' \cite{russo1991integrales} and generalize the concept of It{\^o} integrals to situations where the integrand is `anticipating' instead of predictable.
Our growth bounds for such integrals sharpen the results
from \cite{bosch2024multidimensional}, which also required uniform pathwise estimates on the integrand.

These novel bounds can be used to obtain the scaling behaviour
\begin{equation}
\label{eq:int:full:sc:beh:V}
\mathbb E \sup_{0 \le t \le T} \norm{ V^{(j)}(t) }^{2p}
\sim p^{jp} + [ \ln T ]^{jp}.
\end{equation}
Setting $p = 1$, we can observe that
the natural expansion parameter is $\sigma \sqrt{\ln T}$
rather than $\sigma$. In particular, the timescales over which our expansions can be maintained should satisfy $\ln T \ll  \sigma^{-2}$.
In fact, we will use $\ln T_* = \sigma^{-\alpha}$ for some appropriate $0 < \alpha < 2$. This provides us enough freedom to achieve the stated $\sigma^{r-\frac{1}{2}}$ behaviour of the residual, while retaining timescales that are exponentially long with respect to $\sigma^{-1}$.
In particular, we proceed by introducing the so-called `stability event'
\begin{equation}
    \mathcal{A}_{\rm stb}
    = \{ \norm{ V^{(j)} } \le \sigma^{-1/2} \hbox{ and } \norm{Z} \le \sigma^{r- 1/2} \hbox{ on } [0, T_*] \},
\end{equation}
where we are deliberately not specifying any norms.
The bounds \eqref{eq:int:full:sc:beh:V} for $V^{(j)}$ and
related estimates for $Z$ can be used
to show that $\mathcal{A}_{\rm stb}$
occurs with high probability.

\paragraph{Conditional expectations}
One of the main goals of the paper is
to establish conditional expectation results of the form
\begin{equation}
\label{eq:int:lim:exp:phi}
    \mathbb E \big[ \phi\big( V(T_*) \big) | \mathcal{A}_{\rm stb}  \big]
    = h_{\infty;0} + \sigma h_{\infty;1}
    + \ldots + \sigma^{r-1} h_{\infty;r-1} + O ( \sigma^{r-1/2} ).
\end{equation}
In particular, for smooth functionals $\phi$
we show that the expectation $\phi(V)$, evaluated at the (large) time $T_*$ and conditioned on retaining the stability of the wave, admits a natural Taylor expansion.
Naturally, this expectation will depend on the precise details concerning the definition of the stability event $\mathcal{A}_{\rm stb}$, most notably the parameters underlying the exit time. However,
the expansion coefficients $\big(h_{\infty;0}, \ldots h_{\infty;r-1} \big)$ do not and are hence in some sense `universal'.
In this sense our results are able to provide a fully rigorous meaning to statements of the form \eqref{eq:int:intuitive:def:c:obs}.

In addition to our existence results,
we provide an algorithmic technique to obtain explicit expressions for these coefficients, generalizing the computations in the early work \cite{hamster2019stability,hamster2020transinv}.
The key step is to understand expressions of the form
\begin{equation}
\label{eq:int:multi:linear:lims}
    \lim_{t \to \infty} \mathbb E \Lambda [ V^{(i_1)}, \ldots, V^{(i_{\ell})}](t)
\end{equation}
for $\ell$-linear maps $\Lambda$. In particular, we establish that these limits are well-defined and can be evaluated explicitly in terms of iterated integrals involving the semigroup $S(t)$.
The remaining step towards \eqref{eq:int:lim:exp:phi} is to
use our bounds for the residual $Z$ and
the high-probability nature of $\mathcal{A}_{\rm stb}$ to obtain the error estimate. In particular, conditioning on $\mathcal{A}_{\rm stb}$
only has a small effect on the behaviour of
\eqref{eq:int:multi:linear:lims},
which can be quantified and captured by the error term.

We emphasize that the existence of the limits \eqref{eq:int:multi:linear:lims} is a direct consequence of the `freezing' technique
underlying the Ansatz \eqref{eq:int:def:V}. An alternative approach towards stochastic meta-stability results
\cite{maclaurinnew} requires $V$ to `move' with the wave. This makes the stability analysis somewhat less complicated since delicate terms involving higher derivatives of $V$, coming from the It\^o calculus, are avoided. However, the effects generated by these terms are essential for the Taylor expansions we obtain here.

\paragraph{Broader context}

This work is part of a growing initiative spanning several research groups to study how stochastic forcing impacts the pattern-forming properties of various types of deterministic systems, motivated by applications in fields such as  neuroscience
\cite{Bressloff,bressloff2015nonlinear},
cardiology \cite{Zhang}, finance  \cite{NunnoAdvMathFinance2011} and meteorology \cite{Climate}.
This topic has attracted
significant interest both from the applied community
\cite{Armero1996,GarciaSpatiallyExtended,schimansky1991,vinals1991numerical,shardlow,birzu2018fluctuations}
as well as mathematicians developing rigorous proofs
\cite{KUEHN2013KPP,KUEHN2013FRAME,Kuske2017,gowda2015early,Bloemker,Inglis,maclaurinnew}.
An extensive discussion of these developments
can be found in the review paper \cite{KuehnReview} and the recent work \cite{bosch2024multidimensional}.
These results build on the voluminous work performed during the
the past decades to provide a mathematically rigorous framework to interpret SPDEs; see, e.g.,
\cite{Chow,prevot2007concise,Gawarecki,DaPratoZab,hairer2009introduction,agresti2023reaction,agresti2023b_reaction}.

The work in this paper can be seen in the context of `exit problems', which study
how and when
solutions leave an appropriate neighbourhood of a stable state. Related results have been obtained for finite dimensional systems \cite{day1990large,freidlin1998random}
or problems posed on compact spatial domains \cite{faris1982large,salins2019,adams2022isochronal,adams2024quasi}, where the set of tools is much
richer. For example, the compactness of the associated semigroups can be used to prove the existence of quasi-invariant measures. However, we are not aware of any work that reproduces our explicit $\sigma$-expansions.

\paragraph{Outline}
After stating our main results in {\S}\ref{sec:mr}, we obtain the required
supremum bounds for stochastic and
deterministic convolutions in {\S}\ref{sec:conv}. We study
the smoothness of our nonlinearities between various function spaces in {\S}\ref{sec:sm}, which allows us to
show that our Taylor expansions are well-defined in {\S}\ref{sec:tay},
where we also quantify their growth rates. We analyze the limiting behaviour of expectations in {\S}\ref{sec:lim}. Turning to the residual $Z$, we establish a mild representation in {\S}\ref{sec:res}
and provide bounds for the underlying nonlinearities. This allows
us in {\S}\ref{sec:stb} to obtain
bounds for the probability of $\mathcal{A}_{\rm stb}$ and complete our proofs.

\section{Main Results}
\label{sec:mr}
In this section we formulate our main results, which concern the
stochastic reaction-diffusion system
\begin{equation}
\label{eq:mr:MainEquation}
    \mathrm du=[\Delta u + f(u)]\mathrm dt+\sigma g(u)\mathrm d
W_t^Q.
\end{equation}
Here $u(x,x_\perp,t)\in \mathbb R^n$ evolves in time $t\geq 0$ on a cylindrical domain $\mathcal D=\mathbb R\times \mathbb T^{d-1}\ni (x,x_\perp) $ with dimension $d\geq 1$, and is driven by a translationally  invariant noise process $(W^Q_t)_{t\geq 0}$.
For $d = 1$, we simply have $\mathcal{D} = \mathbb R$.

\begin{remark}
  Our techniques also apply after the replacement $\Delta \mapsto D \Delta$
  where the diffusion matrix $D$ is diagonal with strictly positive diagonal elements. For simplicity in stating and proving our results we work with $D = I_n$, but the approach used in \cite{hamster2020diag,hamster2020transinv}
  can be used to consider the general case.
\end{remark}

We will assume that in the deterministic setting $\sigma = 0$,
our system \eqref{eq:mr:MainEquation} admits a planar travelling wave solution
$u(x,x_\perp,t) = \Phi_0(x - c_0 t)$; see (HTw) below. Our main goal in this paper
is to consider the setting $0 < \sigma \ll 1$ and study solutions to
\eqref{eq:mr:MainEquation}
of the form
\begin{equation}
\label{eq:mr:decomp:u}
    u( x + \Gamma(t), x_\perp , t) = \Phi_0( x  ) + V(x,x_\perp,t),
\end{equation}
with an initial condition
\begin{equation}
\label{eq:mr:def:init:cond}
    V(x,x_\perp, 0) = \delta V_*(x,x_\perp)
\end{equation}
for some $0 \le \delta \ll 1$
and a phase function $\Gamma$ that is chosen to ensure that the solution is `frozen' in such a way that it is possible to extract long-term behaviour concerning $V$.
In contrast to the earlier works \cite{hamster2020expstability,bosch2024multidimensional}, we refrain from applying $\sigma$-dependent corrections
to the waveprofile $\Phi_0$ and rather absorb them into the perturbation $V$.
The main novelty in the present paper
is that we make the further decomposition
\begin{equation}
\label{eq:mr:decomp:v}
    V = Y_1 + \ldots + Y_{r-1} + Z,
\end{equation}
for some integer $r \ge 3$,
where conceptually $Y_j$ captures contributions that are of order $j$ in the pair $(\sigma, \delta)$, while $Z$ represents the higher-order residual.  The restriction $r \ge 3$ is made for technical convenience, related to the fact that certain fundamental stochastic contributions start appearing at second order in $\sigma$.

Fixing an integer $k_* > d / 2$ and introducing the notation
\begin{equation}
    k_j  = k_* + r + 1 - j,
\end{equation}
our goal is to capture the spatial behaviour of the functions above in the spaces\footnote{From this point onwards, we will implicitly use $\mathcal{D}$ and $\mathbb R^n$ as the domain and co-domain of all our function spaces, unless explicitly stated otherwise. In particular, we also write $L^2=L^2(\mathcal D;\mathbb R^n).$}
\begin{equation}
\label{eq:mr:inc:Y:W:j}
    Z(t) \in H^{k_*} = H^{k_*}(\mathcal{D};\mathbb R^n), \qquad \qquad Y_j(t) \in H^{k_j} = H^{k_j}(\mathcal{D};\mathbb R^n).
\end{equation}
We remark that these choices guarantee pointwise control on the functions in \eqref{eq:mr:inc:Y:W:j}, which is highly convenient when estimating nonlinear terms. We refer to \cite{bosch2024multidimensional} for a detailed discussion of settings that require less regularity on solutions, but do not pursue this in the present paper.

In {\S}\ref{subsec:mr:ass} we formulate our assumptions for \eqref{eq:mr:MainEquation} and the deterministic wave $(\Phi_0,c_0)$. The expansion functions $Y_j$ are defined rigorously in {\S}\ref{subsec:mr:tay}, where we also state their main properties. Finally, in {\S}\ref{subsec:mr:res} we formulate our main results concerning the residual $Z$.

\subsection{Assumptions}
\label{subsec:mr:ass}

We first consider the nonlinearities $f$ and $g$ appearing in \eqref{eq:mr:MainEquation}
and impose several global Lipschitz
conditions. We do point out that the pointwise control on $u$ that we establish in our main results means that we can safely modify $f$ and $g$ outside the region of interest to enforce these conditions.
We also note that the constants $u_\pm$ will correspond to the limiting values of the wave profile $\Phi_0$.
The integer $m \ge 1$ corresponds to the number of components of the noise.

\begin{itemize}
    \item[(HNL)] We have $f\in C^{k_* + r +1}(\mathbb R^n;\mathbb R^n)$
    and $g \in C^{k_* + r + 2}(\mathbb R^n;\mathbb R^{n \times m})$
    with
    \begin{equation}
        f(u_-)=f(u_+)=0, \qquad \qquad g(u_-) = g(u_+) = 0
    \end{equation}
    for some pair $u_\pm\in\mathbb R^n.$ In addition, there are  constants $K_f>0$ and $K_g > 0$ such that
    \begin{equation}
        |f(u_A)-f(u_B)|+\ldots+|D^{k_*+r+1}f(u_A)-D^{k_*+r+1}f(u_B)|\leq K_f|u_A-u_B|\label{eq:mr:lipschitz:f}
    \end{equation}
    holds for all  $u_A,u_B\in\mathbb R^{n}$,
    together with
    \begin{equation}
        |g(u_A)-g(u_B)|+\ldots+|D^{k_*+r+2}g(u_A)-D^{k_*+r+2}g(u_B)|\leq K_g|u_A-u_B|.\label{eq:mr:lipschitz:g}
    \end{equation}
\end{itemize}

We now turn to the planar wave solutions $u = \Phi_0(x + c_0 t)$ for \eqref{eq:mr:MainEquation} with $\sigma = 0$, which
must satisfy the travelling wave ODE
\begin{equation}
    \label{eq:mr:trv:wave:ode}
    \Phi_0''+c_0\Phi_0'+f(\Phi_0)=0.
    \end{equation}
Linearising  \eqref{eq:mr:trv:wave:ode}  around the travelling wave $(\Phi_0,c_0)$ leads to the linear operator
\begin{equation}
\label{eq:mr:def:l:tw}
    \mathcal L_{\rm tw}:H^2(\mathbb R;\mathbb R^n)\to L^2(\mathbb R;\mathbb R^n)
\end{equation}
that acts as
\begin{equation}
    [\mathcal L_{\rm tw}u](x)=u''(x)+c_0u'(x)+Df\big(\Phi_0(x)\big)u(x)
\end{equation}
and admits the translational neutral eigenvalue $\mathcal L_{\rm tw}\Phi_0'=0$. We will write
\begin{equation}
    \mathcal L_{\rm tw}^{\rm adj}:H^2(\mathbb R;
    \mathbb R^n)\to L^2(\mathbb R;\mathbb R^n)
\end{equation}
for the associated adjoint operator, which acts as
\begin{equation}
    [\mathcal L_{\rm tw}^{\rm adj}w](x)=w''(x)-c_0w'(x)+Df(\Phi_0(x))^\top w(x).
\end{equation}
Indeed, it is easily verified that $\langle \mathcal L_{\rm tw}v,w\rangle_{L^2(\mathbb R;\mathbb R^n)}=\langle v,\mathcal L_{\rm tw}^{\rm adj}w\rangle_{L^2(\mathbb R;\mathbb R^n)}$ holds for $v,w\in H^2(\mathbb R;\mathbb R^n).$

We impose standard existence, hyperbolicity and spectral stability conditions
on the wave $(\Phi_0,c_0)$. In particular, we assume the presence of a spectral gap.

\begin{itemize}
    \item[(HTw)] There exists a waveprofile $\Phi_0 \in C^{2}(\mathbb R;\mathbb R^n)$  and a wavespeed $c_0\in\mathbb R$
    that satisfy the travelling wave ODE \eqref{eq:mr:trv:wave:ode}.
    In addition,
    there is a constant $K>0$ together with exponents $\nu_\pm>0$ so that the bound
    \begin{equation}
        |\Phi_0(x)-u_-|+|\Phi_0'(\xi)|\leq Ke^{-\nu_-|x|}
    \end{equation}
    holds for all $x\leq 0,$ whereas the bound
    \begin{equation}
        |\Phi_0(x)-u_+|+|\Phi_0'(\xi)|\leq Ke^{-\nu_+|x|}
    \end{equation}
    holds for all $x\geq 0.$  Finally,
the operator $\mathcal L_{\rm tw}$ defined in \eqref{eq:mr:def:l:tw} has a simple eigenvalue at $\lambda=0$ and there exists a constant $\beta_{\rm tw}>0$ so that  $\mathcal L_{\rm tw}-\lambda$ is invertible for all $\lambda \in \mathbb C$ satisfying $\Re \lambda\geq -2\beta_{\rm tw}$.
\end{itemize}

Together with (HNL), these conditions imply that
$\Phi_0\in C^{k_* + r +3}(\mathbb R;\mathbb R^n)$ and that
 $|\Phi_0^{(\ell)}(x)|\to 0$ exponentially fast as $|x|\to\infty$, for any $1\leq \ell\leq k_* + r+3$.
 In particular, this means $\Phi_0' \in H^{k_* + r + 2}$. In addition, these conditions imply the existence of an adjoint eigenfunction $\psi_{\rm tw}$
 that satisfies $\langle \psi_{\rm tw} , \Phi_0'\rangle_{L^2(\mathbb R;\mathbb R^n)} = 1$
 and $\mathcal{L}_{\rm tw}^{\rm adj} \psi_{\rm tw} = 0$.
 In particular, inner products against $\psi_{\rm tw}$ can be used to `project out' translations of the original wave $\Phi_0$. This freedom can be used to impose the following technical restriction on the initial condition $V_*$.

\begin{itemize}
    \item[(H$V_*$)]
    We have the normalization $\norm{V_*}_{H^{k_* + r}} = 1$, together with the orthogonality condition\footnote{Here and elsewhere throughout this paper $\psi_{\rm tw}$
    can be interpreted as a function $\mathcal{D} \to  \mathbb R^n$ that depends in a constant fashion on the $y$ coordinate.
    The same holds for $\Phi_0$.
    }
    \begin{equation}
        \langle V_*, \psi_{\rm tw} \rangle_{L^2} =\langle V_*, \psi_{\rm tw} \rangle_{L^2(\mathcal D;\mathbb R^n)} =  0.
    \end{equation}
\end{itemize}

We continue by considering the covariance function $q$
that governs the noise process by means of the convolution
\begin{equation}
    [Qv](x,x_\perp)=[q*v](x,x_\perp)=
    \int_{\mathbb R}\int_{\mathbb T^{d-1}}q(x-x',x_\perp-x_\perp')v(x',x_\perp')\mathrm dx'\mathrm dx_\perp'.
\end{equation}
Indeed, the following conditions
imply that $Q$
is a bounded, symmetric linear operator on $L^2(\mathcal{D}; \mathbb R^m)$
that satisfies $\langle Q v, v \rangle_{L^2(\mathcal{D};\mathbb R^m)} \ge 0$. This allows us to follow
\cite{da2014stochastic,Gawarecki,hairer2009introduction,hamster2020transinv,karczewska2005stochastic,prevot2007concise} and  construct a cylindrical $Q$-Wiener process $W^Q=(W^Q_t)_{t\geq 0}$ that is defined on a filtered probability space $(\Omega,\mathcal F,\mathbb F,\mathbb P)$
and takes values in (an extended space containing) the Hilbert space $L^2(\mathcal D;\mathbb R^m)$;
see \cite[{\S}3.2]{bosch2024multidimensional}.

\begin{itemize}
    \item[(Hq)] We have  $q\in H^\ell(\mathcal D;\mathbb R^{m\times m})\cap L^1(\mathcal D;\mathbb R^{m\times m})$ for some integer $\ell >2k_1+d/2,$ with $q(-x,-x_\perp)=q(x,x_\perp)$ and $q^\top(x,x_\perp)=q(x,x_\perp)$ for all $(x,x_\perp)\in \mathcal D$. Further, for
    any $(\omega,\xi) \in  \mathbb R\times \mathbb Z^{d-1}$ the
    $m\times m$ matrix
    $\hat q(\omega, \xi)$ is non-negative definite,
    where $\hat q$ denotes the Fourier transform of $q$.
\end{itemize}

Upon introducing the Hilbert space
\begin{equation}
L_Q^2=Q^{1/2}\big(L^2(\mathcal D;\mathbb R^m)\big)
\end{equation}
and picking $0 \le k \le k_1 = k_* + r$, we note that \cite[Lem. 4.5]{bosch2024multidimensional}
allows
any $z \in H^k(\mathcal{D}; \mathbb R^{n \times m})$ to be interpreted as a Hilbert-Schmidt operator from $L^2_Q$ into $H^k$ via pointwise multiplication, with the bound
\begin{equation}
\label{eq:mr:hs:bnd:z}
    \norm{z}_{HS(L^2_Q;H^k)} \le K \norm{z}_{H^k(\mathcal{D}; \mathbb R^{n \times m})}.
\end{equation}
This allows
the noise term in
\eqref{eq:mr:MainEquation}
to be interpreted in a rigorous fashion, writing
\begin{equation}
    (g(u)[\xi])(x,x_\perp)=g\big(u(x,x_\perp)\big)\xi(x,x_\perp)\label{eq:pointwise}
\end{equation}
for any $\xi \in  L^2_Q$.

\subsection{Taylor expansion}
\label{subsec:mr:tay}

We follow the approach in \cite{bosch2024multidimensional}
to define the phase $\Gamma$ in the decomposition \eqref{eq:mr:decomp:u}.
In particular, we couple the evolution
\begin{equation}
\label{eq:mr:def:Gamma}
    \mathrm d \Gamma = \big( c_0 + a_{\sigma}(U, \Gamma)  \big) \mathrm dt
     + \sigma b(U,\Gamma) \mathrm d W^Q_t
\end{equation}
to our main system \eqref{eq:mr:MainEquation}, where the definitions
of $a_\sigma$ and $b$ are provided in Appendix \ref{app:def}.
Recalling \eqref{eq:mr:decomp:u}-\eqref{eq:mr:def:init:cond}
and applying \cite[Prop 6.1]{bosch2024multidimensional},
we see that
for every $T > 0$ there exists a stopping time $\tau_\infty \in (0, T]$
so that the $H^{k_*}(\mathcal{D};\mathbb R^n)$-valued identity
\begin{equation}
\label{eq:mr:weak:form:evol:v}
V(t) = \delta V_*  + \int_0^t \mathcal{R}_\sigma\big(V(s)\big) \, \mathrm ds
+ \sigma \int_0^t \mathcal{S}\big(V(s)\big) \, \mathrm d W^Q_s
\end{equation}
holds $\mathbb{P}$-a.s.\ for all $0 \le t < \tau_\infty$.
To isolate the $\sigma$ dependencies, we make the decomposition
\begin{equation}
\label{eq:mr:decomp:r:sigma}
    \mathcal{R}_\sigma(v; c_0,\Phi_0)
    = \Delta_{x_\perp} v + \mathcal{L}_{\mathrm{tw}} v
    + \mathcal{R}_{I}(v)
+ \sigma^2 \mathcal{R}_{II}(v)
+ \sigma^2 \Upsilon(v)
\end{equation}
and refer to Appendix \ref{app:def} for the full definitions of
$\mathcal{R}_\sigma$,
$\mathcal{R}_I$, $\mathcal{R}_{II}$, $\Upsilon$ and $\mathcal{S}$.
For our purposes here, we note that the inclusions
\begin{equation}
\label{eq:mr:inc}
 \mathcal{R}_I, \mathcal{R}_{II} \in C^j( H^{k+1}, H^k),
 \qquad
 \Upsilon \in C^j( H^{k+2}, H^k),
 \qquad
 \mathcal{S} \in C^j\big( H^{k+1}; HS(L^2_Q; H^k) \big)
\end{equation}
hold for $k_* \le k \le k_j$, together with
\begin{equation}
    \mathcal{R}_I(0) = D \mathcal{R}_I(0) = 0;
\end{equation}
see {\S}\ref{sec:sm} for the full details.

Our main task here is to
utilize the representation
\eqref{eq:mr:weak:form:evol:v}
to provide a Taylor expansion for $V$ with respect to the two (small) parameters $\sigma$ and $\delta$.
For example, the leading order expansion term $Y_1 = Y_1[\sigma, \delta]$ should be a solution to
\begin{equation}
Y_1(t) = \delta V_* + \int_0^t [\mathcal{L}_{\rm tw} + \Delta_{x_\perp}] Y_1(s) \, \mathrm ds
+ \sigma \int_0^t \mathcal{S}(0) \, \mathrm d W^Q_s.
\end{equation}

This motivates the definition
\begin{equation}
\label{eq:mr:def:w:only:one}
    Y_1[\sigma, \delta](t) = \delta \mathrm{exp}[ (\mathcal{L}_{\rm tw} + \Delta_{x_\perp}) t ] V_*  +
    \sigma \int_0^t \mathrm{exp}[ (\mathcal{L}_{\rm tw} + \Delta_{x_\perp}) (t-s) ] \mathcal{S}(0) \, \mathrm d W^Q_s.
\end{equation}
More generally, for any $1 \le j \le r-1$ we introduce
the expressions
\begin{equation}
\label{eq:mr:def:n:j}
\begin{array}{lcl}
    \widetilde{N}_{j;I}[Y_1, \ldots, Y_{j-1}] & = & \sum_{\ell=2}^j \sum_{i_1 + \ldots + i_{\ell} = j} \frac{1}{\ell !}D^{\ell} \mathcal{R}_I(0) [ Y_{i_1}, \ldots , Y_{i_\ell}] ,
\\[0.2cm]
  \widetilde{N}_{j;II}[Y_1, \ldots, Y_{j-2}] & = & \sigma^2  \sum_{\ell = 0}^{j-2} \sum_{i_1 + \ldots + i_{\ell} = j-2} \frac{1}{\ell !} D^{\ell} \mathcal{R}_{II}(0) [ Y_{i_1}, \ldots , Y_{i_\ell}] ,
\\[0.2cm]
\widetilde{N}_{j;III}[Y_1, \ldots, Y_{j-2}] & = & \sigma^2  \sum_{\ell = 0}^{j-2} \sum_{i_1 + \ldots + i_{\ell} = j-2} \frac{1}{\ell !} D^{\ell} \Upsilon(0) [ Y_{i_1}, \ldots , Y_{i_\ell}] ,
\end{array}
\end{equation}
together with
\begin{equation}
\label{eq:mr:def:b:j}
\begin{array}{lcl}
    \widetilde{B}_{j}[Y_1, \ldots, Y_{j-1}] & = & \sigma  \sum_{\ell = 0}^{j-1} \sum_{i_1 + \ldots + i_{\ell} = j-1} \frac{1}{\ell !}D^{\ell} \mathcal{S}(0) [ Y_{i_1}, \ldots , Y_{i_\ell}],
\end{array}
\end{equation}
where in each term we have $i_{\ell'} \ge 1$ for $\ell' \in \{ 1,\ldots,\ell\}$.
In addition, we write
\begin{equation}
    \widetilde{N}_j[Y_1, \ldots, Y_{j-1}] = \widetilde{N}_{j;I}[Y_1, \ldots, Y_{j-1}]
    + \widetilde{N}_{j;II}[Y_1, \ldots, Y_{j-2}]
    + \widetilde{N}_{j;III}[Y_1, \ldots, Y_{j-2}].
\end{equation}
Observe that the inclusions \eqref{eq:mr:inc} guarantee that
$\widetilde{N}_j$ and $\widetilde{B}_j$ map into $H^{k_j}$ provided that $Y_{j'} \in H^{k_{j'}}$ for $1 \le j' \le j-1$.
This allows us to recursively define
\begin{equation}
\label{eq:mr:def:w:j:all}
\begin{array}{lcl}
    Y_j[\sigma, \delta](t)
    & = & \delta \mathrm{exp}[ (\mathcal{L}_{\rm tw} + \Delta_{x_\perp}) t ] V_*\mathbf{1}_{j=1}
    \\[0.2cm]
    & & \qquad
    + \int_0^t  \mathrm{exp}[ (\mathcal{L}_{\rm tw} + \Delta_{x_\perp}) (t-s) ] \widetilde{N}_j\big[Y_1[\sigma, \delta](s), \ldots, Y_{j-1}[\sigma, \delta](s) \big] \, \mathrm ds
    \\[0.2cm]
        & & \qquad
    + \int_0^t  \mathrm{exp}[ (\mathcal{L}_{\rm tw} + \Delta_{x_\perp}) (t-s) ] \widetilde{B}_j\big[Y_1[\sigma, \delta](s), \ldots, Y_{j-1}[\sigma, \delta](s) \big] \, \mathrm dW^Q_s
\end{array}
\end{equation}
for all $1 \le j \le r-1$, extending \eqref{eq:mr:def:w:only:one}.

Our first main result shows that the expansion functions $Y_j$ are well-defined for all $t \ge 0$ and establishes crucial growth rates.
In order to match
the expressions
\eqref{eq:mr:def:w:j:all}
with the original
formulation \eqref{eq:mr:weak:form:evol:v},
it is convenient to introduce the shorthands
\begin{equation}
\label{eq:mr:def:n:b:j:no:tilde}
\begin{array}{lcl}
    N_j[\sigma, \delta](t) & = &
    \widetilde{N}_j\big[Y_1[\sigma, \delta](t), \ldots, Y_{j-1}[\sigma, \delta](t) \big],
\\[0.2cm]
B_j[\sigma, \delta](t) & = &
    \widetilde{B}_j\big[Y_1[\sigma, \delta](t), \ldots, Y_{j-1}[\sigma, \delta](t) \big].
\end{array}
\end{equation}
In addition, we introduce the notation
\begin{equation}
H_j = H^{k_j}, \qquad
    \norm{ \,\cdot\,}_{j} = \norm{\,\cdot\,}_{H^{k_j}}.
\end{equation}

\begin{prop}[{see {\S}\ref{sec:tay}}]
\label{prop:mr:tay}
    Suppose that \textnormal{(HNL)}, \textnormal{(HTw)}, \textnormal{(H$V_*$)} and \textnormal{(Hq)} all hold and pick a sufficiently large constant $K > 0$. Then for all $\sigma \ge 0$, all $\delta \ge 0$
    and all $1 \le j \le r-1$, the map
    \begin{equation}
        Y_{j}[\sigma,\delta] :[0,\infty) \times \Omega \to H_j 
    \end{equation}
    defined in \eqref{eq:mr:def:w:j:all}
    is progressively measurable and satisfies the following properties:
    \begin{itemize}
          \item[(i)] For every $p \ge 1$ and $T \ge 0$, we have the inclusion
          $Y_{j}[\sigma,\delta]\in L^p\big( \Omega,\mathbb{P}; C([0,T]; H_j) \big)$;

        \item[(ii)] The $H_j$-valued identity
        \begin{equation}
        \begin{aligned}
            Y_{j}[\sigma,\delta](t) &= Y_j[\sigma,\delta](0)+\int_0^t [\mathcal{L}_{\rm tw} + \Delta_{x_\perp}] Y_j[\sigma,\delta](s) \, \mathrm ds
            \\
             &  \qquad \qquad
             + \int_0^t N_j[\sigma,\delta](s) \, \mathrm ds
             %\\[0.2cm]
             %& & \qquad
              + \int_0^t B_j[\sigma,\delta](s) \, \mathrm dW_s^Q
        \end{aligned}
        \end{equation}
            holds $\mathbb P$-a.s.\ for all $t \ge 0$;
            
        \item[(iii)]{
          For any $p \ge 1$ and any $T \ge 2$ we have the moment bound
      \begin{equation}
      \label{eq:mr:mo:bnds:w:j}
          \mathbb E \sup_{0 \le t \le T} \norm{Y_{j}[\sigma,\delta](t)}_{j}^{2p}
          \le K^{2p} \Big[ (\sigma^2 p)^{j p} + [\sigma^2 \ln T + \delta^2 ]^{j p} \Big].
      \end{equation}
   }
\end{itemize}

\end{prop}

We now turn to the limiting behaviour
of the expansion functions $Y_j$. To this end, we introduce the constant
\begin{equation}
\label{eq:mr:def:beta}
    \beta = \min\{ \beta_{\rm tw} , \frac{1}{2} \lambda_1 \}.
\end{equation}
Here $\lambda_1 = 4 \pi^2 / |\mathbb T|^2$ denotes the first non-zero eigenvalue of the Laplacian $\Delta_{x_\perp}$ on the torus $\mathbb{T}^{d-1}$, where we take $\lambda_1 = \infty$ when $d = 1$. This constant captures the decay rates of the semiflows used throughout this paper after projecting out the neutral mode associated to the translational eigenvalue.

Our first result shows that multi-linear expressions involving the expansion functions converge in expectation to a well-defined limit at an exponential rate. We remark that our proof provides an algorithmic procedure to obtain an explicit integral expression for the limit $h_\infty$.
In addition, we can take $\Lambda = I$ to show for all $1 \le j \le r-1$
there exists
a function $Y_{j;\infty} \in H_j$
for which we have the convergence
   \begin{equation}
      \norm{ \mathbb E \big[Y_j[\sigma,\delta] (t) \big] - \sigma^{j} Y_{j;\infty }}_{j}
      \le K\big[ \sigma^{j} + \delta^{j}] e^{- \frac{\beta}{2} t}.
    \end{equation}

\begin{prop}[{see {\S}\ref{sec:lim}}]
\label{prop:mr:multiliner:lim}
Suppose that \textnormal{(HNL)}, \textnormal{(HTw)}, \textnormal{(H$V_*$)} and \textnormal{(Hq)} all hold.
Pick a Hilbert space $\mathcal{H}$ together with an integer $\ell \ge 1$ and a tuple $(i_1, \ldots, i_{\ell}) \in \{1, \ldots , r-1\}^{\ell}$ and write $i_{\rm tot} = i_1 + \ldots + i_{\ell}$. Then for any bounded multi-linear map $\Lambda: H_{i_1} \times \cdots \times H_{i_{\ell}} \to \mathcal{H} $ there exists a constant $K_\Lambda > 0$ and a limit $h_\infty \in \mathcal{H}$
   so that the bound
   \begin{equation}
      \norm{ \mathbb E \Lambda\big[Y_{i_1}[\sigma,\delta] (t), \ldots, Y_{i_{\ell}}[\sigma,\delta](t)\big]  - \sigma^{i_{\rm tot}} h_\infty}_{\mathcal{H}}
      \le  \big(\sigma^{i_{\rm tot}} + \delta^{i_{\rm tot}} \big) K_\Lambda e^{- \frac{\beta}{2} t}
    \end{equation}
      holds for all $t \ge 0$, all $\sigma \ge 0$ and all $\delta \ge 0$.
\end{prop}

Naturally, it is of interest to replace the multi-linear maps by more general functionals. To this end, we introduce the notation
\begin{equation}
    Y_{\rm tay}[\sigma,\delta]
     = Y_1[\sigma,\delta] + \ldots + Y_{r-1}[\sigma,\delta]
\end{equation}
and consider the behaviour of
$\phi(Y_{\rm tay})$ for a class of smooth functionals $\phi$. Our following result provides a natural Taylor expansion for this expression,
with coefficients \eqref{eq:mr:phi:w:tay:coeffs} that can be explicitly computed using the
procedure developed in {\S}\ref{sec:lim}. In order
to control the remainder, we impose the following condition on $\phi$.
\begin{itemize}
    \item[(H$\phi$)]{
      There exist $K > 0$ and $N > 0$ so that the map $\phi \in C^{r}(H^{k_*}; \mathcal{H})$ satisfies
      the bound\footnote{Here $\mathscr{L}^{(r)}( H^{k_*}; \mathcal{H})$ denotes the space of
      bounded $r$-linear maps from $(H^{k_*})^r$ into $\mathcal{H}$.}
      \begin{equation}
      \label{eq:mr:h:phi:bnd:deriv:phi}
          \norm{D^{r}\phi(w) }_{\mathscr{L}^{(r)}( H^{k_*}; \mathcal{H})}  \le K [1 + ||w||_{H^{k_*}}^N]
       \end{equation}
       for all $w \in H^{k_*}$.
    }
\end{itemize}

We remark that such smooth functionals can be used as a `core' on which so-called Ornstein-Uhlenbeck transition semigroups can be defined, which govern the evolution of probability measures \cite{goldys2003transition}. We do not pursue this issue in the current paper, but believe that it is of high interest for future work.

\begin{prop}[{see {\S}\ref{sec:lim}}]
\label{prop:mr:tay:lim:phi}
    Suppose that \textnormal{(HNL)}, \textnormal{(HTw)}, \textnormal{(H$V_*$)} and \textnormal{(Hq)} all hold and pick a Hilbert space $\mathcal{H}$ together with a functional $\phi$ that satisfies \textnormal{(H$\phi$)}.
    Then there exist quantities
    \begin{equation}
    \label{eq:mr:phi:w:tay:coeffs}
        (h_{\infty;0}, \ldots , h_{\infty; r-1} ) \in \mathcal{H}^r
    \end{equation}
    and a constant $K > 0$
    so that the expectation of $\phi(Y_{\rm tay})$ can be approximated by
    the limiting polynomial
    \begin{equation}
        h_\infty(\sigma) = h_{\infty;0} + \sigma h_{\infty; 1} + \ldots + \sigma^{r-1} h_{\infty;r-1}
    \end{equation}
    with an error bounded by
    \begin{equation}
        \norm{\mathbb E \big[ \phi\big(Y_{\rm tay}[\sigma,\delta](t) \big) \big] - h_\infty(\sigma)
        }_{\mathcal{H}}
        \le
        K \big[ \delta + \sigma] e^{- \frac{1}{4}\beta t}
        + K \sigma^r
    \end{equation}
    for all $0 \le \sigma \le 1$ and $0 \le \delta \le 1$.
\end{prop}

\subsection{Residual estimates}
\label{subsec:mr:res}

We now turn our attention to the full perturbation \eqref{eq:mr:decomp:v}.
In order to control the size of the residual $Z$, we recall the exponent $\beta > 0$ introduced in \eqref{eq:mr:def:beta} and
introduce the notation
\begin{equation}
\label{eq:mr:def:n:res}
    \mathcal{N}_{\rm res}(t) = \norm{Z(t)}_{H^{k_*}}^2 + \int_0^t e^{-2 \beta(t-s)} \norm{Z(s)}_{H^{k_* + 1}}^2 \, \mathrm ds.
\end{equation}
In addition, upon introducing the shorthand
\begin{equation}
\label{eq:mr:def:alpha}
    \alpha = \sqrt{ \delta^2 + \sigma^2 \ln T },
\end{equation}
we introduce the expression
\begin{equation}
    \mathcal{N}_{\rm full}(t;\sigma,\delta,T)
     = \alpha^{2(r-1)} \norm{Y_1(t)}_{1}^2 + \ldots + \alpha^2 \norm{Y_{r-1}(t)}_{r-1}^2     + \mathcal{N}_{\rm res}(t)
\end{equation}
together with the associated stopping time
\begin{equation}
    t_{\rm st}(\eta ; \sigma, \delta, T)
    = \inf \{ 0 \le t < \tau_\infty(T): \mathcal{N}_{\rm full}(t;\sigma,\delta,T) > \eta \alpha^{2(r-1)} \},
\end{equation}
writing $t_{\rm st} = \tau_\infty(T)$ whenever the infimum is taken over an empty set.
In particular, we point out that for $0 \le t < t_{\rm st}(\eta;\sigma, \delta, T)$ we have
\begin{equation}
    \norm{Y_j(t)}_{j}^2 \le \eta \alpha^{2(j-1)} =
     \eta \big( \sigma^2 \ln T + \delta^2 \big)^{j-1}
\end{equation}
for $1 \le j \le r-1$, together with
\begin{equation}
   \norm{Z(t)}_{H^{k_*}}^2 \le \eta \alpha^{2(r-1)} =  \eta   \big( \sigma^2 \ln T + \delta^2 \big)^{r-1}.
\end{equation}
In addition, we have integrated control over the higher-order norm $\norm{Z(t)}_{H^{k_* + 1}}$.

Intuitively, the event $t_{\rm st} < T$ represents the scenario that one of the expansion functions $Y_j$ or the residual $Z$ becomes `an order too large' in terms of the natural expansion parameter $\alpha$. The presence of the logarithmic term in this expansion parameter is directly related to the growth rate of the supremum of stochastic convolutions; see, e.g., \eqref{eq:mr:mo:bnds:w:j}.
Our main result in this paper shows that this scenario can be prevented with high probability over timescales that are exponentially long with respect to $\sigma$. We note that (formal)\footnote{We reiterate that $r \ge 3$ throughout this paper.} substitution of $r = 1$ recovers the bound obtained in \cite[Thm. 2.6]{bosch2024multidimensional}.

\begin{thm}[{see {\S}\ref{sec:stb}}]
\label{thm:main}
Suppose that \textnormal{(HNL)}, \textnormal{(HTw)}, \textnormal{(H$V_*$)} and \textnormal{(Hq)} hold.
Then   there exist  constants $0<\mu<1$, $\delta_\eta>0$, and $\delta_\sigma>0$ such that, for any integer $T\geq 3,$ any $0<\eta\leq \delta_\eta$, any $0 < \sigma\leq \delta_\sigma$ 
and any $\delta^2 < \mu \eta$, we have
    \begin{equation}
        \mathbb P(t_{\rm st}(\eta;\sigma,\delta, T)<T)\leq 2T\exp\left(-\mu \frac{\eta^{1/r}}{\sigma^{2/r}}\right)\label{eq:mr:bnd:stop:time}.
    \end{equation}
\end{thm}

In order to extract explicit bounds from our main result that can be interpreted in the spirit of our desired Taylor expansions, it is convenient to control the size-parameter $\eta$ and the timescale $T$ directly in terms of $\sigma$.
To this end, we pick a scale-parameter $0 < \theta < 1 $
and write
\begin{equation}
    \eta(\sigma;\theta) = 2^{1-r} \sigma^{2 (1-\theta)},
    \qquad \qquad
    T(\sigma;\theta) =
     \lfloor e^{\frac{1}{2} \mu \sigma^{-2 \theta/r}} \rfloor,
\end{equation}
where $\lfloor \cdot \rfloor$ rounds down to the nearest integer.
For convenience, we define the `stability event'
\begin{equation}
    \mathcal{A}_{\rm stb}(\sigma, \delta;\theta) =
    \{ t_{\rm st}\big(\eta(\sigma;\theta);\sigma,\delta, T(\sigma;\theta)\big)= T(\sigma;\theta) \}
\end{equation}
with the associated probability
\begin{equation}
    p_{\rm stb}(\sigma, \delta;\theta) =
    \mathbb P\big( \mathcal{A}_{\rm stb}(\sigma, \delta;\theta) \big)
\end{equation}
that we will aim to keep close to one.

Our next result represents
a convenient reformulation of the bound \eqref{eq:mr:bnd:stop:time} in a more explicit form. We point out
that the special choice
\begin{equation}
    \theta_* = 1/\big(2 (2 - 1/r)\big) = r/\big(2(2r - 1)\big)
\end{equation}
for $\theta$ allows
\eqref{eq:mr:bnd:theta:y:w:j} to be simplified to
\begin{equation}
\label{eq:mr:pathwise:cond:a:star}
    \norm{Z(t)}_{H^{k_*}}^2 \le
        \sigma^{2r - 1},
        \qquad \qquad
    \norm{Y_j(t)}_{j}^2 \le
     \sigma^{2j  - (r +j-1)/(2r - 1)}
    \le  \sigma^{2j  - 1} .
\end{equation}

\begin{cor}
    Suppose that \textnormal{(HNL)}, \textnormal{(HTw)}, \textnormal{(H$V_*$)} and \textnormal{(Hq)} hold
and pick
$\theta \in [0, \frac{1}{2})$. Recall the parameter $0 < \mu < 1$ defined
in Theorem \ref{thm:main}. Then there exists a constant $\delta_\sigma>0$ such that for any any $0<\sigma\leq \delta_\sigma$ 
and any $0 \le \delta \le \sigma^{1 - \theta/r}$, we have
    \begin{equation}
    \label{eq:mr:bnd:with:sigma}
        p_{\rm stb}(\sigma, \delta;\theta) \geq
        1 - 2\exp\left(- \frac{1}{2} \mu \sigma^{-2 \theta
 / r}\right).
    \end{equation}
In addition, whenever $\mathcal{A}_{\rm stb}$ holds we have
\begin{equation}
\label{eq:mr:bnd:theta:y:w:j}
    \norm{Z(t)}_{H^{k_*}}^2
    \le \sigma^{2r - 2 \theta(2 - 1/r)},
    \qquad
    \norm{Y_j(t)}_{j}^2 \le
  \sigma^{2j  - 2 \theta(1 + (j-1)/r)}
\end{equation}
for all $0 \le t \le T(\sigma;\theta)$ and $1 \le j \le r-1$.
\end{cor}
\begin{proof}
We note first that the restriction
on $\delta$ ensures that
\begin{equation}
    \delta^2 \le \sigma^2 \sigma^{-2 \theta /r} < \mu  \sigma^2 \sigma^{-2 \theta } = \mu \eta(\sigma;\theta)
\end{equation}
whenever $\sigma > 0$ is sufficiently small.
In particular, the estimate \eqref{eq:mr:bnd:with:sigma}
follows from \eqref{eq:mr:bnd:stop:time}.
In addition, we notice that
\begin{equation}
    \sigma^2 \ln T(\sigma;\theta) + \delta^2 \le \frac{1}{2}\mu \sigma^{2(1-\theta/r)} + \sigma^{2(1 - \theta/r)}
    \le 2 \sigma^{2(1 - \theta/r)}.
\end{equation}
The first bound in
\eqref{eq:mr:bnd:theta:y:w:j} can hence be obtained by computing
\begin{equation}
    \norm{Z(t)}_{H^{k_*}}^2 \le
        \eta (\sigma^2 \ln T + \delta^2)^{r-1}
    \le  2^{1-r}\sigma^{2 (1-\theta)}\big[ 2 \sigma^{2(1 - \theta/r)}    ]^{r-1}
    = \sigma^{2r - 2 \theta(2 - 1/r)},
\end{equation}
while the second bound follows from the estimate
\begin{equation}
    \norm{Y_j(t)}_{j}^2 \le
    \eta (\sigma^2 \ln T + \delta^2)^{j-1}
    \le 2^{1-r} \sigma^{2 (1-\theta)}\big[ 2 \sigma^{2(1 - \theta/r)}    ]^{j-1}
    = 2^{j-r}  \sigma^{2j  - 2 \theta(1 + (j-1)/r)}.
\end{equation}

\end{proof}

\subsection{Wave properties}

We are now ready to consider expressions of the form $\phi(V)$ and provide Taylor
expansions
for their expectation, conditioned on the
stability properties encoded in $\mathcal{A}_{\rm stb}$.
In particular, we make the decomposition
    \begin{equation}
    \label{eq:mr:decomp:cnd:exp}
        \mathbb E \big[ \phi(V) | \mathcal{A}_{\rm stb} \big]
         = \mathbb E \big[ \phi(Y_{\rm tay}) | \mathcal{A}_{\rm stb} \big]
         + \mathbb E \big[ \big(\phi(Y_{\rm tay} +Z) - \phi(Y_{\rm tay}) \big) | \mathcal{A}_{\rm stb} \big]
    \end{equation}
and use the pathwise properties \eqref{eq:mr:pathwise:cond:a:star}
to bound the second term as $O( \sigma^{r-\frac{1}{2}})$.
In order to use the expansion in Proposition \ref{prop:mr:tay:lim:phi},
we hence need to control the change to the expectation of $\phi(Y_{\rm tay})$
upon conditioning on the (high-probability) event $\mathcal{A}_{\rm stb}$.
To this end, we use the representation
\begin{equation}
    \mathbb E [\phi(Y_{\rm tay}) | \mathcal{A}_{\rm stb}] - \mathbb E[\phi(Y_{\rm tay})]
    = p_{\rm stb}^{-1} \int_{\Omega} \phi\big(Y_{\rm tay}(\omega)\big) (\mathbf{1}_{\mathcal{A}_{\rm stb}}(\omega) - p_{\rm stb}) \, \mathrm d \mu
\end{equation}
to obtain the estimate
\begin{equation}
\label{eq:mr:delta:exp:cond:phi:w}
\begin{array}{lcl}
    \norm{ \mathbb E \big[\phi(Y_{\rm tay}) | \mathcal{A}_{\rm stb}\big] - \mathbb E[\phi(Y_{\rm tay})] }_{\mathcal{H}}
    &\le& p_{\rm stb}^{-1} \int_{\Omega} \norm{ \phi\big(Y_{\rm tay}(\omega)\big) (\mathbf{1}_{\mathcal{A}_{\rm stb}}(\omega) - p_{\rm stb}) }_{\mathcal{H}} \, \mathrm d \mu
\\[0.2cm]
    & \le & p_{\rm stb}^{-1} \big[ \int_{\Omega} \norm{\phi\big(Y_{\rm tay}(\omega)\big)}_{\mathcal{H}}^2 \, \mathrm d \mu \big]^{1/2} \big[ \int_{\Omega} (\mathbf{1}_{\mathcal{A}_{\rm stb}}(\omega) - p_{\rm stb})^2 \, \mathrm d \mu]^{1/2}
\\[0.2cm]
& = & p_{\rm stb}^{-1} \big[\mathbb E \norm{\phi(Y_{\rm tay})}_{\mathcal{H}}^2 \big]^{1/2}
\big[ p_{\rm stb} ( 1 - p_{\rm stb}) \big]^{1/2}
\end{array}
\end{equation}
and subsequently use the fact that $1 - p_{\rm stb}$ is small.

\begin{prop}
\label{prp:mr:taylor:residual:full}
Suppose that \textnormal{(HNL)}, \textnormal{(HTw)}, \textnormal{(H$V_*$)} and \textnormal{(Hq)} all hold. 
Pick a Hilbert space
$\mathcal{H}$ together with a functional $\phi$ that satisfies \textnormal{(H$\phi$)}.
Then there exist quantities
\begin{equation}
    \big( h_{\infty;0}, \ldots , h_{\infty;r-1} \big) \in \mathcal{H}^r
\end{equation}
together with constants $K > 0$ and $\delta_\sigma > 0$
so that
for all $0 < \sigma \le \delta_\sigma$ 
and
any $0 \le \delta \le \sigma^{1 - \theta_*/r}$,
the conditional expectation of $\phi(V)$ can be approximated
by the limiting polynomial
\begin{equation}
\label{eq:mr:res:lim:poly}
        h_\infty(\sigma) = h_{\infty;0} + \sigma h_{\infty; 1} + \ldots + \sigma^{r-1} h_{\infty;r-1}
    \end{equation}
with an error bounded by
    \begin{equation}
    \label{eq:mr:cond:exp:bnd:error}
       \norm{ \mathbb E \big[ \phi\big(V(t)\big) | \mathcal{A}_{\rm stb}(\sigma,\delta;\theta_*) \big]  - h_\infty(\sigma) }_{\mathcal{H}}
       \le K[ \delta + \sigma]  e^{- \frac{1}{2} \beta t}  +
       K  \sigma^{r - \frac{1}{2}}
        \end{equation}
for any $0 \le t \le T(\sigma; \theta_*)$.
\end{prop}
\begin{proof}
Pick $0 \le t \le T(\sigma;\theta_*)$.
    The smoothness of $\phi$ together with the a-priori bounds \eqref{eq:mr:pathwise:cond:a:star} on $Y_{\rm tay}$ and $Z$ that are available when  $\mathcal{A}_{\rm stb}$ holds imply that
    \begin{equation}
        \norm{ \mathbb E \big[ \big( \phi(Y_{\rm tay}(t) +Z(t)) - \phi(Y_{\rm tay}(t)) \big) | \mathcal{A}_{\rm stb} \big]}_{\mathcal{H}}
        \le C_1 \mathbb E  \big[ \norm{Z(t)}_{H^{k_*}} | \mathcal{A}_{\rm stb} \big]  \le C_1 \sigma^{r - \frac{1}{2}}
    \end{equation}
    for some $C_1 > 0$. In addition, the assumption \eqref{eq:mr:h:phi:bnd:deriv:phi}
    together with the moment bound \eqref{eq:mr:mo:bnds:w:j}
    implies that
    \begin{equation}
        \mathbb E \norm{\phi\big(Y_{\rm tay}(t)\big)}_{\mathcal{H}}^2
        \le C_2 \mathbb E \big[ 1 + \norm{Y_{\rm tay}(t)}_{H^{k_*}}^{N + r}\big]
        \le C_3
    \end{equation}
    for some $C_2 > 0$ and $C_3 > 0$, possibly after restricting the size of $\delta_\sigma$ to ensure that $\sigma^2 \ln T(\sigma;\theta_*) + \delta^2 \le 1$. After a further restriction of $\delta_\sigma$, the exponential bound
    \eqref{eq:mr:bnd:with:sigma} can now be used together with
    \eqref{eq:mr:delta:exp:cond:phi:w} to estimate
    \begin{equation}
     \norm{ \mathbb E \big[\phi(Y_{\rm tay}) | \mathcal{A}_{\rm stb}\big] - \mathbb E[\phi(Y_{\rm tay})] }_{\mathcal{H}}
    \le \sigma^{r - \frac{1}{2}}.
    \end{equation}
    In view of the decomposition \eqref{eq:mr:decomp:cnd:exp}, the desired
    error bound \eqref{eq:mr:cond:exp:bnd:error}
    follows from Proposition \ref{prop:mr:tay:lim:phi}.
\end{proof}

Note that the limiting polynomial $h_\infty$ in \eqref{eq:mr:res:lim:poly} is explicitly computable with the procedure developed in {\S}\ref{sec:lim} and depends only on the functional $\phi$ and the expansion functions $Y_j$. In particular, it does not depend on the precise details of picking the time $T(\sigma;\theta_*)$,
which only affects the remainder bounds.
This allows us to extract expansion coefficients for wave properties that are in some sense canonical.

For example, let us write
\begin{equation}
    C_{\rm obs}(\sigma, \delta) = \frac{2}{T(\sigma;\theta_*)}
    \big[ \Gamma(T(\sigma;\theta_*)) - \Gamma(\tfrac{1}{2} T(\sigma;\theta_*)) \big],
\end{equation}
which can be interpreted as the observed average speed over the
interval $[\frac{1}{2}T(\sigma;\theta_*), T(\sigma;\theta_*)]$,
the second half of the full interval where stability can be expected with high probability. Note that the first half is exluded to avoid any transients. The
translational invariance properties
stated in
\eqref{eq:app:list:comm:rels:a:b}
together with the evolution \eqref{eq:mr:def:Gamma}
allow us to write
\begin{equation}
C_{\rm obs}(\sigma,\delta) =  c_0 + \frac{2}{ T(\sigma;\theta_*)} \int_{\frac{1}{2} T(\sigma;\theta_*)}^{T(\sigma;\theta_*)} a_\sigma(\Phi_0 +V(t), 0) \, \mathrm dt
+ \sigma \mathcal{B}(\sigma, \delta),
\end{equation}
where we have introduced the notation
\begin{equation}
\mathcal{B}(\sigma,\delta) = \frac{2}{ T(\sigma;\theta_*)} \int_{\frac{1}{2} T(\sigma;\theta_*)}^{T(\sigma;\theta_*)} b\big(\Phi_0 + V(t), 0\big) \, \mathrm d W^Q_t  .
\end{equation}

Our results enable us to obtain a rigorous expansion for the conditional expectation of this speed $C_{\rm obs}(\sigma,\delta)$. The computations in \cite{hamster2020transinv}
provide an explicit expression for the coefficient $c_2$ together with numerical evidence to support these predictions.

\begin{cor}
Suppose that \textnormal{(HNL)}, \textnormal{(HTw)}, \textnormal{(H$V_*$)}
and \textnormal{(Hq)} all hold.
Then there exist scalars
$(c_2, \ldots, c_{r-1})$
together with constants $K > 0$ and $\delta_\sigma > 0$
so that
for all $0 < \sigma \le \delta_\sigma$
and
any $0 \le \delta \le \sigma^{1 - \theta_*/r}$,
we have the bound
    \begin{equation}
    | \mathbb E \big[ C_{\rm obs}(\sigma,\delta) | \mathcal{A}_{\rm stb}(\sigma,\delta;\theta_* )\big]
    - c_0 - c_2 \sigma^2 - \ldots - c_{r-1} \sigma^{r-1} |
       \le K[ \delta + \sigma]  e^{- \frac{1}{4} \beta T(\sigma;\theta_*)}  +
       K  \sigma^{r - \frac{1}{2}}
       \le 2 K \sigma^{r - \frac{1}{2}}
       .
        \end{equation}
\end{cor}
\begin{proof}
Note first that $\mathbb E \mathcal{B}(\sigma, \delta) = 0$. In addition, the uniform bound $\norm{b(\Phi_0 +V,0)}_{HS(L^2_Q;\mathbb{R})} \le K_b$ from Lemma \ref{lem:sm:b:nu:kct}
implies that
\begin{equation}
\mathbb E  | \mathcal{B}(\sigma,\delta)|^2 \le \frac{4}{ T(\sigma;\theta_*)^2}
\int_{\frac{1}{2} T(\sigma;\theta_*)}^{T(\sigma;\theta_*)} K_b^2 \, \mathrm dt
= \frac{2K_b^2 }{T(\sigma;\theta_*)} .
\end{equation}
In particular, following the computation \eqref{eq:mr:delta:exp:cond:phi:w}
we obtain the bound
\begin{equation}
    \big| \mathbb E[ \mathcal{B}(\sigma,\delta) | \mathcal{A}_{\rm stb}] \big|
    \le \frac{\sqrt{2} K_b}{ T(\sigma;\theta_*)^{1/2}}
    p_{\rm stb}(\sigma,\delta;\theta_*)^{-1} [p_{\rm stb}(\sigma,\delta;\theta_*)(1 - p_{\rm stb}(\sigma,\delta;\theta_*))]^{1/2},
\end{equation}
which is exponentially small allowing it to be readily absorbed in the desired $O( \sigma^{r-1/2})$ remainder.

In view of the explicit
representation \eqref{eq:list:expr:a:smooth}
for $a_\sigma$ and the smoothness
results in {\S}\ref{sec:sm},
we may apply Proposition \ref{prp:mr:taylor:residual:full}
to the function  $v \mapsto a_\sigma(\Phi_0 + v)$
to obtain the desired expansion and error bound. The fact that $c_1 = 0$
follows by observing that
$a_{\sigma}$ is of second order with respect to the pair $(\sigma,V)$.
\end{proof}

\section{Convolution bounds}
\label{sec:conv}
In this preparatory section, we
consider deterministic and stochastic convolutions against
random evolution families
generated by time-dependent linear operators $
    \mathcal L_{\nu}(t):\Omega\to\mathscr L(H^{2}, L^2)$
that act as
\begin{equation}
    [\mathcal L_\nu(t)(\omega)u](x,x_\perp)= [\mathcal L_{\rm tw}u(\cdot,x_\perp)](x)+\nu(t,\omega)[\Delta_{x_\perp} u(x,\cdot)](x_\perp),\label{eq:lin_gen}
\end{equation}
with $(x,x_\perp) \in \mathbb R \times \mathbb T^{d-1}$ and $\omega \in \Omega$. We impose the following conditions on the coefficient function $\nu$ and the general setting that we consider in this section.

\begin{itemize}
    \item[(hE)] We have $T \ge 1$ and $k_* \le k \le k_* + r$. The function $\nu:[0,T]\times\Omega\mapsto \mathbb R$ is progressively measurable and continuous with respect to the time variable $\mathbb P$-almost surely, with
    \begin{equation}
        \frac{1}{2} \le \nu(t) \le 2
    \end{equation}
    for all $t \in [0,T]$.
\end{itemize}

As discussed at length in \cite[{\S}3]{bosch2024multidimensional},
the conditions (HNL), (HTw) and (hE) ensure that
the linear operators $\mathcal L_{\nu}$
generate a random evolution family $E(t,s)$
on $H^{k}$,
in the sense that $v(t) = E(t,s) v_s$ is a solution to the initial value problem
\begin{equation}
\label{eq:conv:def:evol:probl}
    \partial_t v = \mathcal{L}_\nu(t) v, \qquad v(s) = v_s \in H^k.
\end{equation}
In order to project out the neutral solution $v(x,x_\perp,t) = \Phi_0'(x)$,
we introduce the operators
\begin{equation}
    P = \frac{1}{|\mathbb T|^{d-1}}\langle \, \cdot\,, \psi_{\rm tw} \rangle_{L^2} \Phi_0', \qquad
P^\perp = I - P.
\end{equation}
Recalling the constant $\beta$ defined in \eqref{eq:mr:def:beta},
we note that
\cite[Lem. 3.2]{bosch2024multidimensional} guarantees
the existence of a constant $M \ge 1$ for which the bounds
\begin{equation}
\label{eq:conv:bnds:e}
    \norm{ E(t, s) P^\perp }_{\mathscr{L}(H^k ;H^k)} \le M e^{-\beta(t-s)},
    \qquad
    \norm{ E(t, s) P^\perp }_{\mathscr{L}(H^k ;H^{k+1})} \le M \max \{ 1 , (t-s)^{-1/2} \} e^{-\beta(t-s)}
\end{equation}
hold for all $0 \le s \le t$.

In this section we are interested in convolutions of the form
\begin{equation}
\label{eq:conv:def:e:n:b}
\begin{array}{lcl}
\mathcal{E}_N^d(t) &=& \int_0^t E(t, s)P^\perp N(s) \, \mathrm ds ,
\\[0.2cm]
\mathcal{E}_B^s(t) &=& \int_0^t E(t, s)P^\perp B(s) \, \mathrm d W^{Q;-}_s,
\end{array}
\end{equation}
together with the integrated expressions
\begin{equation}
\label{eq:conv:def:i:n:b}
    \begin{array}{lcl}
\mathcal{I}_N^d(t) &=& \int_0^t e^{- \beta (t-s)} \norm{ \mathcal{E}_N^d(s)}_{H^{k+1}}^2\, \mathrm ds ,
\\[0.2cm]
\mathcal{I}_B^s(t) &=& \int_0^t e^{- \beta (t-s)} \norm{ \mathcal{E}_B^s(s)}_{H^{k+1}}^2 \, \mathrm ds.
\end{array}
\end{equation}
We note that the second integral in \eqref{eq:conv:def:e:n:b} is referred to as a forward-integral,
which is well-defined for all $0 \le t \le T$ when $B$ is a member
of the family
\begin{equation}
    \mathcal{N}^p(T )  = \{ B \in L^p\big(\Omega; L^2\big([0, T];HS(L^2_Q;H^k) \big) \big) :
    B \hbox{ has a progressively measurable version} \}
\end{equation}
for some $p \ge 2$; see \cite[{\S}3.2]{bosch2024multidimensional}.
When $\nu$ does not depend on $\omega$, this forward integral coincides with the regular It{\^o} integral.

We impose the following assumptions on the integrands $N$ and $B$. These are weaker than the corresponding assumptions considered in our prior work \cite{hamster2020expstability,  bosch2024multidimensional}, where we only considered the edge cases $n=0$ respectively $n = 1$ (albeit in a more general setting where only integrated control on $B$ is needed).

\begin{itemize}
    \item[(hN)]{
      There exists an integer $n \ge 0$ and constants $\Theta_1 > 0$ and $\Theta_2 \ge 0$ so that for every integer $p \ge 1$ we have
      $N \in L^{2p}\big(\Omega; C([0, T]; H^k) \big)$
      together with the bound
      \begin{equation}
      \label{eq:conv:bnd:in:hN}
        \mathbb E  \sup_{0 \le s \le T}  \norm{N(s)}_{H^k}^{2p}
        \le \big( p^{np} + \Theta_2^{np} \big) \Theta_1^{2p}.
      \end{equation}
      }
    \item[(hB)]{
      There exists an integer $n \ge 1$ and constants $\Theta_1 > 0$
      and $\Theta_2 \ge 0$ so that for every integer $p \ge 1$
      we have $B \in \mathcal{N}^{2p}(T)$ with the bound
      \begin{equation}
      \label{eq:conv:bnd:in:hB}
      \mathbb E  \sup_{0 \le s \le T}  \norm{B(s)}_{HS(L^2_Q;H^k)}^{2p}
    \le \big[ p^{(n-1)p} + \Theta_2^{(n-1)p} \big] \Theta_1^{2p}   .
\end{equation}
     }
\end{itemize}

Our two main results in this section provide supremum bounds for the expressions \eqref{eq:conv:def:e:n:b} and \eqref{eq:conv:def:i:n:b} defined above. The main feature
is the logarithmic growth term
in \eqref{eq:prp:itgs:eb:hb:full},
which directly contributes to the corresponding term in the growth bound
\eqref{eq:mr:mo:bnds:w:j}
for our expansion functions.

\begin{prop}[{see {\S}\ref{subsec:conv:det}}]
\label{prp:cnv:hn}
Assume that \textnormal{(HTw)}, \textnormal{(HNL)} and \textnormal{(hE)} hold.
Then there exists a constant $K_{\rm gr} > 0$ that does not depend on $T$ so that
for any process $N$ that satisfies \textnormal{(hN)} and any (real) $p \ge 1$
we have the bound
\begin{equation}
\label{eq:itg:en:hn:sup:bnd:full}
\mathbb E \sup_{0 \le t \le T} \norm{\mathcal{E}^d_{N}(t)}_{H^k}^{2p}
 + \mathbb E \sup_{0 \le t \le T} \mathcal{I}^d_N(t)^p
\le K_{\rm gr}^{2p} (p^{np} + \Theta_2^{np}) \Theta_1^{2p} .
\end{equation}
\end{prop}

\begin{prop}[{see {\S}\ref{subsec:conv:st}}]
\label{prp:cnv:hb}
Assume that \textnormal{(HTw)}, \textnormal{(HNL)}, \textnormal{(Hq)} and \textnormal{(hE)} hold
and assume $T \ge 2$ is an integer.
Then there exists a constant $K_{\rm gr}$ that does not depend on $T$ so that
for any process $B$ that satisfies \textnormal{(hB)} and any (real) $p \ge 1$
we have the bound
\begin{equation}
\label{eq:prp:itgs:eb:hb:full}
\mathbb E \sup_{0 \le t \le T} \norm{\mathcal{E}^s_{B}(t)}_{H^k}^{2p}
 + \mathbb E \sup_{0 \le t \le T} \mathcal{I}^s_B(t)^p
\le K_{\rm gr}^{2p} (32 en)^{2np} (p^{np} + [\ln T + \Theta_2]^{np}) \Theta_1^{2p} .
\end{equation}
\end{prop}

\subsection{Deterministic convolutions}
\label{subsec:conv:det}

Our goal here is to establish Proposition \ref{prp:cnv:hn}, which concerns the
deterministic convolution $\mathcal{E}^d_N$ and the associated integral $\mathcal{I}^d_N$.
We proceed in a pathwise fashion, using relatively direct estimates.

\begin{lem}
Assume that \textnormal{(HTw)}, \textnormal{(HNL)} and \textnormal{(hE)} hold
and consider any process $N$ that satisfies \textnormal{(hN)}.
Then for all $0 \le t  \le T$
we have the pathwise bound
\begin{equation}
\label{eq:conv:bnd:e:n:prlm}
    \norm{\mathcal{E}^d_N(t)}_{H^{k+1}}
    \le  2M \big(1 + \beta^{-1}\big)
     \sup_{0 \le s \le t} \norm{N(s)}_{H^k}.
\end{equation}
\end{lem}
\begin{proof}
This follows directly from
the bounds \eqref{eq:conv:bnds:e} and the computation
\begin{equation}
\label{eq:itg:en:prlm:bnd}
\begin{array}{lcl}
    \norm{\mathcal{E}^d_N(t)}_{H^{k+1}}
    & \le &   M \int_0^{t} e^{-\beta(t-s )} \big( 1 + (t-s)^{-1/2} \big) \norm{N(s)}_{L^2}\,\mathrm d s
\\[0.2cm]
    & \le  &  M \big[ \int_0^{t} e^{-\beta(t-s )} \big( 1 + (t-s)^{-1/2} \big)\,\mathrm ds\big]
    \sup_{0 \le s \le t} \norm{N(s)}_{H^k}
    \\[0.2cm]
& \le &
\big[ \big(\frac{ 2 M}{\beta }\big) + M \int_{t-1}^t (t-s)^{-1/2} \, \mathrm ds \big]
  \sup_{0 \le s \le t} \norm{N(s)}_{H^k}
\\[0.2cm]
& \le &
\big[ \big(\frac{ 2 M}{\beta }\big) + 2 M  \big]
  \sup_{0 \le s \le t} \norm{N(s)}_{H^k}.
\end{array}
\end{equation}
\end{proof}

\begin{proof}[Proof of Proposition \ref{prp:cnv:hn}]
First consider the case that $p \ge 1$ is an integer.
The estimate for  $\mathcal{E}^d_N$ then follows
directly from \eqref{eq:conv:bnd:in:hN}
and \eqref{eq:conv:bnd:e:n:prlm}.
The bound for $\mathcal{I}^d_N$ also follows
in a similar fashion
using the pathwise estimate
\begin{equation}
\mathcal{I}^d_N(t)  \le \frac{1}{\beta} \sup_{0 \le s \le t} \norm{\mathcal{E}^d_N(s)}_{H^{k+1}}^2
\le \frac{1}{\beta} (4 M^2) (1 + \beta^{-1})^2 \sup_{0 \le s \le t} \big[ \norm{N(s)}_{H^k}]^2.
\end{equation}
When $p \ge 1$ is not an integer, we
pick $q > 1$ in such a way that $p q$  is an integer
and use the estimate
$\sqrt[q]{a + b} \le \sqrt[q]{a} + \sqrt[q]{b}$ for $a \ge 0$ and $b \ge 0$
to compute
\begin{equation}
    \mathbb E \sup_{0 \le t \le T} \norm{\mathcal{E}^d_N(t)}^{2p}_{H^k}
    \le \left[  \mathbb E \sup_{0 \le t \le T} \norm{\mathcal{E}^d_N(t)}^{2p q}_{H^k} \right]^{1/q}
    \le \big[ K^{2pq}( p^{np q} + \Theta_2^{np q} ) \Theta_1^{2pq} \big]^{1/q}
    \le  K^{2p}( p^{np } + \Theta_2^{np } ) \Theta_1^{2p} ,
\end{equation}
with a similar computation for $\mathcal{I}^d_N$.
\end{proof}

\subsection{Stochastic convolutions}
\label{subsec:conv:st}

We here set out to establish Proposition \ref{prp:cnv:hb}, using a strategy that is significantly more streamlined than in \cite{bosch2024multidimensional,hamster2020expstability}.
As a preparation,
we note that there exist $T$-independent constants $K_{\rm cnv} > 0$,
$K_{\rm dc} > 0$ and $K_{\rm mr} > 0$ so that
for any integer $p \ge 1$
and any $B \in \mathcal{N}^{2p}(T)$
have the maximal inequality \cite[Thm. 3.7]{bosch2024multidimensional}
\begin{equation}
\label{eq:conv:max:ineq}
\mathbb E \sup_{0 \le t \le T} \norm{ \mathcal{E}^s_B(t) }_{H^k}^{2p}
\le p^p K_{\rm cnv}^{2p} \mathbb E \left[ \int_0^T \norm{B(s)}_{HS(L^2_Q;H^k)}^2 \, \mathrm  ds  \right]^{p} ,
\end{equation}
together with the weighted decay estimate \cite[Prop 3.10]{bosch2024multidimensional}
\begin{equation}
\label{eq:conv:wt:decay}
\mathbb E \norm{ \mathcal{E}^s_B(t) }_{H^k}^{2p}
\le p^p K_{\rm dc}^{2p} \mathbb E \left[ \int_0^t e^{-\beta(t-s)} \norm{B(s)}_{HS(L^2_Q;H^k)}^2  \, \mathrm  ds \right]^{p}
\end{equation}
and the maximal regularity bound \cite[Prop. 3.11]{bosch2024multidimensional}
\begin{equation}
\label{eq:conv:max:reg:bnd}
\mathbb E [\mathcal{I}^s_B(t)]^p
\le K_{\rm mr}^p \mathbb E \sup_{0 \le s \le t} \norm{\mathcal{E}^s_B(s)}_{H^k}^{2p}
+ p^p K_{\rm mr}^p \mathbb E  \left[ \int_0^t e^{-\beta(t-s)} \norm{B(s)}_{HS(L^2_Q;H^k)}^2  \, \mathrm  ds \right]^{p} ,
\end{equation}
in which the latter two hold for all $0 \le t \le T$.
Following \cite{bosch2024multidimensional}, it is convenient to impose the following
temporary assumption.
\begin{itemize}
    \item[(hB*)]{
      Assumption (hB) holds and, in addition, $B$ is a finite-rank process that takes values in $HS(L^2_Q;H^{k+2})$.
     }
\end{itemize}
Under this assumption, we note that the integration range of
forward integrals can be split in a customary fashion
(see \cite[Eq. (3.49)]{bosch2024multidimensional}).
In particular, assuming that $T$ is an integer and picking $i \in \{0, \ldots, T-1\}$, we make the decomposition
\begin{equation}
\begin{array}{lcl}
\mathcal{E}^s_B(t) & = &
\mathcal{E}_{B;I}^{(i)}(t) + \mathcal{E}_{B;II}^{(i)}(t)
\end{array}
\end{equation}
for $i \le t \le i+1$.
Here we have defined 
\begin{equation}
\mathcal{E}_{B;I}^{(i)}(t)=\int_0^{i} E(t,s) P^\perp B(s)\, \mathrm d W_s^{Q;-}\quad\text{and} \quad \mathcal{E}_{B;II}^{(i)}(t)=\int_{i}^t E(t,s)P^\perp B(s) \,\mathrm d W_s^{Q;-}.
\end{equation}

\begin{lem}
\label{lem:conv:bnd:e:b:i}
Assume that \textnormal{(HTw)}, \textnormal{(HNL)}, \textnormal{(Hq)} and \textnormal{(hE)} hold, that $T$ is an integer and $i \in \{0, \ldots , T-1\}$. Then for any process $B$ that satisfies \textnormal{(hB*)}
and any integer $p \ge 1$
we have the bound
\begin{equation}
\mathbb E \sup_{i \le t \le i+1} \norm{\mathcal{E}_{B;I}^{(i)}(t)}_{H^k}^{2p}
\le M^{2p} K_{\rm dc}^{2p} p^p \frac{1}{\beta^p}
\big[ p^{(n-1)p} + \Theta_2^{(n-1)p} \big] \Theta_1^{2p}.
\end{equation}
\end{lem}
\begin{proof}
In view of the identity
\begin{equation}
\mathcal{E}_{B;I}^{(i)}(t)
=
E(t, i) \int_0^{i} E(i,s) P^\perp B(s) \,\mathrm d W_s^{Q;-},
\end{equation}
we obtain the pathwise bound
\begin{equation}
\textstyle \norm{ \mathcal{E}_{B;I}^{(i)}(t)}_{H^k} \le M  \norm{ \int_0^i E(i,s) P^\perp B(s)\, \mathrm d W_s^{Q;-} }_{H^k}.
\end{equation}
Applying the weighted decay estimate
\eqref{eq:conv:wt:decay}
subsequently yields
\begin{equation}
\begin{array}{lcl}
\mathbb E \sup_{i \le t \le i+1} \norm{\mathcal{E}_{B;I}^{(i)}(t)}_{H^k}^{2p}
& \le & M^{2p}  \mathbb E  \norm{ \int_0^i E(i,s) P^\perp B(s) \,\mathrm d W_s^{Q;-} }_{H^k}^{2p}
\\[0.2cm]
& \le & M^{2p} K_{\rm dc}^{2p} p^p
\mathbb E  \big[\int_0^i e^{-\beta (i -s)} \norm{B(s)}_{HS(L^2_Q; H^k)}^2 \, \mathrm ds \big]^p .
\end{array}
\end{equation}
The desired estimate can now be obtained by utilizing the bound \eqref{eq:conv:bnd:in:hB}.
\end{proof}

\begin{lem}
\label{lem:conv:bnd:e:b:ii}
Assume that \textnormal{(HTw)}, \textnormal{(HNL)}, \textnormal{(Hq)} and \textnormal{(hE)} hold, that $T$ is an integer and $i \in \{0, \ldots , T-1\}$. Then for any process $B$ that satisfies \textnormal{(hB*)}
and any integer $p \ge 1$
we have the bound
\begin{equation}
\mathbb E \sup_{i \le t \le i+1} \norm{\mathcal{E}_{B;II}^{(i)}(t)}_{H^k}^{2p}
\le  K_{\rm cnv}^{2p} p^p \big[ p^{(n-1)p} + \Theta_2^{(n-1)p} \big] \Theta_1^{2p} .
\end{equation}
\end{lem}
\begin{proof}
Applying the maximal inequality
\eqref{eq:conv:max:ineq}
we find
\begin{equation}
\begin{array}{lcl}
\mathbb E \sup_{i \le t \le i+1} \norm{\mathcal{E}_{B;II}^{(i)}(t)}_{H^k}^{2p}
& \le & K_{rm cnv}^{2p} p^p \mathbb E \big[\int_i^{i+1} \norm{B(s)}_{HS(L^2_Q; H^k)}^2 \, \mathrm ds \big]^p,
\end{array}
\end{equation}
which leads directly to the desired bound by utilizing \eqref{eq:conv:bnd:in:hB}.
\end{proof}

For convenience, we now introduce the constant
\begin{equation}
    K_{\mathcal{E}} =
    2  \big( M^2 K_{\rm dc}^2
    /{\beta} +  K_{\rm cnv}^2 \big)^{1/2}.
\end{equation}
This allows us to combine the
two previous results to obtain
estimates for $\mathcal{E}^s_B$ over short intervals.

\begin{cor}
    Assume that \textnormal{(HTw)}, \textnormal{(HNL)}, \textnormal{(Hq)} and \textnormal{(hE)} hold, that $T$ is an integer and $i \in \{0, \ldots , T-1\}$. Then for any process $B$ that satisfies \textnormal{(hB)} and any integer $p \ge 1$
we have the bound
\begin{equation}
\label{eq:conv:est:sup:i:ip1:esb}
    \mathbb E \sup_{i \le t \le i+1} \norm{\mathcal{E}^s_B(t)}_{H^k}^{2p}
    \le K_\mathcal{E}^{2p}
    \big[ p^{np} + \Theta_2^{np} \big] \Theta_1^{2p}.
\end{equation}
\end{cor}
\begin{proof}
Assume first that (hB*) is satisfied.
Combining Lemmas \ref{lem:conv:bnd:e:b:i} and \ref{lem:conv:bnd:e:b:ii},
the estimate
 $(a+b)^{2p} \le 2^{2p-1} (a^{2p} + b^{2p})$
implies
\begin{equation}
    \mathbb E \sup_{i \le t \le i+1} \norm{\mathcal{E}^s_B(t)}_{H^k}^{2p}
    \le 2^{2p-1} p^p \big( M^2 K_{\rm dc}^2\frac{1}{\beta} +  K_{\rm cnv}^2 \big)^p
    \big[ p^{(n-1)p} + \Theta_2^{(n-1)p} \big] \Theta_1^{2p}.
\end{equation}
Applying Young's inequality we note that
\begin{equation}
    p^p \Theta_2^{(n-1)p} \le  \frac{1}{n} p^{np} + \frac{n-1}{n} \Theta_2^{n p}
    \le  p^{np} +  \Theta_2^{n p},
\end{equation}
which leads directly to the desired bound. A standard limiting argument
involving \cite[Cor. 3.8]{bosch2024multidimensional} allows us
to remove the restriction that $B$ is a finite-rank process.
\end{proof}

\begin{cor}
\label{cor:conv:bnd:sup:e:s:b}
    Assume that \textnormal{(HTw)}, \textnormal{(HNL)}, \textnormal{(Hq)} and \textnormal{(hE)} hold and that $T \ge 2$ is an integer. Then for any process $B$ that satisfies \textnormal{(hB)} and any (real) $p \ge 1$
we have the bound
   \begin{equation}
   \label{eq:conv:bnd:fin:e:s:b}
\mathbb    E \sup_{0 \le t \le T}
    \norm{\mathcal{E}^s_B(t)}^{2p}
    \le
      K_{\mathcal{E}}^{2p} (16e n )^{np} \big(p^{n p} +  [\ln T +   \Theta_2]^{n p}\big)  \Theta_1^{2p}.
\end{equation}
\end{cor}
\begin{proof}
We first note that
\begin{equation}
\begin{array}{lcl}
   \mathbb  E \sup_{0 \le t \le T}
    \norm{\mathcal{E}^s_B(t)}^{2p}
& \le &
\mathbb   E \max_{i \in \{0, \ldots, T-1\} }
      \sup_{i \le t \le i +1 } \norm{\mathcal{E}^s_B(t)}^{2p}.
\end{array}
\end{equation}
In view of \eqref{eq:conv:est:sup:i:ip1:esb},
we may apply
Corollary \ref{cor:prlm:exp:bnd:to:exp:bnd:with:n}
to obtain
the desired bound.
\end{proof}

We now turn to the integrated expression $\mathcal{I}^s_B$. The key observation is that
\begin{equation}
\label{eq:conv:bnd:sup:i:b:vs:i:b:p1}
\sup_{i\leq t\leq i+1}\mathcal I^s_B(t)\leq e^\beta \mathcal I^s_B(i+1),
\end{equation}
which will allow us to apply
Corollary \ref{cor:prlm:exp:bnd:to:exp:bnd:with:n}
once more once the expectation
of $\mathcal{I}^s_B$ at individual times $t$ is understood.

\begin{lem}
    Assume that \textnormal{(HTw)}, \textnormal{(HNL)}, \textnormal{(Hq)} and \textnormal{(hE)} hold and that $T \ge 2$ is an integer. Then for any process $B$ that satisfies \textnormal{(hB)} and
    any integer $p \ge 1$,
we have the bound
\begin{equation}
\label{eq:conv:bnd:i:s:b:p:t}
\begin{aligned}
  \mathbb E \, \mathcal{I}^s_B(t)^{p}
 &\le
    K^{p}_{\rm mr} (16 en)^{np}\big( K_{\mathcal{E}}^{2}  + 2 \beta^{-1} \big)^p  \big[p^{np} + [\ln T + \Theta_2 \big]^{np} \big] \Theta_1^{2p}
\end{aligned}
\end{equation}
for any $0 \le t \le T$.
\end{lem}
\begin{proof}
The maximal regularity bound  \eqref{eq:conv:max:reg:bnd}
together with
 \eqref{eq:conv:bnd:in:hB} imply that
\begin{equation}
\begin{aligned}
  \mathbb E \, \mathcal{I}_B(t) ^{p}
 &\le
    K^{p}_{\rm mr} \, \mathbb E    \sup_{0 \le t \le T}
     \|\mathcal{E}_B(t)\|_{H^k}^{2p}
    + p^{p} K^{p}_{\rm mr} \beta^{-p} \big[p^{(n-1)p} + \Theta_2^{(n-1)p} \big] \Theta_1^{2p} .
\end{aligned}
\end{equation}
Using the prior bound \eqref{eq:conv:bnd:fin:e:s:b}, this leads to
\begin{equation}
\begin{aligned}
  \mathbb E \, \mathcal{I}^s_B(t) ^{p}
 &\le
    K^{p}_{\rm mr} K_{\mathcal{E}}^{2p} (16 en)^{np} \big[p^{np} + [\ln T + \Theta_2 \big]^{np} \big] \Theta_1^{2p}
     \\
     &   \qquad
    + p^{p} K^{p}_{\rm mr} \beta^{-p} \big[p^{(n-1)p} + \Theta_2^{(n-1)p} \big] \Theta_1^{2p}.
\end{aligned}
\end{equation}
Applying Young's inequality, this yields
\begin{equation}
\begin{aligned}
  \mathbb E \, \mathcal{I}^s_B(t) ^{p}
 &\le
    K^{p}_{\rm mr} K_{\mathcal{E}}^{2p} (16 en)^{np} \big[p^{np} + [\ln T + \Theta_2 \big]^{np} \big] \Theta_1^{2p}
     \\
     &   \qquad
    + 2 K^{p}_{\rm mr} \beta^{-p} \big[p^{np} + \Theta_2^{np} \big] \Theta_1^{2p},
\end{aligned}
\end{equation}
which can be absorbed by the stated bound.
\end{proof}

\begin{cor}
\label{cor:conv:sup:i:s:0:T}
    Assume that \textnormal{(HTw)}, \textnormal{(HNL)}, \textnormal{(Hq)} and \textnormal{(hE)} hold and that $T \ge 2$ is an integer. Then for any process $B$ that satisfies \textnormal{(hB)} and any (real) $p \ge 1$
we have the bound
\begin{equation}
   \mathbb E \sup_{0 \le t \le T}
    \mathcal{I}^s_B(t)^{p}
    \le
      e^{\beta p} K^{p}_{\rm mr} \big( K_{\mathcal{E}}^{2}  + 2 \beta^{-1} \big)^p \big(p^{n p} +  [ \ln T +   \Theta_2]^{n p}\big)\big( (32e n )^{2n} \Theta_1^2\big)^{p}.
\end{equation}
\end{cor}
\begin{proof}
Using \eqref{eq:conv:bnd:sup:i:b:vs:i:b:p1},
we observe that
\begin{equation}
\begin{array}{lcl}
    \mathbb E \sup_{0 \le t \le T}
    \mathcal{I}^s_B(t)^{p}
& \le &
    \mathbb E \max_{i \in \{0, \ldots, T-1\} }
      \sup_{i \le t \le i +1 } \mathcal{I}^s_B(t)^{p}
\\[0.2cm]
& \le &
 e^{\beta p} \mathbb E \max_{i \in \{1, \ldots, T\} } \mathcal{I}^s_B(i)^{p}.
\end{array}
\end{equation}
    Appealing to Corollary \ref{cor:prlm:exp:bnd:to:exp:bnd:with:n},
    the estimate \eqref{eq:conv:bnd:i:s:b:p:t}
hence leads to the bound
\begin{equation}
\mathbb    E \sup_{0 \le t \le T}
    \mathcal{I}^s_B(t)^{p}
    \le
      e^{\beta p} K^{p}_{\rm mr} (16 en)^{np}\big( K_{\mathcal{E}}^{2}  + 2 \beta^{-1} \big)^p \big(p^{n p} +  [2 \ln T +   \Theta_2]^{n p}\big) ( (16e n )^n \Theta_1^2)^{p},
\end{equation}
which can be absorbed by the stated estimate.
\end{proof}

\begin{proof}[Proof of Proposition \ref{prp:cnv:hb}]
The desired bound follows directly from Corollaries \ref{cor:conv:bnd:sup:e:s:b}
and \ref{cor:conv:sup:i:s:0:T}.
\end{proof}

\section{Smoothness}
\label{sec:sm}

In this section we study the smoothness of the nonlinear functions in \eqref{eq:mr:decomp:r:sigma} and obtain estimates for these terms together with their derivatives. These results can be seen as an extension of their counterparts in \cite[{\S}4]{bosch2024multidimensional}, where it sufficed to obtain Lipschitz bounds instead of full derivative estimates.

\subsection{Preliminaries}
We start by considering the smoothness of Nemytskii operators
between various $H^k$-spaces.
In particular, we first consider general smooth functions $\Theta: \mathbb R^n \to \mathbb R^N$,
noting that $\Theta$ and its derivatives can be interpreted as Nemytskii
operators that act on functions $h: \mathcal{D} \to \mathbb R^n$
via the standard pointwise substitution
\begin{equation}
\label{eq:sm:def:d:j:theta}
    \big[ D^j \Theta ( h ) \big] (x,x_\perp) = \big[D^j \Theta\big]\big(h(x,x_\perp) \big),
    \qquad \qquad (x,x_\perp) \in \mathbb R \times \mathbb T^{d-1} = \mathcal{D}.
\end{equation}
For ease of notation we introduce the product shorthand
\begin{equation}
    \mathcal{P}^{(j)}_k [v] = ||v_1||_{H^k} \cdots ||v_j||_{H^k}.
\end{equation}
As a preparation, we recall that for any $k > d/2$
we can find a constant $K > 0$ so that
for any bounded $j$-linear map  $\Lambda: (\mathbb R^n)^{j} \to \mathbb R^N$, the bound
\begin{equation}
\label{eq:nl:bnd:multilinear:general}
    \| \Lambda[ \partial^{\alpha_1} v_1, \ldots, \partial^{\alpha_j} v_{j}] \|_{L^2(\mathcal D;\mathbb R^N)}
    \le  K |\Lambda|  \mathcal{P}^{(j)}_j[v] 
\end{equation}
holds for any tuple $(v_1, \ldots, v_j) \in (H^k)^{j}$, provided that $|\alpha_1| + \ldots + |\alpha_j| \le k$. This is related to the fact that $H^k$ is an algebra under multiplication for $k > d/2$ in the sense that $\|vw\|_{H^k} \le K \|v\|_{H^k}\|w\|_{H^k}$; see \cite[Thm. 4.39]{adams2003sobolev}. The function $\Phi$ below should be seen as a reference function, which in the sequel we will take to be $\Phi_0$.

\begin{lem}
\label{lem:sm:deriv:theta:bnds}
Pick $k > d/2$ together with $j \ge 0$ and
consider a $C^{k+j}$-smooth function $\Theta: \mathbb R^n \to \mathbb R^N$ for which $D^{\ell} \Theta$ is globally Lipschitz for all $0 \le \ell \le k+j$. Assume furthermore that $\Phi$ is bounded with $\Theta(\Phi) \in L^2$
and $\partial^\beta \Phi\in H^{k}$ for very multi-index $\beta \in \mathbb{Z}^d_{\ge 0}$ with $|\beta|=1$.
Then there exists a constant $K > 0$ so that for each tuple $(v_1, \ldots v_j)\in (H^k)^{j}$
and any $1 \le \tilde{k} \le k$ we have
\begin{equation}
\label{eq:sm:bnd:d:j:theta}
    || D^j\Theta(\Phi + w)[v_1, \ldots, v_j]||_{H^{\tilde{k}}} \le K (1 + ||w||_{H^k}^{\tilde{k}} ) \mathcal{P}^{(j)}_{k}[v]
\end{equation}
for any $w \in H^k$, while for any pair $w_A,w_B \in H^k$ we have
\begin{equation}
\label{eq:sm:bnd:d:j:theta:lip}
    || \big( D^j\Theta(\Phi + w_A) -D^j\Theta(\Phi + w_B) \big)  [v_1, \ldots, v_j] ||_{H^{\tilde{k}}} \le K \big(1 + ||w_A||_{H^k}^{\tilde{k}-1}
    + ||w_B||_{H^k}^{\tilde{k}}\big) \|w_A - w_B\|_{H^k} \mathcal{P}^{(j)}_{k}[v].
\end{equation}
\end{lem}
\begin{proof}
For any multi-index $\alpha \in \mathbb{Z}^d_{\ge 0}$ with $|\alpha| \le \tilde{k}$, we note that
$\partial^\alpha D^j \Theta(\Phi + w)[v_1, \ldots, v_j]$
can be written as a sum of terms of the form
\begin{equation}
\label{eq:sm:def:i:1:prep}
    \mathcal{I}_1 = D^{j+\ell} \Theta( \Phi + w) [ \partial^{\beta_1} (\Phi + w), \ldots,  \partial^{\beta_\ell} (\Phi + w), \partial^{\beta_{\ell+1}} v_1, \ldots, \partial^{\beta_{\ell + j}} v_j] ,
\end{equation}
in which $0 \le \ell \le |\alpha|$ and $|\beta_i| \ge 1$ for $1 \le i \le \ell$,
together with $|\beta_1| + \ldots + |\beta_{\ell + j}| = |\alpha|$.
When $j+\ell > 0$, we can use \eqref{eq:nl:bnd:multilinear:general}
together with the global Lipschitz properties of $\Theta$,
which ensure that $D^{j+\ell}\Theta$ is bounded pointwise, to conclude
\begin{equation}
\label{eq:sm:bnd:mathcal:i:1}
    \norm{ \mathcal{I}_1 }_{L^2} \le K \big[ 1 + \norm{w}_{H^k}^{\ell}]
    \norm{v_1}_{H^k} \cdots \norm{v_j}_{H^k}.
\end{equation}
When $\ell = j = 0$, we can use the pointwise bound $|\Theta(\Phi + w)| \le |\Theta(\Phi)| + K |w|$ and the inclusion $\Theta(\Phi_0) \in L^2$
to conclude $\norm{\Theta(\Phi + w)}_{L^2} \le K [ 1 + \norm{w}_{L^2}]$,
completing the proof of \eqref{eq:sm:bnd:d:j:theta}.

Turning to the Lipschitz bound \eqref{eq:sm:bnd:d:j:theta:lip}, we note
that $\partial^\alpha D^j \big(\Theta(\Phi + w_A)-\Theta(\Phi + w_B)\big)[v_1, \ldots, v_j]$ can be expressed
as a finite sum of expressions of two types.
Up to permutations of the first
$\ell$ elements,
the first type is given by
\begin{equation}
\label{eq:nl:def:i:1:theta}
\mathcal{I}_{II} =     D^{j+\ell}\Theta(\Phi + w_A)[\partial^{\beta_1}(w_A - w_B), \partial^{\beta_2} (\Phi + w_{\#_1}),
    \ldots,
    \partial^{\beta_{\ell}} (\Phi + w_{\#_{\ell}}),
    \partial^{\beta_{\ell+1}} v_1, \ldots , \partial^{\beta_{\ell+j}} v_j
    ],
\end{equation}
with $\#_i \in \{A , B\}$ and multi-indices $\{\beta_i\}_{i=1}^{\ell+j} \in \mathbb{Z}^d_{ \ge 0} $ that satisfy $|\beta_i| \ge 1$, for each $1 \le i \le \ell\leq |\alpha|$, together with $|\beta_1| + \ldots + |\beta_{\ell+j}| = |\alpha|$.
The second type is given by
\begin{equation}
\label{eq:nl:def:i:2:theta}
\mathcal{I}_{III} =
\Big(D^{j+\ell}\Theta (\Phi + w_A) - D^{j+\ell} \Theta (\Phi + w_B)\Big) \Big[ \partial^{\beta_1} ( \Phi + w_B), \dots, \partial^{\beta_\ell} (\Phi + w_B)  ,
 \partial^{\beta_{\ell+1}} v_1, \ldots , \partial^{\beta_{\ell+j}} v_j
\Big],
\end{equation}
with the same conditions on $\{\beta_i\}_{i=1}^\ell$, but where now $\ell = 0$ is allowed. This can be readily verified with induction.

Using \eqref{eq:nl:bnd:multilinear:general}
we obtain the bounds
\begin{equation}
\begin{array}{lcl}
   \| \mathcal{I}_{II} \|_{L^2(\mathcal D;\mathbb R^N)}
   & \le & K 
   \|w_A - w_B\|_{H^k} \big[ 1 + \|w_A\|^{\ell-1}_{H^k} + \|w_B\|^{\ell-1}_{H^k} \big]
   \norm{v_1}_{H^k} \cdots \norm{v_j}_{H^k}
   ,
   \\[0.2cm]
   \| \mathcal{I}_{III}\|_{L^2(\mathcal D;\mathbb R^N) }
   & \le & K \|w_A - w_B\|_{H^k} \big[ 1 + \|w_B\|^{\ell}_{H^k} \big]
    \norm{v_1}_{H^k} \cdots \norm{v_j}_{H^k} .
\end{array}
\end{equation}
Both terms can be absorbed in \eqref{eq:sm:bnd:d:j:theta:lip},
completing the proof.
\end{proof}

\begin{cor}
\label{cor:sm:diff:of:theta:map}
Consider a $C^{k_* + r + 1}$-smooth function $\Theta: \mathbb R^n \to \mathbb R^N$ for which $D^{\ell} \Theta$ is globally
Lipschitz for all $0 \le \ell \le k_* + r + 1$.
Assume furthermore that $\Phi$ is bounded with $\Theta(\Phi) \in L^2$
and $\partial^\beta \Phi\in H^{k_* + r + 1}$ for very multi-index $\beta \in \mathbb{Z}^d_{\ge 0}$ with $|\beta|=1$.
Then for any integer $0 \le j \le r + 1$
and every $k_* \le k \le k_* + r + 1 - j$
we have the smoothness property
\begin{equation}
  w \mapsto   \Theta(\Phi + w) \in C^j( H^k ; H^k).
\end{equation}
In addition, the derivatives of this map are given by \eqref{eq:sm:def:d:j:theta}
with $h = \Phi + w$.
\end{cor}
\begin{proof}
Fix $0 \le j \le r + 1$ and $k_* \le k \le k_* + r + 1 - j$.
Then for each $0 \le \ell \le j$ the pointwise derivative $w \mapsto D^\ell \Theta(\Phi + w)$
can be interpreted as a well-defined and continuous mapping into the space of
bounded $\ell$-linear functionals $(H^{k})^{\ell} \to H^k$
on account of \eqref{eq:sm:bnd:d:j:theta}
  and \eqref{eq:sm:bnd:d:j:theta:lip}, respectively. In addition,
  for each $0 \le \ell < j$
  and each pair $w_A,w_B \in H^k$
  we may write
  \begin{equation}
      \mathcal{Q} = D^{\ell}\Theta(\Phi + w_B)[v_1, \ldots, v_\ell] - D^{\ell}\Theta(\Phi+w_A)[v_1, \ldots, v_\ell]
      - D^{\ell+1}\Theta(\Phi+w_A)[w_B - w_A, v_1, \ldots, v_\ell]
  \end{equation}
  and observe that
  \begin{equation}
      \mathcal Q = \int_0^1 \big( D^{\ell+1}\Theta(\Phi +w_A + t (w_B - w_A)) - D^{\ell+1}\Theta(\Phi+w_A) \big)[w_B-w_A, v_1, \ldots , v_\ell] \, \mathrm dt.
  \end{equation}
  Applying \eqref{eq:sm:bnd:d:j:theta:lip} we now find
  the quadratic bound
  \begin{equation}
      \norm{\mathcal Q}_{H^k} \le K \big(1 + ||w_A||_{H^k}^{k-1}
    + ||w_B||_{H^k}^{k}\big) \|w_A - w_B\|_{H^k}^2 ||v_1||_{H^k} \cdots ||v_\ell||_{H^k},
  \end{equation}
  which implies the stated differentiability properties.
\end{proof}

For some of our results it is crucial to isolate the highest derivatives,  since we do not always have uniform control over their size; see, e.g., the integral in \eqref{eq:mr:def:n:res}. To this end, we introduce the notation
\begin{equation}
\begin{array}{lcl}
    \mathcal{P}^{(j)}_{k_A, k_B}[v]
    &=&\norm{v_1}_{H^{k_B}} \norm{v_2}_{H^{k_A}} \cdots \norm{v_j}_{H^{k_A}}
    + \ldots + \norm{v_1}_{H^{k_A}}  \cdots \norm{v_{j-1}}_{H^{k_A}} \norm{v_j}_{H^{k_B}}
\\[0.2cm]
    & = & \sum_{j'=1}^j  \norm{v_{j'}}_{H^{k_B}} \prod_{i \neq j'} \norm{v_i}_{H^{k_A}}.
\end{array}
\end{equation}

\begin{lem}
\label{lem:sm:hkp1:prep}
Pick $k > d/2$ together with $j \ge 0$ and
consider a $C^{k+j+1}$-smooth function $\Theta: \mathbb R^n \to \mathbb R^N$ for which $D^{\ell} \Theta$ is globally Lipschitz for all $0 \le \ell \le k+j+1$.
Assume furthermore that $\Phi$ is bounded with $\Theta(\Phi) \in L^2$
and $\partial^\beta \Phi\in H^{k+1}$ for every multi-index $\beta \in \mathbb{Z}^d_{\ge 0}$ with $|\beta|=1$.
Then there exists a constant $K > 0$ so that for each $w \in H^{k+1}$ and each tuple $(v_1, \ldots, v_j)\in (H^{k+1})^{j}$ we have
\begin{equation}
\begin{array}{lcl}
    || D^j\Theta(\Phi_0 + w)[v_1, \ldots, v_j]||_{H^{k+1}}
    & \le & K (1 + ||w||_{H^k}^{k}||w||_{H^{k+1}} ) P^{(j)}_k[v]
+ K (1 + ||w||_{H^k}^{k} ) \mathcal{P}^{(j)}_{k, k+1}[v] .
\end{array}
\end{equation}
\end{lem}
\begin{proof}
    Inspecting the term $\mathcal{I}_1$ in \eqref{eq:sm:def:i:1:prep}
    where now also $|\alpha|=k+1$ is allowed, we see that the desired
    estimate can again be obtained by following the proof of Lemma \ref{lem:sm:deriv:theta:bnds}.
    Indeed, if necessary one can apply a differential operator $\partial^\gamma$
    with $|\gamma| = 1$ to one of the $v_i$ or $w$ before appealing to the bound \eqref{eq:nl:bnd:multilinear:general}.

\end{proof}

\subsection{Auxiliary functions}

We first consider the cut-off functions
\begin{equation}
\chi_h(u,\gamma)=\chi_{\rm high}(\|u-T_\gamma\Phi_{0}\|_{L^2})\quad\text{and}\quad
    \chi_l(u,\gamma)=\big[ \chi_{\rm low}\big(-\langle  u,T_\gamma\psi_{\rm tw}'\rangle_{L^2} \big) \big]^{-1},
\end{equation}
as defined in \eqref{eq:list:def:chi:h:l}.
We note here that $\chi_{\rm high}$ and $\chi_{\rm low}^{-1}$ are infinitely smooth and bounded. Our first result states that also $\chi_h$ and $\chi_l$ are infinitely smooth and provides a uniform bound on the derivatives that we need, complementing the results in \cite[Lem. 4.8]{bosch2024multidimensional}.

\begin{lem}
\label{lem:sm:cutoffs}
Assume that \textnormal{(HNL)} and \textnormal{(HTw)} are satisfied. Then we have
the smoothness properties
\begin{equation}
    w \mapsto \chi_h(\Phi_0 + w,0) \in C^\infty(L^2; \mathbb R),
    \qquad \qquad
    w \mapsto \chi_l(\Phi_0 + w,0) \in C^\infty(L^2; \mathbb R).
\end{equation}
In addition, there exists $K> 0$ so that for all $0 \le j \le r$
and all tuples $(v_1, \ldots, v_j) \in (L^2)^j$
we have the bound
\begin{equation}
    | D^j \chi_h (\Phi_0 + w, 0) [v_1, \ldots, v_j] |
    + | D^j \chi_l (\Phi_0 + w, 0) [v_1, \ldots, v_j] |
    \le K \norm{v_1}_{L^2} \cdots \norm{v_j}_{L^2}. 
\end{equation}
\end{lem}
\begin{proof}
    The statements for $\chi_l$ follow from the fact that the map $w \mapsto \langle w, \psi'_{\rm tw} \rangle$ is bounded and linear from $L^2$ into $\mathbb R$. Turning to $\chi_h$, we first mention that
    the cut-off allows us to assume an a-priori bound for
    $\norm{ w }_{L^2}$. This implies that the two derivatives of
    the map $w \mapsto \langle w , w   \rangle_{L^2} $ are uniformly bounded,
    completing the proof.
\end{proof}

The following two results concern the function $g$. At several points it is convenient to use bounds for $g$ in the lower-regularity spaces $L^2$ and $H^1$, which we hence provide separately.

\begin{lem}
\label{lem:sm:bnds:on:g}
Suppose that \textnormal{(HNL)} and \textnormal{(HTw)} hold.
Pick 
an integer $0 \le j \le r + 1$
and an integer $k_* \le k \le  k_j$.
Then we have
\begin{equation}
    w \mapsto g(\Phi_0 + w) \in C^j(H^{k};H^{k}) \cup C^j(H^{k+1};H^{k+1}) .
\end{equation}
In addition, there is a constant $K > 0$ so that
for any $w \in H^k$ and any tuple $(v_1, \ldots v_j) \in (H^k)^{j}$
we have the bound
\begin{equation}
\begin{array}{lcl}
    \norm{ D^j g(\Phi_0 + w)[v_1, \ldots, v_j]}_{H^{k}}
    & \le & K (1 + ||w||_{H^k}^{k} ) \mathcal{P}^{(j)}_k[v],
\\[0.2cm]
\end{array}
\end{equation}
while for any $w \in H^{k+1}$ and any tuple $(v_1, \ldots, v_j) \in (H^{k+1})^{j}$
we have
\begin{equation}
\begin{array}{lcl}
\norm{ D^j g(\Phi_0 + w)[v_1, \ldots, v_j]}_{H^{k+1}}
    & \le & K (1 + ||w||_{H^k}^{k}||w||_{H^{k+1}} ) \mathcal{P}^{(j)}_k[v]
\\[0.2cm]
& & \qquad
+ K (1 + ||w||_{H^k}^{k} ) \mathcal{P}^{(j)}_{k,k+1}[v].
\end{array}
\end{equation}
\end{lem}
\begin{proof}
In view of the smoothness assumptions on $g$ formulated in (HNL), these statements follow from Lemmas \ref{lem:sm:deriv:theta:bnds} and \ref{lem:sm:hkp1:prep} and Corollary \ref{cor:sm:diff:of:theta:map}.
\end{proof}

\begin{lem}
\label{lem:sm:g:l2:h1}
Suppose that \textnormal{(HNL)} and \textnormal{(HTw)} hold.
Then there exists $K > 0$ so that for any integer
$1 \le j \le r$,
any  $w \in H^{k_*}$ and any tuple $(v_1, \ldots, v_j) \in (H^{k_*})^j$ we have
\begin{equation}
\label{eq:sm:bnds:d:j:g:hks}
\begin{array}{lcl}
    \norm{ D^j g(\Phi_0 + w)[v_1, \ldots, v_j] }_{ L^2}
    & \le & K  \mathcal{P}^{(j)}_{k_*}[v] ,
\\[0.2cm]
\norm{ D^j g(\Phi_0 + w)[v_1, \ldots, v_j] }_{ H^1}
    & \le & K\big[ 1 + ||w||_{H^{k_*}}]  \mathcal{P}^{(j)}_{k_*}[v] .
\end{array}
\end{equation}
In addition, for any $w \in H^1$ we have the bounds
\begin{equation}
\label{eq:sm:g:l2:h1:bnd}
\begin{array}{lcl}
    \norm{  g(\Phi_0 + w)}_{ L^2}
    & \le & K [ 1 + \norm{w}_{L^2}] ,
\\[0.2cm]
\norm{  \chi_h(\Phi_0 + w, 0) g(\Phi_0 + w)}_{ L^2}
    & \le & K  ,
\\[0.2cm]
\norm{  g(\Phi_0 + w) }_{ H^1}
    & \le & K\big[ 1 + ||w||_{H^{1}}] .
\end{array}
\end{equation}
\end{lem}
\begin{proof}
The $L^2$-estimate in \eqref{eq:sm:bnds:d:j:g:hks}
follows from \eqref{eq:nl:bnd:multilinear:general} and the uniform pointwise bounds
available for $D^j g$. The $H^1$-estimate in \eqref{eq:sm:bnds:d:j:g:hks}
is a consequence of Lemma \ref{lem:sm:deriv:theta:bnds}.
Finally, the bounds \eqref{eq:sm:g:l2:h1:bnd} follow from \cite[Lem. 4.6, Eq. (4.36) and Lem. 4.12]{bosch2024multidimensional}.
\end{proof}

We now turn to the functions $b$, $\tilde{\nu}$ and $\tilde{\mathcal{K}}_C$.
Inspecting the definition \eqref{eq:list:def:b:nu:tilde},
we introduce the function
\begin{equation}
    \Gamma_b: L^2(\mathcal{D}; \mathbb R^{n \times m}) \ni h \to \langle h[ \,\cdot\,] , \psi_{\rm tw} \rangle_{L^2} \in HS(L^2_Q; \mathbb R),
\end{equation}
which is bounded on account of \eqref{eq:mr:hs:bnd:z},
and write
\begin{equation}
\label{eq:sm:repr:b}
    b(u, 0) = - \chi_h(u, 0)^2 \chi_l(u, 0) \Gamma_b[ g(u)] .
\end{equation}
In addition,
we introduce the functional
\begin{equation}
    \Gamma_{\tilde{\nu}}:
L^2(\mathbb{R}^{m \times n})
\times L^2(\mathbb{R}^{m \times n}) \ni
   (  v, w) \mapsto  \langle Q  v \psi_{\rm tw}, w \psi_{\rm tw} \rangle_{L^2}
    \in \mathbb R,
\end{equation}
which satisfies the bound
\begin{equation}
    | \Gamma_{\tilde{\nu}}[ v, w] | \le \norm{q}_{L^1(\mathcal{D};\mathbb{R}^{m \times m})}
     \norm{v}_{L^2(\mathbb{R}^{m \times n})} \norm{w}_{L^2(\mathbb{R}^{m \times n})}\norm{\psi_{\rm tw}}_\infty^2
\end{equation}
and allows us to write
\begin{equation}
\label{eq:sm:repr:nu:t}
    \tilde{\nu}( u, 0) =
    \frac{1}{2} \chi_h(u, 0)^4  \chi_l(u,0)^2 \Gamma_{\tilde{\nu}}[ g^T(u), g^T(u)
    ].
\end{equation}
Finally, inspecting the definition \eqref{eq:list:def:wt:k:c}
we introduce the functional
\begin{equation}
    \Gamma_C: L^2(\mathcal{D}; \mathbb R^{m \times n}) \ni  h \mapsto Q h \psi_{\rm tw} \in L^2_Q
\end{equation}
and write
\begin{equation}
\label{eq:sm:repr:wt:k:c}
    \widetilde{\mathcal{K}}_C(u, 0)
    = \chi_l(u, 0) \chi_h(u,0) \Gamma_C[ g^T(u)].
\end{equation}
Note that the computations in \cite[Lem 4.8 and 4.17]{bosch2024multidimensional}
imply that for any multi-index $\beta \in \mathbb Z^d_{\ge 0}$ with $|\beta|=1$ we have
\begin{equation}
||\Gamma_C h||_{L^2_Q }
\le K  ||h||_{L^2(\mathcal{D}; \mathbb R^{m \times n})},
\qquad
||\partial^\beta \Gamma_C h||_{L^2_Q} \le K ||h||_{H^1(\mathcal{D}; \mathbb R^{m \times n})}.
\end{equation}

\begin{lem}
\label{lem:sm:b:nu:kct}
Suppose that \textnormal{(HNL)}, \textnormal{(HTw)} and \textnormal{(Hq)} hold.
    Pick 
    a multi-index $\beta \in \mathbb Z^{d}_{\ge 0}$ with $| \beta | = 1$. Then we have
    \begin{equation}
    \begin{array}{lcl}
        w \mapsto b(\Phi_0 + w, 0)  &\in & C^r(H^{k_*}; HS( L^2_Q;\mathbb R )),
        \\[0.2cm]
        w \mapsto   \widetilde{\nu}(\Phi_0 + w, 0) &\in & C^r( H^{k_*} ; \mathbb R) ,
        \\[0.2cm]
        w \mapsto \widetilde{\mathcal{K}}_C (\Phi_0 + w, 0)  &\in & C^r(H^{k_*}; L^2_Q),
        \\[0.2cm]
        w \mapsto \partial^\beta \widetilde{\mathcal{K}}_C (\Phi_0 + w, 0)  &\in & C^r(H^{k_*}; L^2_Q) .
    \end{array}
\end{equation}
In addition, there exists $K > 0$ so that for all $0 \le j \le r$,
all $w \in H^{k*}$ and all tuples $(v_1, \ldots, v_j) \in (H^{k_*})^j$
we have the bounds
\begin{equation}
\label{eq:nl:est:deriv:b:nut:wtkc}
\begin{array}{lcl}
 \norm{    D^j b(\Phi_0 + w) [v_1, \ldots, v_j] }_{HS(L^2_Q;\mathbb R)}
    &\le & K  \mathcal{P}^{(j)}_{k_*}[v] ,
\\[0.2cm]
| D^j \widetilde{\nu}(\Phi_0 + w) [v_1, \ldots, v_j] |
   &\le &  K  \mathcal{P}^{(j)}_{k_*}[v],
\\[0.2cm]
\norm{D^j \widetilde{\mathcal{K}}_C(\Phi_0 + w) [v_1, \ldots, v_j] }_{L^2_Q}
    &\le & K  \mathcal{P}^{(j)}_{k_*}[v] ,
\\[0.2cm]
\norm{ D^j \partial^\beta \widetilde{\mathcal{K}}_C(\Phi_0 + w) [v_1, \ldots, v_j] }_{L^2_Q}
    &\le & K \big[ 1 + \norm{w}_{H^{k_*}} \big]
    \mathcal{P}^{(j)}_{k_*}[v] .
\end{array}
\end{equation}
\end{lem}
\begin{proof}
In view of the representations \eqref{eq:sm:repr:b}, \eqref{eq:sm:repr:nu:t}
and \eqref{eq:sm:repr:wt:k:c},
these statements follow from the smoothness of $\chi_h$ and $\chi_l$
as outlined in Lemma \ref{lem:sm:cutoffs}, the smoothness
of linear maps and the bounds for $g$
established in Lemma \ref{lem:sm:g:l2:h1}.
\end{proof}

\subsection{Bounds for \texorpdfstring{$\mathcal{R}_I$}{R1}}

In this part we set out to obtain bounds for the nonlinearity
\begin{equation}
\label{eq:sm:reminder:r:i}
\mathcal{R}_{I}(v)
    =
     \mathcal{N}_f(v)
    - \chi_l (\Phi_0 + v, 0) \langle \mathcal{N}_f(v), \psi_{\mathrm{tw}} \rangle \partial_x [\Phi_0 + v],
\end{equation}
where $\mathcal{N}_f$ is given by
\begin{equation}
    \mathcal{N}_f(w) = f(\Phi_0 + w) - f(\Phi_0) - Df(\Phi_0) w.
\end{equation}
Our first three results concern $f$ in $H^k$ and $L^2$ together with $D\mathcal{N}_f$,
showing that $\mathcal{N}_f$ is indeed quadratic.

\begin{lem}
\label{lem:nl:bnds:f}
Suppose that \textnormal{(HNL)} and \textnormal{(HTw)} hold.
    Pick  an integer $0 \le j \le r + 1$
    and an integer $k_* \le k \le  k_j$.
    Then we have
\begin{equation}
    w \mapsto f(\Phi_0 + w) \in C^j(H^{k};H^{k}).
\end{equation}
In addition, there is a constant $K > 0$ so that
for any $w \in H^k$ and any tuple $(v_1, \ldots v_j) \in (H^k)^{j}$
we have the bound
\begin{equation}
\begin{array}{lcl}
    \norm{ D^j f(\Phi_0 + w)[v_1, \ldots, v_j]}_{H^{k}}
    & \le & K (1 + ||w||_{H^k}^{k} ) \mathcal{P}^{(j)}_k[v] .
\\[0.2cm]
\end{array}
\end{equation}
\end{lem}
\begin{proof}
In view of the smoothness assumptions in (HNL),
the statements
 follow directly from Lemma \ref{lem:sm:deriv:theta:bnds}
 and Corollary \ref{cor:sm:diff:of:theta:map}.
\end{proof}

\begin{lem}
\label{lem:sm:f:l2}
Suppose that \textnormal{(HNL)} and \textnormal{(HTw)} hold.
    Then there exists $K > 0$ so that for any integer $1 \le j \le r$, any $w \in H^{k_*}$ and any tuple $(v_1, \ldots, v_j) \in (H^{k_*})^j$ we have
    \begin{equation}
    \label{eq:sm:bnd:df:j:l2}
        \norm{D^j f(\Phi_0 + w)[v_1, \ldots, v_j]}_{L^2} \le K \mathcal{P}^{(j)}_{k_*}[v] ,
    \end{equation}
    together with
    \begin{equation}
    \label{eq:sm:bnd:f:l2}
        \norm{f(\Phi_0 + w)}_{L^2} \le K [1 + \norm{w}_{L^2}] .
    \end{equation}
\end{lem}
\begin{proof}
The bound \eqref{eq:sm:bnd:df:j:l2} follows by considering \eqref{eq:sm:bnd:mathcal:i:1} with $\ell= 0$.
    On the other hand, \eqref{eq:sm:bnd:f:l2} follows
    from the uniform pointwise bound on $Df$
    and the fact that $f(\Phi_0) \in L^2$.
\end{proof}

\begin{lem}
\label{lem:sm:deriv:n:f}
Suppose that \textnormal{(HNL)} and \textnormal{(HTw)} hold.
Then there exists $K > 0$ so that for any $w \in H^{k_*}$ and $v \in H^{k_*}$ we have the bounds
\begin{equation}
\begin{array}{lcl}
\norm{  D\mathcal{N}_f(w)[v]}_{H^{k_*}}
    & \le & K (1 + ||w||_{H^{k_*}}^{k_*})||w||_{H^{k_*}} \norm{v}_{H^{k_*}} ,
\\[0.2cm]
    \norm{  D\mathcal{N}_f( w)[v]  }_{L^2}
    & \le & K \norm{w}_{H^{k_*}} \norm{v}_{H^{k_*}} .
\\[0.2cm]
\end{array}
\end{equation}
\end{lem}
\begin{proof}
    In view of the fact that $D \mathcal{N}_f(0)= 0$
    and $D^2 \mathcal{N}_f = D^2 f$, we may write
    \begin{equation}
    D\mathcal{N}_f(w)[v]
    = \int_0^1 D^2 f( \Phi_0 + t w) [w, v] \, \mathrm dt.
    \end{equation}
    The desired bounds now
    follow directly from Lemmas  \ref{lem:nl:bnds:f}
    and \ref{lem:sm:f:l2}.
\end{proof}

\begin{cor}
\label{cor:sm:r:i}
Suppose that \textnormal{(HNL)} and \textnormal{(HTw)} hold.
Pick an integer $0 \le j \le r$
and an integer  $k_* \le k \le  k_j$.
Then we have
\begin{equation}
\begin{array}{lcl}
    w \mapsto \mathcal{R}_{I}(\Phi_0 + w)
     &\in & C^j(H^{k+1};H^{k}).
\\[0.2cm]
\end{array}
\end{equation}
In  addition, there is a constant $K > 0$ so that
for any $w \in H^k$ and any tuple $(v_1, \ldots v_j) \in (H^k)^{j}$
we have the bound
\begin{equation}
\label{eq:nl:bnd:d:j:r:i}
\begin{array}{lcl}
    \norm{ D^j \mathcal{R}_{I}(\Phi_0 + w)[v_1, \ldots, v_j]}_{H^{k}}
    & \le & K (1 + ||w||_{H^k}^{k} ) \mathcal{P}^{(j)}_k[v]
+ K ||w||_{H^{k+1}}  \mathcal{P}^{(j)}_{k_*}[v]
\\[0.2cm]
& & \qquad
+ K \mathcal{P}^{(j)}_{k_*, k + 1} [v]. 
\\[0.2cm]
\end{array}
\end{equation}
\end{cor}
\begin{proof}
Upon inspecting \eqref{eq:sm:reminder:r:i}, these statements follow
readily from Lemmas \ref{lem:nl:bnds:f} and \ref{lem:sm:f:l2}.
\end{proof}

\begin{cor}
Suppose that \textnormal{(HNL)} and \textnormal{(HTw)} hold.
Then there exists $K > 0$ so that for any
$w \in H^{k_*+1}$ and $v \in H^{k_*+1}$
we have the bound
\begin{equation}
\label{eq:nl:bnd:d:r:i:compact}
\begin{array}{lcl}
    \norm{ D \widetilde{\mathcal{R}}_{I}(\Phi_0 + w)[v]}_{H^{k_*}}
    & \le & K \big(1 + ||w||_{H^{k_*}}^{k_*} \big) \big( ||w||_{H^{k_*}}  ||v||_{H^{k_*+1}} + ||w||_{H^{k_*+1}} ||v||_{H^{k_*}} \big) .
\\[0.2cm]
\end{array}
\end{equation}
\end{cor}
\begin{proof}
As a consequence of
Lemma \ref{lem:sm:deriv:n:f} and the identity $\mathcal{N}_f(0) = 0$,
we note first that
\begin{equation}
    \norm{\mathcal{N}_f(w)}_{L^2} \le K \norm{w}_{H^{k_*}}^2  .
\end{equation}
In particular, a direct inspection of
\eqref{eq:sm:reminder:r:i} using the bounds in
Lemma \ref{lem:sm:deriv:n:f} leads to the
estimate
\begin{equation}
\begin{array}{lcl}
    \norm{ D \widetilde{\mathcal{R}}_{I}(\Phi_0 + w, 0)[v]}_{H^{k_*}}
    & \le & K (1 + ||w||_{H^{k_*}}^{k_*})||w||_{H^{k_*}}  ||v||_{H^{k_*}}
\\[0.2cm]
& & \qquad
+ K  ||w||_{H^{k_*}}^2 ||v||_{H^{k_*+1}}
+ K ||w||_{H^{k_*}} ||v||_{H^{k_*}} (1 + ||w||_{H^{k_*+1}}),
\end{array}
\end{equation}
which can be absorbed by \eqref{eq:nl:bnd:d:r:i:compact}.
\end{proof}

\subsection{Bounds for \texorpdfstring{$\mathcal{R}_{II}$}{RII}}

We now turn to the nonlinearity
\begin{equation}
\label{eq:sm:repr:r:ii}
    \mathcal{R}_{II}(v)
     =  \partial_x \mathcal{K}_C(\Phi_0 +v , 0)
    + \chi_l (\Phi_0 + v) \langle  \mathcal{K}_C(\Phi_0 +v , 0), \psi_{\mathrm{tw}}' \rangle \partial_x [\Phi_0 + v],
\end{equation}
in which $\mathcal{K}_C$ is given by
\begin{equation}
\label{eq:sm:def:k:c}
    \mathcal{K}_C(u, \gamma) = \chi_h(u, \gamma) g(u) \widetilde{\mathcal{K}}_C(u,\gamma);
\end{equation}
see \eqref{eq:list:def:r:i:i:ups}
and \eqref{eq:list:def:wt:k:c}.
The key towards
obtaining $H^{k+1}$-bounds for this combination is to use the splitting
    \begin{equation}
    \label{eq:sm:splt:k:c}
     \partial^\beta  \mathcal{K}(u) = \chi_h(u) [\partial^\beta g(u)] [ \widetilde{\mathcal{K}}_C(u) ]
     + [\chi_h(u) g(u) ] [ \partial^\beta \widetilde{\mathcal{K}}_C(u) ]
    \end{equation}
for any multi-index $\beta \in \mathbb Z^{d}_{\ge 0}$ with $|\beta|=1$.
Indeed, one can consider $g(u)$ and $\partial^\beta g(u)$ as elements in $HS(L^2_Q;H^k)$ via the bound \eqref{eq:mr:hs:bnd:z},
while both $\widetilde{\mathcal{K}}_C(u)$ and
$ \partial^\beta \widetilde{\mathcal{K}}_C(u) $ can be considered
as elements in $L^2_Q$ in view of Lemma \ref{lem:sm:b:nu:kct}.

\begin{lem}
\label{lem:sm:k:c:h:kp1}
Suppose that \textnormal{(HNL)}, \textnormal{(HTw)} and \textnormal{(Hq)} hold.
Pick 
an integer $0 \le j \le r$
and an integer  $k_* \le k \le  k_j$.
Then we have
\begin{equation}
    w \mapsto \mathcal{K}_C(\Phi_0 + w, 0) \in C^j(H^{k+1};H^{k+1}).
\end{equation}
In  addition, there is a constant $K > 0$ so that
for any $w \in H^{k+1}$ and any tuple
$(v_1, \ldots v_j) \in (H^{k+1})^{j}$
we have the bounds
\begin{equation}
\begin{array}{lcl}
    \norm{ D^j \mathcal{K}_C(\Phi_0 + w, 0)[v_1, \ldots, v_j]}_{H^{k+1}}
    & \le & K (1 + ||w||_{H^k}^{k}||w||_{H^{k+1}} ) \mathcal{P}^{(j)}_k[v] 
\\[0.2cm]
& & \qquad
+ K (1 + ||w||_{H^k}^{k} ) \mathcal{P}^{(j)}_{k,k+1}[v] ,
\\[0.2cm]
    \norm{ D^j \mathcal{K}_C(\Phi_0 + w, 0)[v_1, \ldots, v_j]}_{L^2}
    & \le & K   \mathcal{P}^{(j)}_{k_*}[v].
\end{array}
\end{equation}
\end{lem}
\begin{proof}
These bounds follow from the definition
\eqref{eq:sm:def:k:c}
and the splitting \eqref{eq:sm:splt:k:c},
using \eqref{eq:mr:hs:bnd:z} together with Lemmas
\ref{lem:sm:cutoffs},
    \ref{lem:sm:g:l2:h1} and
    \ref{lem:sm:b:nu:kct}.
Note in particular that the $L^2$-bound requires the
uniform bound on the product $\chi_h(\Phi_0 + w, 0) g(\Phi_0 + w)$
provided in \eqref{eq:sm:g:l2:h1:bnd}.
\end{proof}

\begin{cor}
\label{cor:sm:r:ii}
Suppose that \textnormal{(HNL)}, \textnormal{(HTw)} and \textnormal{(Hq)} hold.
Pick 
an integer
$0 \le j \le r$ and an integer $k_* \le k \le k_j$.
Then we have
\begin{equation}
\begin{array}{lcl}
w \mapsto \widetilde{\mathcal{R}}_{II}(\Phi_0 + w)
     &\in & C^j(H^{k+1};H^{k}) .
\\[0.2cm]
\end{array}
\end{equation}
In  addition, there is a constant $K > 0$ so that
for any $w \in H^{k+1}$ and any tuple $(v_1, \ldots v_j) \in (H^{k+1})^j$
we have the bound
\begin{equation}
\label{eq:nl:bnd:d:j:r:ii}
\begin{array}{lcl}
    \norm{ D^j \widetilde{\mathcal{R}}_{II}(\Phi_0 + w, 0)[v_1, \ldots, v_j]}_{H^{k}}
    & \le & K (1 + ||w||_{H^k}^{k}  ||w||_{H^{k+1}} )
      \mathcal{P}^{(j)}_k[v]
\\[0.2cm]
& & \qquad + ||w||_{H^{k+1}}       \mathcal{P}^{(j)}_{k_*}[v]
+ K (1 + ||w||_{H^k}^{k} ) \mathcal{P}^{(j)}_{k,k+1}[v].  
\end{array}
\end{equation}
\end{cor}
\begin{proof}
    This follows from Lemma \ref{lem:sm:k:c:h:kp1}
    and inspection of the product structure
    in \eqref{eq:sm:repr:r:ii}.
\end{proof}

\subsection{Bounds for \texorpdfstring{$\Upsilon$}{Y}}

In this part we consider the nonlinearity
\begin{equation}
\Upsilon = \Upsilon_I + \Upsilon_{II}
\end{equation}
defined in \eqref{eq:list:def:ups:i:ii}
and \eqref{eq:list:def:ups}.
In particular, we have
\begin{equation}
\label{eq:sm:def:ups:i}
        \Upsilon_I(v)  = \widetilde{\nu}(\Phi_0 + v, 0) \partial^2_x [\Phi_0 +v],
\end{equation}
while $\Upsilon_{II}$ can be written in the form
\begin{equation}
\label{eq:sm:def:ups:ii}
    \Upsilon_{II}(v) =   -\chi_l (\Phi_0 + v, 0) \tilde{\nu}(\Phi_0 + v, 0)\langle \Phi_0 + v, \psi_{\mathrm{tw}}'' \rangle_{L^2} \partial_x [\Phi_0 + v] .
\end{equation}

\begin{cor}
\label{cor:sm:bnd:on:ups:i}
Suppose that \textnormal{(HNL)} and \textnormal{(HTw)} hold.
Pick 
an integer $k_* \le k \le k_* + r$.
Then we have
\begin{equation}
  w \mapsto   \Upsilon_I(w)  \in C^r( H^{k+2} ; H^k ) .
\end{equation}
In addition, there exists $K >0$ so that for all $0 \le j \le r$,
any $w \in H^{k+2}$ and any tuple $(v_1, \ldots, v_j) \in (H^{k+2} )^j$
we have the bounds
\begin{equation}
\begin{array}{lcl}
    \norm{ D^j \Upsilon_I( w)
    [v_1, \ldots v_j] }_{H^k}
    & \le &
    K \big[ 1 + \norm{w}_{H^{k+2}}]  \mathcal{P}^{(j)}_{k_*}[v]
     + K \mathcal{P}^{(j)}_{k_*,k+2}[v] ,
    \\[0.2cm]
    \abs{ \langle  D^j \Upsilon_I( w)
    [v_1, \ldots v_j] , \psi_{\rm tw} \rangle_{L^2} }
    & \le &
       K \big[ 1 + \norm{w}_{L^2}] \mathcal{P}^{(j)}_{k_*}[v] .
    \\[0.2cm]
\end{array}
\end{equation}
\end{cor}
\begin{proof}
Recalling that
 $\Phi_0'' \in H^{k_* + r + 1}$ and hence also in $H^{k_* + r}$,
 this follows from the product structure of \eqref{eq:sm:def:ups:i} and the properties of $\tilde{\nu}$ outlined in Lemma \ref{lem:sm:b:nu:kct}.
\end{proof}

\begin{cor}
\label{cor:sm:bnd:on:ups:ii}
Suppose that \textnormal{(HNL)} and \textnormal{(HTw)} hold.
Pick 
an integer $k_* \le k \le k_* + r$.
Then we have
\begin{equation}
  w \mapsto   \Upsilon_{II}(w)  \in C^r( H^{k+1} ; H^k ).
\end{equation}
In addition, there exists $K > 0$ so that for all $0 \le j \le r$,
any $w \in H^{k+1}$ and any tuple $(v_1, \ldots , v_j) \in (H^{k+1})^j$
we have the bound
\begin{equation}
\label{eq:nl:bnd:dj:ups:ii}
\begin{array}{lcl}
    \norm{ D^j \Upsilon_{II}( w)
    [v_1, \ldots v_j] }_{H^k}
    & \le &
    K \big( 1 + \norm{w}_{L^2} \big) \big(1 +  \norm{w}_{H^{k+1}})
    \mathcal{P}^{(j)}_{k_*}[v]
    \\[0.2cm]
    & & \qquad
        + K \big[1 + \norm{w}_{L^2} \big] \mathcal{P}^{(j)}_{k_*,k+1}[v].
\end{array}
\end{equation}
\end{cor}
\begin{proof}
This follows from the product structure of \eqref{eq:sm:def:ups:ii}
together with the bounds in Corollary \ref{cor:sm:bnd:on:ups:i}.
\end{proof}

\begin{cor}
\label{cor:sm:ups}
Suppose that \textnormal{(HNL)} and \textnormal{(HTw)} hold.
Then for any $k_* \le k \le k_* + r$
we have
\begin{equation}
  w \mapsto   \Upsilon(w)  \in C^r( H^{k+2} ; H^k ).
\end{equation}
In addition, there exists $K >0$ so that for all $0 \le j \le r$,
any $w \in H^{k+2}$ and any tuple $(v_1, \ldots, v_j) \in (H^{k+2} )^j$
we have the bound
\begin{equation}
\label{eq:sm:bnd:derivs:ups}
\begin{array}{lcl}
    \norm{ D^j \Upsilon( w)
    [v_1, \ldots v_j] }_{H^k}
    & \le &
    K \big[ 1 + \norm{w}_{H^{k+2}} + \norm{w}_{L^2} \norm{w}_{H^{k+1}} \big]
    \mathcal{P}^{(j)}_{k_*}[v]
    \\[0.2cm]
    & & \qquad
        + K \mathcal{P}^{(j)}_{k_*,k+2}[v]
        + K \big[1 + \norm{w}_{L^2} \big] \mathcal{P}^{(j)}_{k_*,k+1}[v] .
\end{array}
\end{equation}
\end{cor}
\begin{proof}
These statements follow by combining Corollaries \ref{cor:sm:bnd:on:ups:i} and \ref{cor:sm:bnd:on:ups:ii}.
\end{proof}

\subsection{Bounds for \texorpdfstring{$\mathcal{S}$}{S}}

Our final result here concerns the nonlinearity $\mathcal{S}$ given by
\begin{equation}
\label{eq:sm:def:S}
\begin{array}{lcl}
\mathcal S(v)[\xi]
&=&g(\Phi_0+v)[\xi]+\partial_x(\Phi_0+v)  b(\Phi_0+v,0)[\xi];
\end{array}
\end{equation}
see \eqref{eq:app:def:R:sigma:S}.
Using our previous results for $g$ and $b$ this nonlinearity can be readily analyzed.

\begin{lem}
\label{lem:sm:s}
Suppose that \textnormal{(HNL)}, \textnormal{(HTw)} and \textnormal{(Hq)} hold.
Pick 
an integer $0 \le j \le r$
and an integer $k_* \le k \le k_j$.
Then we have
\begin{equation}
    w \mapsto \mathcal{S}(\Phi_0 + w) \in C^j\big(H^{k+1};HS(L^2_Q;H^{k}) \big).
\end{equation}
In  addition, there is a constant $K > 0$ so that
for any $w \in H^{k+1}$ and any tuple $(v_1, \ldots v_j) \in (H^{k+1})^j$
we have the bound
\begin{equation}
\label{eq:nl:bnd:d:j:mathcal:s}
\begin{array}{lcl}
    \norm{ D^j \mathcal{S}(\Phi_0 + w, 0)[v_1, \ldots, v_j]}_{HS(L^2_Q;H^{k})}
    & \le & K (1 + ||w||_{H^k}^{k}  ) \mathcal{P}^{(j)}_k[v]
\\[0.2cm]
& & \qquad
+ K ||w||_{H^{k+1}} \mathcal{P}^{(j)}_{k_*}[v]
+ K \mathcal{P}^{(j)}_{k_*, k+1}[v] .
\end{array}
\end{equation}
\end{lem}
\begin{proof}
   The statements for the first term $g$ follow from \eqref{eq:mr:hs:bnd:z}
   together with Lemma \ref{lem:sm:bnds:on:g}.
   The desired properties for the product term involving $b$ follow from Lemma \ref{lem:sm:b:nu:kct}.
\end{proof}

\section{Taylor expansion}
\label{sec:tay}

In this section we study the Taylor expansion terms
defined in \eqref{eq:mr:def:w:j:all}.
In particular, we establish Proposition \ref{prop:mr:tay}. A key role is reserved for the following working hypothesis, which is
defined in terms of an integer  $2 \le j_* \le r$.
\begin{itemize}
    \item[(WH)]{
    There exists $K > 0$ so that
for every $1 \le j < j_*$, any real $p \ge 1$, any $\sigma \ge 0$, any $\delta \ge 0$ and $T \ge 2$ we have the bound
\begin{equation}
\label{eq:tay:work:hyp}
E \sup_{0 \le t \le T} \norm{Y_j[\sigma,\delta](t)}_{j}^{2p}
    \le  \big( [\sigma^{2}p]^{jp} +  [\sigma^2 \ln T + \delta^2]^{jp}
     \big)
    K^{2p} .
\end{equation}
}
\end{itemize}

\begin{prop}[{see {\S}\ref{subsec:tay:size}}]
\label{prop:tay:wh:holds}
Suppose that \textnormal{(HNL)}, \textnormal{(HTw)}, \textnormal{(H$V_*$)} and \textnormal{(Hq)} all hold.
Then \textnormal{(WH)} holds for $j_* = r$.
\end{prop}

We are also interested in the remainder that arises
when substituting the full Taylor approximation
\begin{equation}
    Y_{\rm tay}[\sigma, \delta] = Y_1[\sigma, \delta] + \ldots + Y_{r-1}[\sigma, \delta]
\end{equation}
into the full nonlinearities
$\mathcal{R}_I$, $\mathcal{R}_{II}$, $\Upsilon$ and $\mathcal{S}$.
To this end,
we recall
\eqref{eq:mr:def:n:b:j:no:tilde}
and
define the remainder functions
\begin{equation}
\label{eq:tay:def:rem:fncs}
\begin{array}{lcl}
    N_{\rm rem}[\sigma, \delta]
    &=& \mathcal{R}_I(Y_{\rm tay}[\sigma,\delta]) +
    \sigma^2 \mathcal{R}_{II}(Y_{\rm tay}[\sigma,\delta]) + \sigma^2 \Upsilon(Y_{\rm tay}[\sigma,\delta])
    - \sum_{j=1}^{r-1} N_j[\sigma,\delta] ,
\\[0.2cm]
B_{\rm rem}[\sigma,\delta] & = &
\sigma \mathcal{S}\big(Y_{\rm tay}[\sigma,\delta] \big)
    - \sum_{j=1}^{r-1} B_j[\sigma,\delta].
\end{array}
\end{equation}
We also introduce the stopping time
\begin{equation}
    t_{\rm tay}(\eta) =
    \mathrm{inf} \{ t \ge 0: \norm{Y_1}_{1}^2 + \ldots + \norm{Y_{r-1}}_{r-1}^2 \ge \eta \},
\end{equation}
writing $t_{\rm tay}(\eta) = \infty$ if the set is empty.

\begin{prop}[{see {\S}\ref{subsec:tay:rem}}]
\label{prop:tay:bnd:rem}
Suppose that \textnormal{(HNL)}, \textnormal{(HTw)}, \textnormal{(H$V_*$)} and \textnormal{(Hq)} hold.
    Then there exists $K > 0$ so that for all real $p \ge 1$,
    all $0 \le \sigma  \le 1$, all $\delta \ge 0$ and  all $T \ge 2$
    we have the bounds
    \begin{equation}
    \label{eq:tay:bnds:n:b:rem}
\begin{array}{lcl}
\mathbb E \sup_{0 \le t \le T\wedge t_{\rm tay}(1) } \norm{N_{\rm rem}[\sigma,\delta](t)}_{H^{k_*}}^{2p}  
 & \le & K^{2 p }\big( [\sigma^2 p]^{  p r }
+ [\sigma^2 \ln T + \delta^2]^{  p r  }    \big) ,
        \\[0.2cm]
\mathbb E \sup_{0 \le t \le T\wedge t_{\rm tay}(1) } \norm{B_{\rm rem}[\sigma,\delta](t)}_{HS(L^2_Q;H^{k_*})}^{2 p} 
 & \le & \sigma^{2 p} K^{2 p} \big( [\sigma^2 p]^{  p (r-1) }
+ [\sigma^2 \ln T + \delta^2]^{  p (r-1)  }   \big) .
\end{array}
\end{equation}
\end{prop}

\subsection{Bounds on \texorpdfstring{$Y_j$}{Yj}}
\label{subsec:tay:size}

Our main task here is to inductively establish the working hypothesis (WH).
To this end, we recall that $Y_1$ is given by
\begin{equation}
\label{eq:tay:def:w1:bis}
    Y_1[\sigma, \delta](t) = \delta \mathrm{exp}[ (\mathcal{L}_{\rm tw} + \Delta_{x_\perp}) t ] V_*  +
    \sigma \int_0^t \mathrm{exp}[ (\mathcal{L}_{\rm tw} + \Delta_{x_\perp}) (t-s) ] \mathcal{S}(0) \, \mathrm d W^Q_s,
\end{equation}
in which $\norm{V_*}_{H^{k_1}} = 1$ and $P^\perp V_* = V_*$, together with
\begin{equation}
\mathcal{S}(0) = g(\Phi_0)  + b(\Phi_0, 0) \partial_x \Phi_0 \in HS(L^2_Q;H^{k_1}) ,
\end{equation}
which implies that also $P^\perp \mathcal{S}(0) = \mathcal{S}(0)$.
Upon taking $\nu = 1$, we note that
the evolution family $E$ discussed in {\S}\ref{sec:conv} can be represented as
\begin{equation}
\label{eq:tay:id:flow:E}
    E(t, s) = \mathrm{exp}\big[ (\mathcal{L}_{\rm tw} + \Delta_{x_\perp}) (t-s) \big].
\end{equation}
In particular, the estimate \eqref{eq:conv:bnds:e} guarantees that
the first term in \eqref{eq:tay:def:w1:bis} satisfies the bound
\eqref{eq:tay:work:hyp}. The convolution term can also be seen to satisfy
\eqref{eq:conv:bnds:e}
by applying Proposition \ref{prp:cnv:hb} and taking $n=1$, $\Theta_2 = 0$
and $\Theta_1 = \norm{ \mathcal{S}(0) }_{H^{k_1}}$.

In particular, (WH) holds for $j_* = 2$.
Before turning to the induction step, we first record a useful consequence
of this hypothesis.

\begin{lem}
\label{lem:tay:first:est:w:ind:hyp}
Assume \textnormal{(WH)}. Then there exist $K > 0$ so that for any
$1 \le j < j_*$, any $\ell \ge j$, any (real)
$\tilde{p} \ge 1$, any $\sigma \ge 0$, any $\delta \ge 0$
and $T \ge 2$ we have the bound
\begin{equation}
\label{eq:tay:bnd:w:j:s:ell:p:div:j}
    \mathbb E \sup_{0 \le t \le T}
    \norm{Y_{j}[\sigma,\delta](t)}_{j}^{2 \ell \tilde{p}/j}
    \le
    \big( [\sigma^2 \tilde{p}]^{ \tilde{p} \ell }
+ [\sigma^2 \ln T + \delta^2]^{\tilde{p} \ell  }   \big)
      \ell^{\tilde{p} \ell} K^{2 \ell \tilde{p}}  ,
\end{equation}
together with
\begin{equation}
\label{eq:tay:bnd:w:j:s:ell:p}
    \mathbb E \sup_{0 \le t \le T}
    \norm{Y_{j}[\sigma,\delta](t)}_{j}^{2 \tilde{p} \ell }
    \le
    \big( [\sigma^2 \tilde{p}]^{  j \tilde{p} \ell }
+ [\sigma^2 \ln T + \delta^2]^{ j \tilde{p} \ell  }   \big)
       \ell^{j \tilde{p} \ell} K^{2 \ell \tilde{p}}.
\end{equation}
\end{lem}
\begin{proof}
After adjusting the constant $K > 0$,
the first bound follows from (WH) by taking
    $p = \tilde{p} \ell / j \ge \tilde{p} \ge 1$,
while the second bound follows by picking
    $p = \ell \tilde{p} \ge 1$.
\end{proof}

In order to proceed, we need to consider integers $2 \le j \le r-1$
and analyze deterministic convolutions
of the semigroup \eqref{eq:tay:id:flow:E}
with respect to the functions
\begin{equation}
\begin{array}{lcl}
N_{j;I}[\sigma,\delta]
& = & \widetilde{N}_{j;I}\big[ Y_1[\sigma,\delta], \ldots, Y_{j-1}[\sigma,\delta] \big],
\\[0.2cm]
N_{j;II}[\sigma,\delta]
& = & \widetilde{N}_{j;II}\big[ Y_1[\sigma,\delta], \ldots, Y_{j-2}[\sigma,\delta] \big],
\\[0.2cm]
N_{j;III}[\sigma,\delta]
& = & \widetilde{N}_{j;III}\big[ Y_1[\sigma,\delta], \ldots, Y_{j-2}[\sigma,\delta] \big],
\end{array}
\end{equation}
together with stochastic convolutions with respect to the functions
\begin{equation}
    B_j[\sigma,\delta] = \tilde{B}_j\big[ Y_1[\sigma,\delta], \ldots, Y_{j-1}[\sigma,\delta] \big].
\end{equation}
These involve the expressions
\begin{equation}
\label{eq:tay:def:n:j}
\begin{array}{lcl}
    \widetilde{N}_{j;I}[y_1, \ldots, y_{j-1}] & = & \sum_{\ell=2}^j \sum_{i_1 + \ldots + i_{\ell} = j} \frac{1}{\ell !}D^{\ell} \mathcal{R}_I(0) [ y_{i_1}, \ldots , y_{i_\ell}] ,
\\[0.2cm]
  \widetilde{N}_{j;II}[y_1, \ldots, y_{j-2}] & = & \sigma^2  \sum_{\ell = 0}^{j-2} \sum_{i_1 + \ldots + i_{\ell} = j-2} \frac{1}{\ell !} D^{\ell} \mathcal{R}_{II}(0) [ y_{i_1}, \ldots , y_{i_\ell}] ,
\\[0.2cm]
\widetilde{N}_{j;III}[y_1, \ldots, y_{j-2}] & = & \sigma^2  \sum_{\ell = 0}^{j-2} \sum_{i_1 + \ldots + i_{\ell} = j-2} \frac{1}{\ell !} D^{\ell} \Upsilon(0) [ y_{i_1}, \ldots , y_{i_\ell}],
\end{array}
\end{equation}
together with
\begin{equation}
\label{eq:tay:def:b:j}
\begin{array}{lcl}
    \widetilde{B}_{j}[y_1, \ldots, y_{j-1}] & = & \sigma  \sum_{\ell = 0}^{j-1} \sum_{i_1 + \ldots + i_{\ell} = j-1} \frac{1}{\ell !}D^{\ell} \mathcal{S}(0) [ y_{i_1}, \ldots , y_{i_\ell}],
\end{array}
\end{equation}
where in each term we have $i_{\ell'} \ge 1$ for $\ell' \in \{ 1,\ldots,\ell\}$; see \eqref{eq:mr:def:n:j} and \eqref{eq:mr:def:b:j}.
To assist our computations, we first consider an arbitrary tuple $\{i_1, \ldots, i_{\ell} \} \in \{1 , \ldots, r-1 \}^{\ell}$ with $\ell \ge 1$ and write $i_{\rm tot} = i_1 + \ldots + i_{\ell}$.
The weighted arithmetic-geometric-mean 
inequality now yields the useful bound
\begin{equation}
\label{eq:tay:prod:bnd:w}
    \norm{y_{i_1}}_{i_1} \cdots \norm{y_{i_{\ell}}}_{i_\ell}
    \le \frac{1}{i_{\rm tot}}\big[ i_1 \norm{y_{i_1}}_{i_1}^{i_{\rm tot}/i_1}
    + \ldots +
    i_\ell \norm{y_{i_\ell}}_{i_{\ell}}^{i_{\rm tot}/i_\ell}
    \big].
\end{equation}

\begin{lem}
\label{lem:tay:bnd:n:j:i}
Suppose that \textnormal{(HNL)}, \textnormal{(HTw)} and \textnormal{(Hq)} hold.
Then there exists $K_{I} > 0$ so that
for any $2 \le j \le r - 1$
and any tuple $(y_1 , \ldots y_{j-1}) \in H_1 \times \ldots \times H_{j-1}$
we have the bound
\begin{equation}
\begin{array}{lcl}
    \norm{\widetilde{N}_{j;I}[y_1, \ldots , y_{j-1}]}_{j}
     & \le & K_I \sum_{i=1}^{j-1} \norm{y_i}_{i}^{j/i} .
\\[0.2cm]
\end{array}
\end{equation}
\end{lem}
\begin{proof}
Consider one of the terms in the sum
\eqref{eq:tay:def:n:j} for $\widetilde{N}_{j;I}$ and note that $\ell > 0$.
Observe that the term is well-defined in $H^{k_j}$
on account of Corollary \ref{cor:sm:r:i}
and the fact that for every $1 \le \ell' \le \ell$
we have $i_\ell' \le j-1$, which means
$k_{i_{\ell'}} \ge k_j + 1$. In particular, we also have
\begin{equation}
    \norm{ y_{i_{\ell'}} }_{H^{k_j + 1}} \le \norm{y_{i_{\ell'}} }_{i_{\ell'}}.
\end{equation}
This allows the desired estimate to be read off from
\eqref{eq:nl:bnd:d:j:r:i} and \eqref{eq:tay:prod:bnd:w}.
\end{proof}

\begin{cor}
\label{cor:tay:bnd:sup:n:i}
Suppose that \textnormal{(HNL)}, \textnormal{(HTw)} and \textnormal{(Hq)} hold,
together with \textnormal{(WH)} for some $2 \le j_* \le r-1$. Then
there exists $K > 0$ so that
for any real $p \ge 1$, any $\sigma \ge 0$, any $\delta > 0$
and any $T \ge 3$ we have the bound
\begin{equation}
\label{eq:tay:bnd:sup:n:j:i}
\begin{array}{lcl}
\mathbb E \sup_{0 \le t \le T } \norm{N_{j_*;I}[\sigma,\delta](t)}_{j_*}^{2p}
 & \le &  \big( [\sigma^2 p ]^{ p j_* }
+ [\sigma^2 \ln T + \delta^2]^{  p j_*  }  \big) K^{2 p }.
\\[0.2cm]
\end{array}
\end{equation}
\end{cor}
\begin{proof}
We first note that Lemma \ref{lem:tay:bnd:n:j:i}
implies that
\begin{equation}
\begin{array}{lcl}
    \norm{N_{j_*;I}[\sigma,\delta](t)}_{j_*}^{2p}
    & \le & K_I^{2p} j_*^{2p} \sum_{i=1}^{j_*-1}
    \norm{Y_i[\sigma,\delta](t)}_{i}^{2p j_*/i}.
\\[0.2cm]
\end{array}
\end{equation}
Applying \eqref{eq:tay:bnd:w:j:s:ell:p:div:j} with $\ell= j_*$,
we hence conclude
\begin{equation}
\begin{array}{lcl}
\mathbb E \sup_{0 \le t \le T}   \norm{N_{j_*;I}[\sigma,\delta](t)}_{k_{j_*}}^{2p}
& \le &
K_I^{2p} j_*^{2p} \sum_{i=1}^{j_*-1}
    \big( [\sigma^2 p]^{p j_*} + [\sigma^2 \ln T + \delta^2]^{p j_*}
    \big) j_*^{p j_*} K^{2 j_* p},
\end{array}
\end{equation}
which fits the stated bound.
\end{proof}

\begin{lem}
\label{lem:tay:bnd:n:j:ii:ii}
Suppose that \textnormal{(HNL)}, \textnormal{(HTw)} and \textnormal{(Hq)} hold.
Pick $k_* > d/2$. Then there exists $K_{II;III} > 0$ so that
for any $2 \le j \le r - 1$,
any tuple $(y_1 , \ldots y_{j-2}) \in H_1 \times \ldots \times H_{j-2}$
and any $\sigma \ge 0$ we have the bound
\begin{equation}
\begin{array}{lcl}
\norm{\widetilde{N}_{j;II}[y_1, \ldots, y_{j-2}]}_{j} +  \norm{\widetilde{N}_{j;III}[y_1, \ldots , y_{j-2}]}_{j}
 & \le & K_{II;III} \sigma^2 \big[ \mathbf{1}_{j=2} + \sum_{i=1}^{j-2} \norm{y_i}_{i}^{(j-2)/i} \big] .
 \\[0.2cm]
\end{array}
\end{equation}
\end{lem}
\begin{proof}
Consider one of the terms in the sum
\eqref{eq:tay:def:n:j} for $\widetilde{N}_{j;II}$ or $\widetilde{N}_{j;III}$.
We note that $\ell = 0$
is only possible when $j=2$, which is covered
by the $\mathbf{1}_{j=2}$ term in the bounds.
If $\ell > 0$, the term is well-defined in $H^{k_j}$
on account of Corollary \ref{cor:sm:r:ii} or
Corollary \ref{cor:sm:ups}
and the fact that for every $1 \le \ell' \le \ell$
we have $i_{\ell'} \le j-2$, which implies
$k_{i_{\ell'}} \ge k_j + 2$. Arguing as in the proof of Lemma
\ref{lem:tay:bnd:n:j:i}, the desired bound
now follows from \eqref{eq:tay:prod:bnd:w},
\eqref{eq:nl:bnd:d:j:r:ii}
and \eqref{eq:sm:bnd:derivs:ups}.
\end{proof}

\begin{cor}
\label{cor:tay:bnd:sup:n:ii:iii}
Suppose that \textnormal{(HNL)}, \textnormal{(HTw)} and \textnormal{(Hq)} hold,
together with \textnormal{(WH)} for some $2 \le j_* \le r-1$. Then
there exists $K > 0$ so that
for any real $p \ge 1$, any $\sigma \ge 0$, any $\delta \ge 0$
and any $T \ge 2$ we have the bound
\begin{equation}
\label{eq:tay:exp:bnd:n:j:ii:iii}
\begin{array}{lcl}
\mathbb E \sup_{0 \le t \le T }\big[ \norm{N_{j_*;II}[\sigma,\delta](t)}_{{j_*}}^{2p}
+ \norm{N_{j_*;III}[\sigma,\delta](t)}_{{j_*}}^{2p}\big]
 & \le &  \big( [\sigma^2 p ]^{ p j_* }
+ [\sigma^2 \ln T + \delta^2]^{  p j_*  }    \big) K^{2 p }.
\\[0.2cm]
\end{array}
\end{equation}
\end{cor}
\begin{proof}
For $j_* = 2$, we note that Lemma
\ref{lem:tay:bnd:n:j:ii:ii} implies
the pathwise bound
\begin{equation}
    \norm{N_{2;II}[\sigma,\delta](t)}_{2}^{2p} \le K^{2p} \sigma^{4 p},
\end{equation}
which can be absorbed by \eqref{eq:tay:exp:bnd:n:j:ii:iii} since $p \ge 1$. For $j_* \ge 3$, Lemma \ref{lem:tay:bnd:n:j:ii:ii} implies
\begin{equation}
\begin{array}{lcl}
    \norm{N_{j_*;II}[\sigma,\delta](t)}_{j_*}^{2p}
    & \le & K_{II;III}^{2p} \sigma^{4p} j_*^{2p} \sum_{i=1}^{j_*-2}
    \norm{Y_i[\sigma,\delta](t)}_{i}^{2p (j_*-2)/i}.
\\[0.2cm]
\end{array}
\end{equation}
Applying \eqref{eq:tay:bnd:w:j:s:ell:p:div:j} with $\ell= j_*-2$,
this yields
\begin{equation}
    \mathbb E \sup_{0 \le t \le T}  \norm{N_{j_*;II}[\sigma,\delta](t)}_{j_*}^{2p}
 \le
K_{II;III}^{2p} \sigma^{4p} j_*^{2p+1}
    \big( [\sigma^2 p]^{p (j_*-2)} + [\sigma^2 \ln T + \delta^2]^{p (j_*-2)}
     \big) (j_*-2)^{p (j_*-2)} K^{2 j_* p}.
\end{equation}
This can be absorbed in \eqref{eq:tay:exp:bnd:n:j:ii:iii}
after using Young's inequality to find
\begin{equation}
\label{eq:tay:yng:ineq}
    \sigma^{4 p} [\sigma^2 \ln T + \delta^2]^{p(j_*-2)} \le \sigma^{2p j_*}\frac{2}{j_*}+ [\sigma^2 \ln T + \delta^2]^{p j_*} \frac{j_*-2}{j_*}
\end{equation}
and noting that $p \ge 1$. The same estimates hold for $N_{j_*;III}$.
\end{proof}

\begin{lem}
\label{lem:tay:bnd:b:j}
Suppose that \textnormal{(HNL)}, \textnormal{(HTw)} and \textnormal{(Hq)} hold.
Then there exists $K >0$
so that for any $2 \le j \le r-1$,
any tuple $(y_1 , \ldots y_{j-1}) \in H_1 \times \ldots \times H_{j-1}$
and any $\sigma \ge 0$ we have
the bound
\begin{equation}
\begin{array}{lcl}
    \norm{\widetilde{B}_{j}[y_1, \ldots , y_{j-1}]}_{HS(L^2_Q; H_j)}
 & \le & K_B \sigma  
 \sum_{i=1}^{j-1} \norm{y_i}_{i}^{(j-1)/i} . 
 \\[0.2cm]
\end{array}
\end{equation}
\end{lem}
\begin{proof}
Arguing as in the proof of Lemma \ref{lem:tay:bnd:n:j:i},
this bound follows by inspecting
\eqref{eq:tay:def:b:j}
and using
Lemma \ref{lem:sm:s} together with \eqref{eq:tay:prod:bnd:w}.
\end{proof}

\begin{cor}
\label{cor:tay:bnd:sup:b}
Suppose that \textnormal{(HNL)}, \textnormal{(HTw)} and \textnormal{(Hq)} hold,
together with  \textnormal{(WH)} for some $2 \le j_* \le r-1$
Then
there exists $K > 0$ so that
for any real $p \ge 1$, any $\sigma \ge 0$, any $\delta \ge 0$
and any $T \ge 2$ we have the bound
\begin{equation}
\label{eq:tay:bnd:sup:b:j}
\begin{array}{lcl}
\mathbb E \sup_{0 \le t \le T } \norm{B_{j_*}[\sigma,\delta](t)}_{HS(L^2_Q; H_{j_*})}^{2p}
 & \le & \sigma^{2 p} K^{2p} \big( [\sigma^2 p]^{  p (j_*-1) }
+ [\sigma^2 \ln T + \delta^2]^{  p (j_*-1)  }    \big).
\\[0.2cm]
\end{array}
\end{equation}
\end{cor}
\begin{proof}
Applying Lemma \ref{lem:tay:bnd:b:j} we obtain
the pathwise bound
\begin{equation}
\begin{array}{lcl}
\norm{B_{j_*}[\sigma,\delta](t)}_{HS(L^2_Q; H_{j_*})}^{2p}
    & \le & K_B^{2p} \sigma^{2p} j_*^{2p} \sum_{i=1}^{j_*-1}
    \norm{Y_i[\sigma,\delta](t)}_{i}^{2p (j_*-1)/i} .
\\[0.2cm]
\end{array}
\end{equation}
Using
\eqref{eq:tay:bnd:w:j:s:ell:p:div:j}
with $\ell = j_* - 1$
we hence conclude
\begin{equation}
\begin{array}{lcl}
\mathbb E \sup_{0 \le s \le T}
\norm{B_{j_*}[W]}_{k_{j_*}}^{2p}
& \le &
K_B^{2p} j_*^{2p+1} \sigma^{2p}
    \big( [\sigma^2 p]^{p (j_*-1)} + [\sigma^2 \ln T + \delta^2]^{p (j_*-1)}
     \big) j_*^{p (j_*-1)} K^{2 (j_*-1) p},
\end{array}
\end{equation}
which fits the stated bound.
\end{proof}

\begin{proof}[Proof of Proposition \ref{prop:tay:wh:holds}]
We proceed by induction, noting that the base case $j_* = 2$
is covered by the discussion following \eqref{eq:tay:def:w1:bis}.
Assuming that (WH) holds for some $2 \le j_* \le r - 1$,
we write
the definition \eqref{eq:mr:def:w:j:all}
in the form
\begin{equation}
    Y_{j_*} = \mathcal{E}^d_{N_{j_*:I}}
    + \mathcal{E}^d_{N_{j_*:II}}
    + \mathcal{E}^d_{N_{j_*;III}}
    + \mathcal{E}^s_{B_{j_*}}.
\end{equation}
Assuming without loss that $\sigma > 0$,
the estimates \eqref{eq:tay:bnd:sup:n:j:i}
and \eqref{eq:tay:exp:bnd:n:j:ii:iii}
imply that $N_{j_*;I}$, $N_{j_*;II}$ and $N_{j_*;III}$
all satisfy
(hN) with
\begin{equation}
\Theta_1 = \sigma^{j_*} K,
\qquad \Theta_2 = \sigma^{-2}[ \sigma^2 \ln T + \delta^2],
\qquad n = j_*.
\end{equation}
In particular, an application of Proposition \ref{prp:cnv:hn}
shows that the corresponding convolutions
satisfy the bounds in \eqref{eq:tay:work:hyp}.
In addition, the estimate \eqref{eq:tay:bnd:sup:b:j}
shows that the $B_{j_*} / \sigma$
satisfies (hB) with
\begin{equation}
        \Theta_1 = \sigma^{j_*-1} K,
         \qquad
         \Theta_2 = \sigma^{-2}[ \sigma^2 \ln T + \delta^2],
         \qquad n= j_*.
    \end{equation}
An application of Proposition \ref{prp:cnv:hb}
now yields
\begin{equation}
    \sigma^{-2p} \mathbb E \sup_{0 \le t \le T }
    \norm{\mathcal{E}^s_{B_{j_*}}}^{2p}_{k_{j_*}}
    \le K_{\rm gr}^{2p} (32 e j_*)^{2j_*p}
    \big( p^{j_* p} + [2 \ln T + \sigma^{-2} \delta^2 ]^{j_*p } \big)
    \sigma^{ 2p(j_* - 1)} K^{2p}.
\end{equation}
This can be rewritten as
\begin{equation}
     \mathbb E \sup_{0 \le t \le T }
    \norm{\mathcal{E}^s_{B_{j_*}}}^{2p}_{k_{j_*}}
    \le K_{\rm gr}^{2p} (32 e j_*)^{2j_*p} 2^{j_* p}
    \big( (\sigma^2 p)^{j_* p} + [\sigma^2 \ln T + \delta^2 ]^{j_*p } \big)
     K^{2p},
\end{equation}
establishing that $\mathcal{E}^s_B$ satisfies
the bounds in \eqref{eq:tay:work:hyp}. In particular,
(WH) is satisfied for $j_* + 1$.
\end{proof}

\begin{proof}[Proof of Proposition \ref{prop:mr:tay}]
Item (i) follows from \cite[Cor. 3.8]{bosch2024multidimensional},
item (ii) follows from the proof of \cite[Prop. 6.4]{bosch2024multidimensional}
and item (iii) follows from Proposition \ref{prop:tay:wh:holds}.
\end{proof}

\subsection{Remainder bounds}
\label{subsec:tay:rem}

We now focus on obtaining estimates for the remainder terms \eqref{eq:tay:def:rem:fncs}.
As a first observation,
we note that
\begin{equation}
\label{eq:tay:rem:bnds:k_j:big}
    k_j \ge k_* + 2 , \hbox{ for all } 1 \le j \le r-1.
\end{equation}
In particular, for any $\ell \ge 1$ we see that
\begin{equation}
\label{eq:tay:bnd:w:tay}
    \norm{Y_{\rm tay}}_{H^{k_* + 2}}^{\ell}  \le (r-1)^{\ell-1} \big[ \norm{Y_1}_{1}^{\ell} + \ldots + \norm{Y_{r-1}}_{r-1}^{\ell} \big].
\end{equation}
In addition, the properties of the stopping time
allow us to assume a uniform pathwise bound
\begin{equation}
\label{eq:tay:bnd:w:unif}
    \mathbf{1}_{0 \le t \le t_{\rm tay}(\eta)}
    \norm{Y_{\rm tay}[\sigma,\delta](t)}_{H^{k_* + 2}} \le K_{\rm tay} \sqrt{\eta}.
\end{equation}

Let us first write
\begin{equation}
    N_{\rm rem;I}[\sigma,\delta] = \mathcal{R}_I\big(Y_{\rm tay}[\sigma, \delta] \big) - \sum_{j=1}^{r-1} N_{j;I}\big[Y_1[\sigma, \delta] , \ldots , Y_{j-1}[\sigma, \delta] \big].
\end{equation}
The strategy is to make the decomposition
\begin{equation}
    N_{\rm rem;I}[\sigma,\delta] = N_{\rm rem;Ia}[\sigma,\delta]
    + N_{\rm rem;Ib}[\sigma,\delta]
\end{equation}
involving the multi-linear term
\begin{equation}
\label{eq:rem:def:n:rem:i:a}
\begin{array}{lcl}
    N_{\rm rem;Ia}[\sigma,\delta] & = &
    \sum_{\ell=2}^{r-1} \sum_{i_1 + \ldots + i_{\ell} \ge r} \frac{1}{\ell !} D^{\ell} \widetilde{\mathcal{R}}_I(0) \big[ Y_{i_1}[\sigma,\delta], \ldots , Y_{i_\ell}[\sigma,\delta]
    \big]
\end{array}
\end{equation}
with $1 \le i_{\ell'} \le r-1$ for $\ell'= 1 \ldots \ell$,
together with the nonlinear residual
\begin{equation}
\label{eq:rem:def:n:rem:i:b}
    N_{\rm rem;Ib}[\sigma,\delta] =\int_0^1 \cdots \int_0^1
    \mathcal{Q}_I( Y_{\rm tay}[\sigma,\delta];t_1, \ldots , t_{r-1},t_r )
    \, \mathrm dt_r \cdots \mathrm dt_1
\end{equation}
featuring the integrand
\begin{equation}
\mathcal{Q}_I ( Y_{\rm tay}; t_1, \ldots , t_{r-1}, t_r)
=
    D^r \mathcal{R}_I\big( t_1 \cdots t_r Y_{\rm tay}\big)\big[Y_{\rm tay}, t_1 Y_{\rm tay},  \ldots, t_1 \cdots t_{r-1} Y_{\rm tay}\big].
\end{equation}

\begin{lem}
\label{lem:rem:bnd:n:rem:i:a}
Suppose that \textnormal{(HNL)}, \textnormal{(HTw)}, \textnormal{(H$V_*$)} and \textnormal{(Hq)} hold.
Then there exists $K > 0$ so that
for all $t < t_{\rm tay}(1)$, all $\sigma \ge 0$ and all $\delta \ge 0$ we have the bound
\begin{equation}
\label{eq:tay:bnd:n:rem:i:a}
\begin{array}{lcl}
    \norm{N_{\rm rem;Ia}[\sigma,\delta](t)  }_{H^{k_*} }
     & \le & K \sum_{i=1}^{r-1} \norm{Y_i[\sigma,\delta](t)}_{i}^{r/i} .
\\[0.2cm]
\end{array}
\end{equation}
\end{lem}
\begin{proof}
Consider one of the terms in the sum
 \eqref{eq:rem:def:n:rem:i:a}
and write $i_{\rm tot} = i_1 + \ldots + i_{\ell}$.
We note that the term is well-defined in $H^{k_*}$
on account of Corollary \ref{cor:sm:r:i} and
\eqref{eq:tay:rem:bnds:k_j:big}.
Using \eqref{eq:nl:bnd:d:j:r:i} and \eqref{eq:tay:prod:bnd:w}
we obtain the bound
\begin{equation}
    \norm{D^{\ell} \mathcal{R}_{I}(0)[Y_{i_1}, \ldots , Y_{i_{\ell}}]}_{H^{k_*}}
    \le K \big[ \norm{Y_{i_1}}_{i_1}^{i_{\rm tot}/i_1}
    + \ldots + \norm{Y_{i_\ell}}_{i_\ell}^{i_{\rm tot}/i_{\ell}}
    \big].
\end{equation}
The desired estimate \eqref{eq:tay:bnd:n:rem:i:a} hence follows by noting that $i_{\rm tot} \ge r$
and using the a-priori bound $\norm{Y_{i_{\ell'}}}_{{i_{\ell'}}} \le 1$ that is available for $1 \le \ell' \le \ell$ as a consequence of the stopping time.
\end{proof}

\begin{lem}
\label{lem:rem:bnd:n:rem:i:b}
Suuppose that \textnormal{(HNL)}, \textnormal{(HTw)}, \textnormal{(H$V_*$)} and \textnormal{(Hq)} hold. Then there exists $K > 0$
so that for all $t \le t_{\rm tay}(1)$, all $\sigma \ge 0$ and all $\delta \ge 0$ we have the bound
\begin{equation}
\label{eq:tay:bnd:rem:i:b}
\begin{array}{lcl}
    \norm{N_{rem;Ib}[\sigma,\delta](t) }_{H^{k_*} }
     & \le & K \sum_{i=1}^{r-1} \norm{Y_i[\sigma,\delta](t)}_{i}^{r}.
\end{array}
\end{equation}
\end{lem}
\begin{proof}
Using
the bound \eqref{eq:nl:bnd:d:j:r:i}
together with \eqref{eq:tay:bnd:w:unif},
we see that there exists $C > 0$ so that
\begin{equation}
    \norm{\mathcal{Q}_I(Y_{\rm tay}[\sigma,\delta](t); t_1, \ldots , t_{r-1}}_{H^{k_*}}
    \le C \norm{Y_{\rm tay}[\sigma,\delta](t)}_{H^{k_* +1}}^r
\end{equation}
holds for $0 \le t \le t_{\rm tay}(1)$. The desired bound
\eqref{eq:tay:bnd:rem:i:b}
now follows from the representation
\eqref{eq:rem:def:n:rem:i:b}
and the estimate \eqref{eq:tay:bnd:w:tay}.
\end{proof}

\begin{cor}
\label{cor:tay:rem:full:n:i}
    Suppose that \textnormal{(HNL)}, \textnormal{(HTw)}, \textnormal{(H$V_*$)} and \textnormal{(Hq)} hold.
        Then there exists $K > 0$ so that for all (real) $p \ge 1$,
    all $\sigma  \ge 0$, all $\delta \ge 0$ and all $T \ge 2$
    we have the bound
    \begin{equation}
\begin{array}{lcl}
E \sup_{0 \le t \le T \wedge t_{\rm tay}(1) } \norm{N_{\rm rem;I}[\sigma,\delta](t)}_{H^{k_*}}^{2p} 
 & \le & K^{2 p} \big( [\sigma^2 p]^{  p r }
+ [\sigma^2 \ln T + \delta^2]^{  p r  }    \big) .
        \\[0.2cm]
\end{array}
\end{equation}
\end{cor}
\begin{proof}
We first note that $\norm{Y_j[\sigma,\delta](t)}^r_{j} \le \norm{Y_j[\sigma,\delta](t)}^{r/j}_j$
for $0 \le t \le t_{\rm tay}(1)$ and $1 \le j \le r-1$.
 Using the estimates  in Lemmas \ref{lem:rem:bnd:n:rem:i:a}
 and \ref{lem:rem:bnd:n:rem:i:b},
the desired bound now follows by applying \eqref{eq:tay:bnd:w:j:s:ell:p:div:j}  with $\ell = r$.
\end{proof}

Turning to $\mathcal{R}_{II}$ and $\Upsilon$,
we recall that $r \ge 3$ and introduce the decomposition
\begin{equation}
    N_{\rm rem;II}[\sigma,\delta] = N_{\rm rem;IIa}[\sigma,\delta]
    + N_{\rm rem;IIb}[\sigma,\delta],
    \qquad
    N_{\rm rem;III}[\sigma,\delta] = N_{\rm rem;IIIa}[\sigma,\delta]
    + N_{\rm rem;IIIb}[\sigma,\delta],
\end{equation}
involving the multi-linear terms
\begin{equation}
\label{eq:rem:def:n:rem:ii:iii:a}
\begin{array}{lcl}
      N_{\rm rem;IIa}[\sigma,\delta] & = & \sigma^2  \sum_{\ell = 1}^{r-3} \sum_{i_1 + \ldots + i_{\ell} \ge r-2} \frac{1}{\ell !} D^{\ell} \mathcal{R}_{II}(0) \big[ Y_{i_1}[\sigma,\delta], \ldots , Y_{i_\ell}[\sigma,\delta]\big] ,
\\[0.2cm]
N_{\rm rem;IIIa}[\sigma,\delta] & = & \sigma^2  \sum_{\ell = 1}^{r-3} \sum_{i_1 + \ldots + i_{\ell} \ge r-2} \frac{1}{\ell !} D^{\ell} \Upsilon(0) \big[ Y_{i_1}[\sigma,\delta], \ldots , Y_{i_\ell}[\sigma,\delta] \big] ,
\end{array}
\end{equation}
together with the nonlinear residuals
\begin{equation}
\label{eq:rem:def:n:rem:ii:iii:b}
\begin{array}{lcl}
    N_{\rm rem;IIb}[\sigma,\delta]
    &=&\sigma^2 \int_0^1 \cdots \int_0^1
    \mathcal{Q}_{II}(Y_{\rm tay}[\sigma,\delta]; t_1, \ldots, t_{r-2} ) \, \mathrm d t_{r-2} \cdots \mathrm  dt_1 ,
\\[0.2cm]
N_{\rm rem;IIIb}[\sigma,\delta]
    &=&\sigma^2 \int_0^1 \cdots \int_0^1
    \mathcal{Q}_{III}(Y_{\rm tay}[\sigma,\delta]; t_1, \ldots, t_{r-2} ) \, \mathrm dt_{r-2} \cdots \mathrm  dt_1,
\end{array}
\end{equation}
featuring the integrands
\begin{equation}
\label{eq:rem:def:q:ii:iii}
\begin{array}{lcl}
    \mathcal{Q}_{II}(Y_{\rm tay}; t_1, \ldots, t_{r-2})
    &=& D^{r-2} \mathcal{R}_{II}( t_1 \cdots t_{r-2} Y_{\rm tay})\big[Y_{\rm tay}, t_1 Y_{\rm tay}, t_1 t_2 Y_{\rm tay}, \ldots, t_1 \cdots t_{r-3} Y_{\rm tay}\big] ,
\\[0.2cm]
\mathcal{Q}_{III}(Y_{\rm tay}; t_1, \ldots, t_{r-2})
    &=& D^{r-2} \Upsilon( t_1 \cdots t_{r-2} Y_{\rm tay})\big[Y_{\rm tay}, t_1 Y_{\rm tay}, t_1 t_2 Y_{\rm tay}, \ldots, t_1 \cdots t_{r-3} Y_{\rm tay}\big] .
\end{array}
\end{equation}

\begin{lem}
\label{lem:tay:bnd:n:rem:ii:iii:a}
Suppose that \textnormal{(HNL)}, \textnormal{(HTw)}, \textnormal{(H$V_*$)} and \textnormal{(Hq)} hold. Then there exists
$K > 0$ so that for all $t < t_{\rm tay}(1)$,
all $\sigma \ge 0$ and all $\delta \ge 0$ we have the bound
\begin{equation}
\label{eq:tay:bnd:n:rem:ii:iii:a}
\begin{array}{lcl}
\norm{N_{\rm rem;IIa}[\sigma,\delta](t)}_{H^{k_*} } +  \norm{N_{\rm rem;IIIa}[\sigma,\delta](t)}_{H^{k_*} }
 & \le & K \sigma^2 \sum_{i=1}^{r-2} \norm{Y_i[\sigma,\delta](t)}_{i}^{(r-2)/i}
\\[0.2cm]
& & \qquad \qquad
   + K \sigma^2 \norm{Y_{r-1}[\sigma,\delta](t)}_{r-1}.
\end{array}
\end{equation}
\end{lem}
\begin{proof}
Consider one of the terms in the sum
 \eqref{eq:rem:def:n:rem:ii:iii:a}
and write $i_{\rm tot} = i_1 + \ldots + i_{\ell}$.
We note that the term is well-defined in $H^{k_*}$
on account of \eqref{eq:tay:rem:bnds:k_j:big} together with Corollaries \ref{cor:sm:r:ii} and \ref{cor:sm:ups}.
Using \eqref{eq:nl:bnd:d:j:r:ii} and \eqref{eq:tay:prod:bnd:w}
we obtain the bound
\begin{equation}
    \norm{D^{\ell} \mathcal{R}_{II}(0)[Y_{i_1}, \ldots , Y_{i_{\ell}}]}_{H^{k_*}}
    \le C \big[ \norm{Y_{i_1}}_{i_1}^{i_{\rm tot}/i_1}
    + \ldots + \norm{Y_{i_\ell}}_{i_\ell}^{i_{\rm tot}/i_{\ell}}
    \big].
\end{equation}
If $i_{\rm tot} = r-2$ holds, then we must have $1 \le i_{\ell'} \le r-2$ for all $1 \le \ell' \le \ell$.
Using the a-priori bound $\norm{Y_{i_{\ell'}}}_{i_{\ell'}} \le 1$, we hence find
\begin{equation}
    \norm{D^{\ell} \mathcal{R}_{II}(0)[Y_{i_1}, \ldots , Y_{i_{\ell}}]}_{H^{k_*}}
    \le C \ell \sum_{i=1}^{r-2} \norm{Y_i}_{i}^{(r-2)/i}.
\end{equation}
However, when $i_{\rm tot} \ge r-1$, we may conclude
\begin{equation}
    \norm{D^{\ell} \mathcal{R}_{II}(0)[Y_{i_1}, \ldots , Y_{i_{\ell}}]}_{H^{k_*}}
    \le C \ell \sum_{i=1}^{r-1} \norm{Y_i}_{i}^{(r-1)/i}.
\end{equation}
Both cases can be absorbed in the stated estimate
\eqref{eq:tay:bnd:n:rem:ii:iii:a} and $N_{\rm rem;IIIa}$ 
can be treated in the same fashion.
\end{proof}

\begin{lem}
\label{lem:rem:bnd:n:rem:ii:iii:b}
Suppose that \textnormal{(HNL)}, \textnormal{(HTw)}, \textnormal{(H$V_*$)} and \textnormal{(Hq)} hold. Then there exists $K > 0$
so that for all $t \le t_{\rm tay}(1)$,
all $\sigma \ge 0$ and all $\delta \ge 0$ we have the bound
\begin{equation}
\begin{array}{lcl}
\norm{N_{\rm rem;IIb}[\sigma,\delta](t)}_{H^{k_*} } +  \norm{N_{\rm rem;IIIb}[\sigma,\delta](t)}_{H^{k_*} }
 & \le & K \sigma^2 \sum_{i=1}^{r-1} \norm{Y_i[\sigma,\delta](t)}_{i}^{r-2}.
 \\[0.2cm]
\end{array}
\end{equation}
\end{lem}
\begin{proof}
Using
the bounds \eqref{eq:nl:bnd:d:j:r:ii}
and \eqref{eq:sm:bnd:derivs:ups}
together with \eqref{eq:tay:bnd:w:unif},
we see that there exists $C > 0$ so that
\begin{equation}
\begin{array}{lcl}
    \norm{\mathcal{Q}_{II}(Y_{\rm tay}[\sigma,\delta](t); t_1, \ldots , t_{r-2})}_{H^{k_*}}
    &\le& C \norm{Y_{\rm tay}[\sigma,\delta](t)}_{H^{k_* +1}}^{r-2}
\\[0.2cm]
    \norm{\mathcal{Q}_{III}(Y_{\rm tay}[\sigma,\delta](t); t_1, \ldots , t_{r-2})}_{H^{k_*}}
    &\le& C \norm{Y_{\rm tay}[\sigma,\delta](t)}_{H^{k_* +2}}^{r-2}
\end{array}
\end{equation}
both hold for $0 \le t \le t_{\rm tay}(1)$. The desired bound
\eqref{eq:tay:bnd:rem:i:b}
now follows from the representation
\eqref{eq:rem:def:n:rem:ii:iii:b}
and the estimate \eqref{eq:tay:bnd:w:tay}.
\end{proof}

\begin{cor}
\label{cor:tay:rem:full:n:ii:iii}
    Suppose that \textnormal{(HNL)}, \textnormal{(HTw)}, \textnormal{(H$V_*$)} and \textnormal{(Hq)} hold.
        Then there exists $K > 0$ so that for all (real) $p \ge 1$,
    all $0 \le \sigma  \le 1$, all $\delta \ge 0$ and  all $T \ge 2$
    we have the bound
    \begin{equation}
\label{eq:tay:bnd:rem:sup:n:ii:iii}
\begin{array}{lcl}
\mathbb E \sup_{0 \le t \le T \wedge t_{\rm tay}(1) }
\big[ \norm{N_{\rm rem;II}[\sigma,\delta](t)}_{H^{k_*}}^{2p}
+ \norm{N_{\rm rem;III}[\sigma,\delta](t)}_{H^{k_*}}^{2p}
\big]
 & \le & K^{2 p} \big( [\sigma^2 p]^{  p r }
+ [\sigma^2 \ln T + \delta^2]^{  p r  }    \big) .
\end{array}
\end{equation}
\end{cor}
\begin{proof}
We first note that for $0 \le t \le t_{\rm tay}(1)$
we have
 \begin{equation}
     \norm{Y_j}^{r-2}_{j} \le \norm{Y_j}^{(r-2)/j}_{j} 
 \end{equation}
 for all $1 \le j \le r-2$, together with
\begin{equation}
     \norm{Y_{r-1}}^{r-2}_{r-1} \le \norm{Y_{r-1}}_{r-1}.
 \end{equation}
  In particular, there exists a constant $C > 0$
  so that
\begin{equation}
\begin{array}{lcl}
    \norm{N_{\rm rem;II}[\sigma,\delta](t)}_{H^{k_*}}^{2p}
& \le & C^{2p} \sigma^{4p}
 \sum_{i=1}^{r-2} \norm{Y_i[\sigma,\delta](t)}_{i}^{2p (r-2)/i}
+ C^{2p} \sigma^{2p} \norm{Y_{r-1}[\sigma,\delta](t)}_{r-1}^{2p}
\end{array}
\end{equation}
for all $0 \le t \le t_{\rm tay}(1)$,
where we exploited the fact that $\sigma^2 \le \sigma$ to obtain the second term.

Applying \eqref{eq:tay:bnd:w:j:s:ell:p:div:j}  with $\ell = r-2$
together with \eqref{eq:tay:work:hyp},
we arrive at
\begin{equation}
\begin{array}{lcl}
\mathbb E \sup_{0 \le t \le T \wedge t_{\rm tay}(1) }
\norm{N_{\rm rem;II}[\sigma,\delta](t)}_{H^{k_*}}^{2p}
 & \le & K^{2 p} \sigma^{4p} \big( [\sigma^2 p]^{  p (r-2) }
+ [\sigma^2 \ln T + \delta^2]^{  p (r-2)  }    \big) ,
        \\[0.2cm]
&  & \qquad  + K^{2p} \sigma^{2p}
\big( [\sigma^2 p]^{  p (r-1) }
+ [\sigma^2 \ln T + \delta^2]^{  p (r-1)  }    \big).
\end{array}
\end{equation}
This can be absorbed in \eqref{eq:tay:bnd:rem:sup:n:ii:iii}
by applying Young's inequality as in \eqref{eq:tay:yng:ineq}.
We conclude by noting that $N_{\rm rem;III}$ satisfies the same estimates.
\end{proof}

Turning to $\mathcal{S}$, we introduce the decomposition
\begin{equation}
B_{\rm rem}[\sigma, \delta] = B_{\rm rem;a}[\sigma, \delta] + B_{\rm rem;b}[\sigma,\delta],
\end{equation}
now involving the multi-linear term
\begin{equation}
\label{eq:rem:def:b:rem:a}
      B_{\rm rem;a}[\sigma,\delta] =\sigma  \sum_{\ell = 1}^{r-2} \sum_{i_1 + \ldots + i_{\ell} \ge r-1} \frac{1}{\ell !} D^{\ell} \mathcal{S}(0) \big[ Y_{i_1}[\sigma,\delta], \ldots , Y_{i_\ell}[\sigma,\delta] \big],
\end{equation}
together with the nonlinear residual
\begin{equation}
\label{eq:rem:def:b:rem:b}
    B_{\rm rem;b}[\sigma,\delta] =\sigma \int_0^1 \cdots \int_0^1
    \mathcal{Q}_B( Y_{\rm tay}[\sigma,\delta]; t_1, \ldots, t_{r-1} )
    \, \mathrm dt_{r-1} \cdots  \mathrm dt_1,
\end{equation}
featuring the integrand
\begin{equation}
    \mathcal{Q}_B (Y_{\rm tay}; t_1, \ldots , t_{r-1})
    = D^{r-1} \mathcal{S}(t_1 \cdots t_{r-1} Y_{\rm tay})\big[Y_{\rm tay}, t_1 Y_{\rm tay}, t_1 t_2 Y_{\rm tay}, \ldots, t_1 \cdots t_{r-2} Y_{\rm tay}\big].
\end{equation}

\begin{lem}
\label{lem:tay:bnd:b:rem:a}
Suppose that \textnormal{(HNL)}, \textnormal{(HTw)}, \textnormal{(H$V_*$)} and \textnormal{(Hq)} hold. Then there exists
$K > 0$ so that for all $t < t_{\rm tay}(1)$,
all $\sigma \ge 0$ and all $\delta \ge 0$ we have the bound
\begin{equation}
\label{eq:tay:bnd:b:rem:a}
\begin{array}{lcl}
\norm{B_{\rm rem;a}[\sigma,\delta](t)}_{HS(L^2_Q;H^{k_*} )}
 & \le & K \sigma \sum_{i=1}^{r-1} \norm{Y_i[\sigma,\delta](t)}_{i}^{(r-1)/i}.
 \end{array}
\end{equation}
\end{lem}
\begin{proof}
Consider one of the terms in the sum
 \eqref{eq:rem:def:b:rem:a}
and write $i_{\rm tot} = i_1 + \ldots + i_{\ell}$.
We note that the term is well-defined in $H^{k_*}$
on account of Lemma \ref{lem:sm:s} and
\eqref{eq:tay:rem:bnds:k_j:big}.
Using \eqref{eq:nl:bnd:d:j:mathcal:s} and \eqref{eq:tay:prod:bnd:w}
we obtain the bound
\begin{equation}
    \norm{D^{\ell} \mathcal{S}(0)[Y_{i_1}, \ldots , Y_{i_{\ell}}]}_{HS(L^2_Q;H^{k_*})}
    \le K \big[ \norm{Y_{i_1}}_{i_1}^{i_{\rm tot}/i_1}
    + \ldots + \norm{Y_{i_\ell}}_{i_\ell}^{i_{\rm tot}/i_{\ell}}
    \big].
\end{equation}
The desired estimate \eqref{eq:tay:bnd:b:rem:a} hence follows by noting that $i_{\rm tot} \ge r-1$
and using the a-priori bound $\norm{Y_{i_{\ell'}}}_{i_{\ell'}} \le 1$ that is available for $1 \le \ell' \le \ell$ as a consequence of the stopping time.
\end{proof}

\begin{lem}
\label{lem:rem:bnd:b:rem:b}
Suppose that \textnormal{(HNL)}, \textnormal{(HTw)}, \textnormal{(H$V_*$)} and \textnormal{(Hq)} hold. Then there exists $K > 0$
so that for all $0 \le t \le t_{\rm tay}(1)$,
all $\sigma \ge 0$ and all $\delta \ge 0$ we have the bound
\begin{equation}
\begin{array}{lcl}
\norm{B_{\rm rem;b}[\sigma,\delta](t)}_{HS(L^2_Q;H^{k_*}) }
 & \le & K \sigma \sum_{i=1}^{r-1} \norm{Y_i[\sigma,\delta](t)}_{i}^{r-1}.
 \\[0.2cm]
\end{array}
\end{equation}
\end{lem}
\begin{proof}
This follows as in the proof of Lemma \ref{lem:rem:bnd:n:rem:i:b}
by considering the representation \eqref{eq:rem:def:b:rem:b}
and applying the bound \eqref{eq:nl:bnd:d:j:mathcal:s}.
\end{proof}

\begin{cor}
\label{cor:tay:rem:full:b}
    Suppose that \textnormal{(HNL)}, \textnormal{(HTw)}, \textnormal{(H$V_*$)} and \textnormal{(Hq)} hold.
        Then there exists $K > 0$ so that for all (real) $p \ge 1$,
    all $\sigma  \ge 0$, all $\delta \ge 0$  and all $T \ge 2$
    we have the bound
    \begin{equation}
\begin{array}{lcl}
\mathbb E \sup_{0 \le t \le T \wedge t_{\rm tay}(1) } \norm{B_{\rm rem}[\sigma,\delta](t)}_{HS(L^2_Q;H^{k_*})}^{2p} 
 & \le & K^{2 p} \sigma^{2p} \big( [\sigma^2 p]^{  p (r-1) }
+ [\sigma^2 \ln T + \delta^2]^{  p (r-1)  }    \big) .
        \\[0.2cm]
\end{array}
\end{equation}
\end{cor}
\begin{proof}
    This can be established by following the proof of Corollary \ref{cor:tay:rem:full:n:i} and replacing $r$ by $r-1$.
\end{proof}

\begin{proof}[Proof of Proposition \ref{prop:tay:bnd:rem}]
This follows directly from Corollaries \ref{cor:tay:rem:full:n:i},
\ref{cor:tay:rem:full:n:ii:iii} and \ref{cor:tay:rem:full:b}.
\end{proof}

\section{Limiting expectations}
\label{sec:lim}

We here study the limiting behaviour of the expectation of functionals that act on the expansion functions $Y_{j}$
defined in \eqref{eq:mr:def:w:j:all}. In particular,
in {\S}\ref{subsec:lim:ml} we consider multilinear operators
and establish Proposition \ref{prop:mr:multiliner:lim},
using an explicit procedure to compute the associated limits.
We then turn to general smooth functionals in {\S}\ref{subsec:lim:phi} and establish Proposition \ref{prop:mr:tay:lim:phi}.

\subsection{Multilinear forms}
\label{subsec:lim:ml}

Given a multi-linear expression of the form $\Lambda[Y_{i_1}, \ldots , Y_{i_{\ell}}]$ for some tuple $\{i_1, \ldots, i_{\ell}\} \in \{1, \ldots, r-1\}^{\ell}$, our strategy is to repeated apply the It{\^o} lemma
to eliminate the references to the functions $Y_j$ and construct a  representation that only involves the constant
expressions
\begin{equation}
    \varrho_N = \mathcal{R}_{II}(0) + \Upsilon(0) \in H_1 = H^{k_1},
    \qquad \qquad
    \varrho_B = \mathcal{S}(0) \in HS(L^2_Q;H_1).
\end{equation}
Indeed, inspecting
\eqref{eq:mr:def:n:j} and \eqref{eq:mr:def:b:j},
we observe that
\begin{equation}
N_{2;II}[\sigma,\delta] + N_{2;III}[\sigma,\delta] = \sigma^2 \varrho_N   ,
\qquad \qquad
B_{1}[\sigma,\delta] = \sigma \varrho_B,
\end{equation}
while all other expressions vanish upon taking $Y_1 = \ldots = Y_{r-1} =0$.

To achieve the above, we extend the class of multi-linear maps that we consider to also allow $\Lambda$ to depend
on the pair $(\varrho_N, \varrho_B)$. In particular, we impose the following structural condition.

\begin{itemize}
    \item[(h$\Lambda$)]{
    The map
    \begin{equation}
       \Lambda: [H_1]^{m_N} \times [HS(L^2_Q;H_1)]^{m_B} \times H_{i_1} \times \ldots \times H_{i_{\ell}} \to \mathcal{H}
    \end{equation}
    is a bounded $(m_N + m_B + \ell)$-linear map into $\mathcal{H}$. In addition, we have $m_N \ge 0$, $m_B \ge 0$ and $\ell \ge 0$, together with the ordering
    $r-1 \ge i_1 \ge \ldots \ge i_{\ell} \ge 1$.
    }
\end{itemize}

For convenience, we introduce the notation
\begin{equation}
S(t) = \mathrm{exp}[ (\mathcal{L}_{\rm tw} + \Delta_{x_\perp})t],
\end{equation}
which corresponds to the evolution family in {\S}\ref{sec:conv}
with $\nu = 1$ via the relation $E(t, s) = S(t-s)$.
For any pair
\begin{equation}
\label{eq:lim:def:theta}
\theta = (\theta_N, \theta_B) \in [H_1]^{m_N} \times [HS(L^2_Q;H_1)]^{m_B},
\end{equation}
we are interested in the
expression
\begin{equation}
    \mathcal{I}_\Lambda(t,s;\theta) = \Lambda[ \theta, S(t-s)Y_{i_{1}}(s) , \ldots,  S(t-s)Y_{i_\ell}(s)].
\end{equation}

\begin{lem}
\label{lem:tay:red:single:step}
Suppose that \textnormal{(HNL)}, \textnormal{(HTw)}, \textnormal{(H$V_*$)} and \textnormal{(Hq)} all hold. Pick a Hilbert space $\mathcal{H}$ together with a multi-linear map $\Lambda$ that satisfies \textnormal{(h$\Lambda$)}.
Then for any $t \ge s \ge 0$
and any pair \eqref{eq:lim:def:theta},
the difference
\begin{equation}
 \mathbb{E}  \Big[  \mathcal{I}_\Lambda(t,s;\theta) \Big]
 -
\mathcal{I}_\Lambda(t,0;\theta)
\end{equation}
can be written as a finite sum of terms of the form
\begin{equation}
   \sigma^{i_0^*} \mathbb {E} \,  \int_0^s \mathcal{I}_{\Lambda^*}(t, s'; \theta^*(s') ) \, \mathrm ds'
\end{equation}
in which $i^*_0 \ge 0$ and $\Lambda^*$ satisfies \textnormal{(h$\Lambda$)}
with an index set
$(i_1^*,\ldots, i^*_{\ell^*})$ that is strictly less than $(i_1,\ldots,i_\ell)$ in lexicographical order, with
\begin{equation}
\label{eq:tay:sum:i:1:through:i:ell}
   i_1+ \ldots + i_\ell  = i_0^* + i_1^* + \ldots i^*_{\ell^*},
\end{equation}
together with
\begin{equation}
\label{eq:lim:sum:m:n:b:i:0}
    m_N + m_B  + i_0^* = m_N^* + m_B^*.
\end{equation}
In addition, one of the following four options hold.
\begin{itemize}
\item[\textnormal{(a)}]{We have $m_N^* = m_N$ and $m_B^* = m_B$ together with $\theta^* = \theta$}; or
\item[\textnormal{(b)}]{
We have $m_N^* =m_N + 1$ and $m_B^* = m_B$,
together with $\theta_N^*(s') = \big( \theta_N, S(t-s') \varrho_N \big)$ and $\theta_B^* = \theta_B$; or
}
\item[\textnormal{(c)}]{
We have $m_N^* = m_N$ and $m_B^* = m_B + 1$, together with
$\theta_N^* = \theta_N$ and
\begin{equation}
    \theta_B^*(s') = \big( \theta_B, S(t-s') \varrho_B \big).
\end{equation}
}
\item[\textnormal{(d)}]{
We have $m_N^* = m_N$ and $m_B^* = m_B + 2$, together with
$\theta_N^* = \theta_N$ and
\begin{equation}
    \theta_B^*(s') = \big( \theta_B, S(t-s') \varrho_B, S(t-s') \varrho_B \big).
\end{equation}
}
\end{itemize}
\end{lem}
\begin{proof}
Applying the mild It{\^o} formula \cite{dapratomild},
we obtain
\begin{equation}
\begin{array}{lcl}
    \mathcal{I}_\Lambda(t,s;\theta)
    & = &
    \Lambda[ \theta, S(t) Y_{i_1}(0), ... , S(t) Y_{i_\ell}(0)]
      \\[0.2cm]
      & & \qquad
    + \sum_{1 \le \ell' \le \ell}  \int_0^{s} \Lambda [\theta, S(t-s') Y_{i_1}(s'), \ldots , S(t-s) 
    N_{i_{\ell'}}(s'),
    \ldots , S(t-s') Y_{i_\ell}(s')] \, \mathrm ds'
\\[0.2cm]
& & \qquad
    + \sum_{k \in \mathbb{Z}}
    \sum_{1 \le \ell'_1 < \ell'_2 \le \ell}
     \int_0^{s} \Lambda \Big[ \theta, S(t-s')Y_{i_1}(s'), \ldots,
    S(t-s') B_{i_{\ell'_1}}(s') \sqrt{Q} e_k, \ldots,
    \\[0.2cm]
    & & \qquad \qquad \qquad \qquad \qquad
    \qquad \qquad
    S(t-s')B_{i_{\ell'_2} }(s') \sqrt{Q} e_k, \ldots , S(t-s')Y_{i_\ell}(s') \Big] \,  \mathrm ds'
\\[0.2cm]
& & \qquad
+
\sum_{1 \le \ell' \le \ell}  \int_0^{s} \Lambda [ \theta, S(t-s') Y_{i_1}(s), \ldots , S(t-s') 
    B_{i_{\ell'}}(s')[\,\cdot\,],
    \ldots , S(t-s') Y_{i_\ell}(s')] \, \mathrm dW^Q_{s'}.
\end{array}
\end{equation}
The stochastic integral vanishes upon taking expectations.
We can now use the definitions
\eqref{eq:mr:def:n:j} and \eqref{eq:mr:def:b:j}
to obtain the stated representation, using the fact that
for any pair $b_i \in HS(L^2_Q; H_i)$ and $b_j \in HS(L^2_Q;H_j)$ we have
\begin{equation}
\begin{array}{lcl}
\sum_{k \in \Wholes} \norm{ b_i \sqrt{Q} e_k }_{H^{k_i}} \norm{ b_j \sqrt{Q} e_k }_{j}  &\le & \big[ \sum_{k \in \Wholes } \norm{ b_i \sqrt{Q} e_k }_{i}^2 \big]^{1/2} \big[ \sum_{k \in \Wholes } \norm{ b_j \sqrt{Q} e_k }_{j}^2 \big]^{1/2}
\\[0.2cm]
& = & \norm{b_i}_{HS(L^2_Q;H_i)} \norm{b_j}_{HS(L^2_Q;H_j)}
\end{array}
\end{equation}
to show that each $\Lambda^*$ is well-defined as a bounded multi-linear map.
The strictly decreasing lexicographic order follows from the ordering
on the indices $(i_{\ell'})_{\ell' = 1 \ldots \ell}$ and $(i^*_{\ell'})_{\ell' = 1 \ldots \ell^*}$ and the
fact that $N_j$ and $B_j$ only depend on $Y_{j'}$ with
$1 \le j < j'$, besides the constant terms $\varrho_N$ and $\varrho_B$.
\end{proof}

It is convenient to introduce the notation
\begin{equation}
    \begin{array}{lcl}
    \theta^*_N(t;s_1, \ldots ,s_n)  &= &
        \Big( [S(t-s_1)\varrho_N]^{m^*_{N;1}},
        \ldots [S(t-s_n)\varrho_N]^{m^*_{N;n}} \Big),
    \\[0.2cm]
    \theta^*_B(t;s_1, \ldots, s_n)
        & = &
        \Big( [S(t-s_1) \varrho_B]^{m^*_{B;1}},
\ldots , [S(t-s_n) \varrho_B]^{m^*_{B;n}} \Big)
    \end{array}
    \end{equation}
with integers $(m^*_{N;1}, \ldots , m^*_{N;n}) \in \{0, 1\}^n$
and $(m^*_{B;1}, \ldots , m^*_{B;n}) \in \{0, 1,2\}^n$
that satisfy
\begin{equation}
    m^*_{N} = m^*_{N;1} + \ldots + m^*_{N;n},
    \qquad \qquad
    m^*_{B} = m^*_{B;1} + \ldots + m^*_{B;n}.
\end{equation}
Here the notation $[\,\cdot\,]^m$  means that the argument
should be repeated $m$ times,
which ensures that
\begin{equation}
\theta^*(t;s_1 , \ldots, s_n) = \big(\theta^*_N(t;s_1, \ldots , s_n), \theta^*_B(t;s_1, \ldots, s_n)\big) \in [H^{k_1}]^{m^*_N} \times [HS(L^2_Q;H^{k_1})]^{m^*_B}.
\end{equation}
We recall that $Y_j(0) = 0$ for all $2 \le j < r$,
simplifying some of the expressions below.

\begin{cor}
\label{cor:tay:repr:e:gamma}
Suppose that \textnormal{(HNL)}, \textnormal{(HTw)}, \textnormal{(H$V_*$)} and \textnormal{(Hq)} all hold. Pick a Hilbert space $\mathcal{H}$ together with a multi-linear map $\Lambda$ that satisfies \textnormal{(h$\Lambda$)} with $m_N = m_B =0$.
Then for any
    $t \ge 0$ the difference
    \begin{equation}
    \label{eq:lim:defn:diff:to:be:written}
        \mathbb{E} \Lambda \big[ Y_{i_1}(t), \ldots, Y_{i_\ell}(t) \big]
        - \Lambda[S(t) Y_{i_1}(0), \ldots S(t) Y_{i_{\ell}(0)}]
    \end{equation}
    can be written as a finite sum of terms of the form
    \begin{equation}
    \label{eq:tay:decomp:e:star}
        \mathcal{E}^*(t) = \sigma^{i^*_0} \int_0^t \cdots \int_0^{s_{n-1}}
\mathcal{I}_{\Lambda^*}[ t, 0; \theta^*(t; s_1, \ldots , s_n)
]
 \, \mathrm ds_n \cdots \mathrm d s_1
    \end{equation}
    in which $n \ge 1$
    and the map $\Lambda^*$ with the corresponding starred integers satisfy \textnormal{(h$\Lambda$)}
     and the identity \eqref{eq:tay:sum:i:1:through:i:ell}.
    If $\ell^* = 0$, then we either have
    $m_{N;n}^* = 1$ or $m_{B;n}^* \ge 1$.
\end{cor}
\begin{proof}
    This follows from an iterative application
    of Lemma \ref{lem:tay:red:single:step}.
    The fact that $m_{N;n}^* = 1$ or $m_{B;n}^* \ge 1$
    holds when $\ell^* = 0$
    follows from the fact that in the final
    step either option (b), (c) or (d) must hold
    in order to eliminate the remaining reference(s)
    to the processes $Y_j$. 
\end{proof}

\begin{lem}
\label{lem:tay:lim:e:star:decay}
Consider the setting of Corollary
\ref{cor:tay:repr:e:gamma}. Then
there exists $K > 0$ so that
for any $t \ge 0$
 the expressions
\eqref{eq:tay:decomp:e:star} satisfy
the bound
\begin{equation}
    \norm{\mathcal{E}^*(t)}_{\mathcal{H}}
    \le K \sigma^{i_0^*} \delta^{\ell^*} e^{-\beta t} t^n
\end{equation}
if $\ell^* \ge 1$, with $i_0^* + \ell_* = i_1 + \ldots + i_{\ell}$.
\end{lem}
\begin{proof}
Since $\ell^* \ge 1$, the expression is non-zero only if $i^*_1 = \ldots = i^*_{\ell^*} = 1$, in which case we may use
$\norm{ Y_{1}(0) }_{H^{k_1}} = \delta$
to obtain the bound
\begin{equation}
    \norm{ \mathcal{I}_{\Gamma^*}[ t, 0; \theta^*(t; s_1, \ldots , s_n)  }_{\mathcal{H}}
    \le \norm{\Lambda} M^{m_N^* + m_B^*} \norm{\varrho_N}_{H^{k_1}}^{m_N^*}\norm{\varrho_B}_{HS(L^2_Q;H^{k_1})}^{m_B^*}
    M^{\ell_*} e^{-\beta \ell^* t} \delta^{\ell^*}.
\end{equation}
Repeatedly applying \eqref{eq:lim:sum:m:n:b:i:0} we see that
$i_0^* = m_N^* + m_B^*$. In addition, \eqref{eq:tay:sum:i:1:through:i:ell}
shows that $i_0^* + \ell^* = i_1 + \ldots + i_{\ell}$,
which allows us to write
\begin{equation}
    \norm{ \mathcal{I}_{\Gamma^*}[ t, 0; \theta^*(t; s_1, \ldots , s_n)  }_{\mathcal{H}}
    \le K e^{-\beta \ell^* t} \delta^{\ell^*}
\end{equation}
for some constant $K > 0$ that only depends on the multi-linear map $\Lambda$.
The desired bound now follows by performing the integration
\eqref{eq:tay:decomp:e:star}, which contributes the factor $\sigma^{i_0^*} t^n$.
\end{proof}

\begin{lem}
\label{lem:tay:lim:exists}
Consider the setting of Corollary \ref{cor:tay:repr:e:gamma}. Then
there exists $K > 0$ so that the following holds true.
For each of the expressions
\eqref{eq:tay:decomp:e:star} with $\ell^* = 0$ there exists $\mathcal{E}^*_\infty \in \mathcal{H}$
so that the bound
\begin{equation}
    \norm{\mathcal{E}^*(t) - \mathcal{E}^*_\infty}_{\mathcal{H}}
    \le K \sigma^{i_0^*}  e^{-\frac{1}{2} \beta t}
\end{equation}
holds for any $t \ge 0$.
\end{lem}
\begin{proof}
  Reversing the integration order in \eqref{eq:tay:decomp:e:star} and writing $\tilde{s}_n = t - s_n$, we obtain
\begin{equation}
\label{eq:tay:limit:e:star}
\mathcal{E}^*(t)
=
\sigma^{i_0^*} \int_{0}^t \int_0^{\tilde{s}_n} \cdots \int_0^{\tilde{s}_{2}}
\Lambda^*[  \tilde{\theta}^*(\tilde{s}_1, \ldots, \tilde{s}_n) ]
  \, \mathrm d\tilde{s}_1 \cdots \mathrm d\tilde{s}_n ,
\end{equation}
in which $\tilde{\theta}^* = (\tilde{\theta}_N^*, \tilde{\theta}_B^*)$ is given by
\begin{equation}
    \begin{array}{lcl}
    \tilde{\theta}^*_N(\tilde{s}_1, \ldots ,\tilde{s}_n)  &= &
        \Big( [S(\tilde{s}_1)\varrho_N]^{m^*_{N;1}},
        \ldots, [S(\tilde{s}_n)\varrho_N]^{m^*_{N;n}} \Big),
    \\[0.2cm]
    \tilde{\theta}^*_B(\tilde{s}_1, \ldots ,\tilde{s}_n)
        & = &
        \Big( [S(\tilde{s}_1) \varrho_B]^{m^*_{B;1}},
\ldots , [S(\tilde{s}_n) \varrho_B]^{m^*_{B;n}} \Big).
    \end{array}
    \end{equation}
Since $m^*_{N:n} =1 $ or $m^*_{B;n} \ge 1$,
the inner integrals
in \eqref{eq:tay:limit:e:star} satisfy the bound
\begin{equation}
\begin{array}{lcl}
\norm{\int_0^{\tilde{s}_n} \cdots \int_0^{\tilde{s}_{2}}
\Gamma^*[ \tilde{\theta}^*(\tilde{s}_1, \ldots, \tilde{s}_n) ]
 \, \mathrm d\tilde{s}_1 \cdots \mathrm d\tilde{s}_{n-1}
 }_{\mathcal{H}}
  & \le &
  \norm{\Lambda} M^{m_N^* + m_B^*} \norm{\varrho_N}_{H^{k_1}}^{m_N^*}\norm{\varrho_B}_{HS(L^2_Q;H^{k_1})}^{m_B^*}
  e^{-\beta\tilde{s}_n} |\tilde{s}_n|^{n-1}
\\[0.2cm]
& \le &
K  e^{-\beta\tilde{s}_n} |\tilde{s}_n|^{n-1},
\end{array}
\end{equation}
where the second inequality follows by noting that $i_0^* = m_N^* + m_B^* = i_1 + \ldots + i_{\ell}$.
Since this is integrable with respect to $\tilde{s}_n$,
the representation \eqref{eq:tay:limit:e:star}
implies that indeed $\mathcal{E}^*(t)$ converges to a limit $\mathcal{E}^*_\infty$ at the specified
exponential rate.
\end{proof}

We have hence shown that the expectation of any multi-linear expression involving the expansion functions $Y_j$ 
converges
exponentially to a well-defined limit.
In addition, we have an
explicit reduction procedure that shows that this limit can be represented
as a sum of expressions of the form \eqref{eq:tay:limit:e:star} with $t = \infty$.

\begin{proof}[Proof of Proposition \ref{prop:mr:multiliner:lim}]
The result follows from
the representation of the difference \eqref{eq:lim:defn:diff:to:be:written}
provided in Corollary \ref{cor:tay:repr:e:gamma}
by applying
Lemmas \ref{lem:tay:lim:e:star:decay} and \ref{lem:tay:lim:exists}.
\end{proof}

\subsection{Smooth functionals}
\label{subsec:lim:phi}

We now turn our attention to the proof of Proposition \ref{prop:mr:tay:lim:phi}.
Given a functional $\phi$ that satisfies (H$\phi$), we write
\begin{equation}
    h_0= \phi(0)
\end{equation}
together with
\begin{equation}
\label{eq:lim:def:z:j}
    h_j[\sigma, \delta]
    = \sum_{\ell=1}^j \sum_{i_1 + \ldots + i_{\ell} = j} \frac{1}{\ell!} D^{\ell} \phi(0)\big[Y_{i_1}[\sigma,\delta], \ldots, Y_{i_{\ell}}[\sigma,\delta] \big]
\end{equation}
for $1 \le j \le r-1$, where in each term we take $1 \le i_{\ell'} \le r-1$ for each $1 \le \ell'\le \ell$. We also introduce the remainder function
\begin{equation}
    h_{\rm rem}[\sigma, \delta]
    = \phi(Y_{\rm tay}[\sigma, \delta]) - \sum_{j=0}^{r-1} h_{j}[\sigma, \delta],
\end{equation}
which can be decomposed as
\begin{equation}
    h_{\rm rem}[\sigma, \delta] = h_{\rm rem;a}[\sigma, \delta] + h_{\rm rem;b}[\sigma, \delta]
\end{equation}
by writing
\begin{equation}
\label{eq:tay:lim:def:z:rem:a}
    h_{\rm rem;a}[\sigma,\delta] = \sum_{\ell= 1}^{r-1} \sum_{i_1 + \ldots + i_{\ell} \ge r} \frac{1}{\ell!} D^\ell \phi(0) \big[Y_{i_1}[\sigma,\delta], \ldots, Y_{i_{\ell}}[\sigma,\delta]  \big],
\end{equation}
together with
\begin{equation}
\label{eq:lim:def:z:rem:b}
    h_{\rm rem;b} =\int_0^1 \cdots \int_0^1
    \mathcal{Q}_h\big(Y_{\rm tay}[\sigma,\delta]; t_1, \ldots, t_{r} \big) \, \mathrm dt_r \cdots \mathrm dt_1 ,
\end{equation}
in which the integrand is given by
\begin{equation}
\label{eq:lim:def:q}
\mathcal{Q}_h( Y_{\rm tay}; t_1, \ldots, t_r)
= D^r \phi( t_1 \cdots t_r Y_{\rm tay})\big[Y_{\rm tay}, t_1 W, t_1 t_2 Y_{\rm tay}, \ldots, t_1 \cdots t_{r-1} Y_{\rm tay}\big].
\end{equation}

Observe that the $h_j$ and $h_{\rm rem;a}$ terms are multi-linear, which allows us to apply the theory developed in {\S}\ref{subsec:lim:ml}.
In particular, the expectation of these terms all converge to a well-defined limit. 

\begin{lem}
\label{lem:tay:bnd:z:j}
Suppose that \textnormal{(HNL)}, \textnormal{(HTw)}, \textnormal{(H$V_*$)} and \textnormal{(Hq)} hold. Pick a Hilbert space $\mathcal{H}$
together with a functional $\phi$ that satisfies \textnormal{(H$\phi$)}.
Then there exist quantities
\begin{equation}
    \big( h_{\infty;1}, \ldots, h_{\infty;r-1} \big) \in \mathcal{H}^{r-1}
\end{equation}
and a constant $K > 0$ so that for all $1 \le j \le r-1$ we have the bound
\begin{equation}
    \norm{ \mathbb E h_j[\sigma,\delta](t) - \sigma^j h_{\infty;j} }_{\mathcal{H}}
    \le ( \sigma^j + \delta^j) K e^{-\frac{\beta}{2} t}
\end{equation}
for all $\sigma \ge 0$ and $\delta \ge 0$.
\end{lem}
\begin{proof}
    This follows directly by applying Proposition \ref{prop:mr:multiliner:lim}
    to each of the terms in the definition \eqref{eq:lim:def:z:j}.
\end{proof}

\begin{lem}
\label{lem:tay:bnd:z:rem:a}
Suppose that \textnormal{(HNL)}, \textnormal{(HTw)}, \textnormal{(H$V_*$)} and \textnormal{(Hq)} hold. Pick a Hilbert space $\mathcal{H}$
together with a functional $\phi$ that satisfies \textnormal{(H$\phi$)}.
Then there exists a function
\begin{equation}
    h_{\rm rem;a;\infty}: [0, 1] \to \mathcal{H}
\end{equation}
together with a constant $K > 0$ so that
for all $0 \le \sigma \le 1$ and $0 \le \delta \le 1$ we have
\begin{equation}
    \norm{ \mathbb E h_{\rm rem;a}[\sigma,\delta](t) - h_{\rm rem;a;\infty}(\sigma)}_{\mathcal{H}}
    \le ( \sigma^r + \delta^r) K e^{-\frac{\beta}{2} t},
\end{equation}
together with
\begin{equation}
    \norm{ h_{\rm rem;a;\infty}(\sigma) } \le K \sigma^r.
\end{equation}
\end{lem}
\begin{proof}
    This follows by applying Proposition \ref{prop:mr:multiliner:lim}
    to each of the terms in the sum \eqref{eq:tay:lim:def:z:rem:a}.
\end{proof}

Our approach for the second remainder $h_{\rm rem;b}$ is rather crude in the sense that we make no attempt to identify the limiting behaviour. In particular,
we simply establish a global residual bound on the size of this expression.
As a preparation, we obtain moment bounds for the size of the
expansion functions $Y_j$ at individual times $t$,
which should be contrasted to the supremum bounds \eqref{eq:mr:mo:bnds:w:j}.

\begin{lem}
\label{lem:tay:bnd:on:w:j:norm}
Suppose that \textnormal{(HNL)}, \textnormal{(HTw)}, \textnormal{(H$V_*$)} and \textnormal{(Hq)} hold.
    For each integer $p \ge 1$, there exists a constant $K_{p}> 0$ so that
    for all $t \ge 0$ we have the bound
    \begin{equation}
        \mathbb E \norm{Y_j(t)}_{k_j}^{2 p} \le K_{p} \big[ \sigma^{2 p j}  +  \delta^{2 p j} e^{-\frac{\beta}{4} t } \big].
    \end{equation}
\end{lem}
\begin{proof}
We first consider the case that $p \ge 1$ is an integer.
We note that the functional
\begin{equation}
\Lambda: (H^{k_j})^{2p} \ni \big( y_1, \ldots , y_{2p} \big) \mapsto \langle y_1 , y_2\rangle_{H^{k_j}} \cdots  \langle y_{2p-1} , y_{2p}\rangle_{H^{k_j}}
\end{equation}
is a bounded $2p$-linear map and that
\begin{equation}
\norm{Y_j(t)}_{k_j}^{2p} = \Lambda\big[ Y_j, \ldots, Y_j ].
\end{equation}
In particular, using Proposition \ref{prop:mr:multiliner:lim} we obtain
\begin{equation}
        \mathbb E \norm{Y_j(t)}_{k_j}^{2 p} \le K_{p} \big[ \sigma^{2 p j}  +  \delta^{2 p j} e^{-\frac{\beta}{2} t } \big].
    \end{equation}
When $p \ge 1$ is not an integer, we
pick $ 1< q < 2$ in such a way that $p q$  is an integer
and use the estimate
$\sqrt[q]{a + b} \le \sqrt[q]{a} + \sqrt[q]{b}$ for $a \ge 0$ and $b \ge 0$
to compute
\begin{equation}
    \mathbb E  \norm{Y_j(t)}_{k_j}^{2 p}
    \le \big[  \mathbb E \norm{Y_j(t)}_{k_j}^{2 p q} \big]^{1/q}
    \le \big[ K_{pq}( \sigma^{2pq j}  + \delta^{2pqj}e^{-\frac{\beta}{2} t} ) \big]^{1/q}
    \le  K_{pq}^{1/q}( \sigma^{2 pj} + \delta^{2 p j} e^{-\frac{\beta}{2q} t} ),
\end{equation}
which satisfies the stated bound since $2q < 4$.
\end{proof}

\begin{lem}
\label{lem:tay:bnd:z:rem:b}
Suppose that \textnormal{(HNL)}, \textnormal{(HTw)}, \textnormal{(H$V_*$)} and \textnormal{(Hq)} hold. Pick a Hilbert space $\mathcal{H}$
together with a functional $\phi$ that satisfies \textnormal{(H$\phi$)}.
Then there exists a constant $K > 0$ so that for all $0 \le \sigma \le 1$, all $0 \le \delta \le 1$ and all $t \ge 0$ we have the bound
    \begin{equation}
        \mathbb E \norm{h_{\rm rem;b}(t)}_{\mathcal{H}} \le
        K \big[ \sigma^r  + \delta^r e^{-\frac{\beta}{4} t} \big].
    \end{equation}
\end{lem}
\begin{proof}
Using (H$\phi$), the expression \eqref{eq:lim:def:q}
can be bounded by
 \begin{equation}
     \mathbb{E} \norm{\mathcal{Q}_h\big(Y_{\rm tay}; t_1, \ldots , t_r\big)}_{\mathcal{H}}
     \le K
     \mathbb E \big[ 1 + \norm{Y_{\rm tay}}_{H^{k_*}}^{N}\big]
        \norm{Y_{\rm tay}}_{H^{k_*}}^r .
 \end{equation}
 Applying \eqref{eq:tay:bnd:w:tay}, we see that
\begin{equation}
     \mathbb{E} \norm{\mathcal{Q}_h\big(Y_{\rm tay}; t_1, \ldots , t_r\big)}_{\mathcal{H}}
    \le K  r^{N + r}
     \mathbb E \big[ 1 + \norm{Y_1}_1^{N} + \ldots + \norm{Y_{r-1}}_{r-1}^{N}]
        \big[ \norm{Y_1}_1^r + \ldots + \norm{Y_{r-1}}_{r-1}^r \big].
 \end{equation}
Inspecting the definition \eqref{eq:lim:def:z:rem:b},
the desired estimate now follows by applying
Lemma \ref{lem:tay:bnd:on:w:j:norm}.
\end{proof}

\begin{proof}[Proof of Proposition \ref{prop:mr:tay:lim:phi}]
Setting $h_{0;\infty} = \phi(0)$,
the result follows by combining
Lemmas \ref{lem:tay:bnd:z:j}, \ref{lem:tay:bnd:z:rem:a}
and \ref{lem:tay:bnd:z:rem:b}.
\end{proof}

\section{Residual}
\label{sec:res}
In this section we study the residual $Z = V - Y_{\rm tay}$ featuring in the full expansion \eqref{eq:mr:decomp:v}. In particular, we derive the relevant evolution system, obtain bounds for the associated nonlinearities in Proposition \ref{prop:res:bnds} and use a time transformation to establish a mild representation for $Z$ in Proposition \ref{prop:res:mild}.

We first introduce the difference expressions
\begin{equation}
\begin{array}{lcl}
\mathcal{R}^\odot_{I}(z;y) &=&   \mathcal{R}_I( y + z) - \mathcal{R}_I( y),
\\[0.2cm]
\mathcal{R}^\odot_{II}(z;y) &=&   \mathcal{R}_{II}( y+z) - \mathcal{R}_{II}( y) ,
\\[0.2cm]
\mathcal{S}^\odot(z;y) & = & \mathcal{S}(y + z) -\mathcal{S}(y) ,
\end{array}
\end{equation}
together with
\begin{equation}
    \Upsilon^\odot(z;y) =  \Upsilon(y + z) - \Upsilon(y),
    \qquad \qquad
    \Upsilon^\odot_{II}(z;y) =
    \Upsilon_{II}(y+z) - \Upsilon_{II}(y) .
\end{equation}
Appealing to the results in {\S}\ref{sec:sm}, we obtain the following bounds.
\begin{lem}
Suppose that \textnormal{(HNL)}, \textnormal{(HTw)} and \textnormal{(Hq)} are satisfied.
Then there exists $K> 0$ so that for any $y \in H^{k_* + 1}$
and $z \in H^{k_* + 1}$ we have the bounds
\begin{equation}
\label{eq:res:est:r:odot:ups:ii:odot}
\begin{array}{lcl}
    \norm{\mathcal{R}^\odot_I(z;y)}_{H^{k_*}}
    & \le & K (1 + \norm{y}_{H^{k_*}}^{k_*} + \norm{z}_{H^{k_*}}^{k_*} )\Big(  (\norm{y}_{H^{k_*}} + \norm{z}_{H^{k_*}} ) \norm{z}_{H^{k_* + 1}}
\\[0.2cm]
& & \qquad \qquad \qquad
     + (\norm{y}_{H^{k_* +1}} +\norm{z}_{H^{k_*+1}}   )\norm{z}_{H^{k_*}} \Big),
\\[0.2cm]
\norm{\mathcal{R}^\odot_{II}(z;y)}_{H^{k_*}}
    & \le & K (1 + \norm{y}_{H^{k_*}}^{k_*} + \norm{z}_{H^{k_*}}^{k_*} )
    \big( 1 + \norm{y}_{H^{k_*+1}} + \norm{z}_{H^{k_*+1}} \big)
    \norm{z}_{H^{k_*}}
\\[0.2cm]
& & \qquad + K (1 + \norm{y}_{H^{k_*}}^{k_*} + \norm{z}_{H^{k_*}}^{k_*} )
    \norm{z}_{H^{k_* + 1}},
\\[0.2cm]
\norm{\Upsilon^\odot_{II}(z;y)}_{H^{k_*}}
& \le &
K (1 + \norm{y}_{L^2} + \norm{z}_{L^2} )
    \big( 1 + \norm{y}_{H^{k_*+1}} + \norm{z}_{H^{k_*+1}} \big)
    \norm{z}_{H^{k_*}}
\\[0.2cm]
& & \qquad + K (1 + \norm{y}_{L^2} + \norm{z}_{L^2} )
    \norm{z}_{H^{k_* + 1}},
\\[0.2cm]
    \norm{\mathcal{S}^\odot(z;y)}_{HS(L^2_Q;H^{k_*})}
    & \le & K (1 + \norm{y}_{H^{k_*}}^{k_*} +
    \norm{z}_{H^{k_*}}^{k_*}
    +
    \norm{y}_{H^{k+1}} +  \norm{z}_{H^{k+1}})\norm{z}_{H^{k_*}}
     + K \norm{z}_{H^{k_* + 1}} \big).
\end{array}
\end{equation}
\end{lem}
\begin{proof}
Upon using the integral representation
\begin{equation}
  \mathcal{R}^\odot_{I}(z;y) =   \mathcal{R}_I( y + z) - \mathcal{R}_I( y) = \int_0^1 D \mathcal{R}_I( y + tz)[z] \, \mathrm dt,
\end{equation}
the desired bound for $\mathcal{R}^\odot_{I}$ follows from
\eqref{eq:nl:bnd:d:r:i:compact} by taking $v = z$ and replacing each occurrence of $\norm{w}$ by $\norm{y} + \norm{z}$.
The remaining bounds can be obtained in the same fashion,
now using \eqref{eq:nl:bnd:d:j:r:ii}, \eqref{eq:nl:bnd:dj:ups:ii}
    and \eqref{eq:nl:bnd:d:j:mathcal:s}.
\end{proof}

Inspecting the remainder definitions \eqref{eq:tay:def:rem:fncs}
and the evolution \eqref{eq:mr:weak:form:evol:v} for $V$
and dropping the $(\sigma,\delta)$ dependencies of $Y_{\rm tay}$, $N_{\rm rem}$ and $B_{\rm rem}$,
we see that the $H^{k_*}$-valued identity
\begin{equation}
\label{eq:res:evo:y:first}
\begin{array}{lcl}
    Z(t) &= & \int_0^t \big[ \Delta_{x_\perp}  +  \mathcal{L}_{\rm tw}  ] Z(s) \, \mathrm ds
\\[0.2cm]
& & \qquad
    + \int_0^t [\mathcal{R}^\odot_I\big(Z(s);Y_{\rm tay}(s)\big)  + \sigma^2 \mathcal{R}^\odot_{II}\big(Z(s);Y_{\rm tay}(s)\big) + \sigma^2 \Upsilon^\odot\big(Z(s);Y_{\rm tay}(s)\big) + N_{\rm rem}(s) ] \, \mathrm ds
\\[0.2cm]
& & \qquad
    + \int_0^t [ \sigma \mathcal{S}^\odot\big(Z(s);Y_{\rm tay}(s)\big) + B_{\rm rem}(s)] \, \mathrm d W^Q_s
\end{array}
\end{equation}
holds $\mathbb P$-a.s. for all $0 \le t < \tau_\infty$.
We now need to isolate the second derivative of $Z$ that is contained in $\Upsilon^\odot$. To this end, we introduce the expression
\begin{equation}
\label{eq:nl:def:e:upsilon}
\begin{array}{lcl}
    \mathcal{E}_{\Upsilon}(z;y)
    & = &
    [\tilde{\nu}(\Phi_0 + y + z,0) - \tilde{\nu}(\Phi_0 + y,0)] \partial_{xx} [ \Phi_0 + y]
    + \Upsilon_{II}^\odot(z;y)
\end{array}
\end{equation}
and inspect the definitions \eqref{eq:list:def:ups:i:ii} to see that
\begin{equation}
    \Upsilon^\odot(z;y) = \tilde{\nu}(\Phi_0 + y + z, 0) \partial_{xx} z + \mathcal{E}_\Upsilon(z;y) .
\end{equation}
We now write
\begin{equation}
    \kappa_\sigma(z;y) = 1 + \sigma^2 \tilde{\nu}(\Phi_0 +y  +z,0 )
\end{equation}
to represent the full coefficient in front of $Z_{xx}$ in \eqref{eq:res:evo:y:first}.
In addition, we introduce the expression
\begin{equation}
\label{eq:res:def:e:l}
    \mathcal{E}_{\mathcal{L}}(z;y)
    = -  \tilde{\nu}(\Phi_0 + y + z,0)
    [ \mathcal{L}_{\rm tw} - \partial_{xx} ] z
\end{equation}
and record the following useful estimates.
\begin{lem}
Suppose that \textnormal{(HNL)}, \textnormal{(HTw)} and \textnormal{(Hq)} are satisfied. Then there exists $K > 0$
so that for any $y \in H^{k_* + 2}$ and $z \in H^{k_*+1}$ we have the bounds
\begin{equation}
\label{eq:res:bnd:e:ups:l}
\begin{array}{lcl}
    \norm{\mathcal{E}_\Upsilon(z;y)}_{H^{k_*}}
    & \le & K (1 + \norm{y}_{H^{k_*+2}}   ) \norm{z}_{H^{k_*}}
    + \norm{\Upsilon^\odot_{II}(z;y)}_{H^{k_*}}
\\[0.2cm]
\norm{\mathcal{E}_{\mathcal{L}}(z;y}_{H^{k_*}}
& \le & K \norm{z}_{H^{k_* + 1}}.
\end{array}
\end{equation}
\end{lem}
\begin{proof}
Inspecting the definition  \eqref{eq:nl:def:e:upsilon}, the bound
for $\mathcal{E}_\Upsilon(z;w)$
follows from \eqref{eq:nl:est:deriv:b:nut:wtkc}.
    The bound for $\mathcal{E}_{\mathcal{L}}$
    follows from \eqref{eq:res:def:e:l}
    by applying the uniform bound \eqref{eq:nl:est:deriv:b:nut:wtkc}
    for $\widetilde{\nu}$.
\end{proof}

Inspecting the definitions above, one may readily verify the useful identity
\begin{equation}
\label{eq:res:id:for:e:ups:plus:l}
    \mathcal{E}_{\Upsilon}(z;y) + \mathcal{E}_{\mathcal{L}}(z;y)
    = \Upsilon^\odot(z;y)- \tilde{\nu}(\Phi_0 + y + z,0)\mathcal{L}_{\mathrm{tw}} z .
\end{equation}
This allows us to reformulate \eqref{eq:res:evo:y:first} as
\begin{equation}
\begin{array}{lcl}
    Z(t) & = & \int_0^t \big[ \Delta_{x_\perp}  +  \kappa_\sigma\big(Z(s);Y_{\rm tay}(s)\big) \mathcal{L}_{\rm tw}
        \big] Z(s) \, \mathrm ds
\\[0.2cm]
& & \qquad
    + \int_0^t [\mathcal{R}^\odot_I\big(Z(s);Y_{\rm tay}(s)\big)  + \sigma^2 \mathcal{R}^\odot_{II}\big(Z(s);Y_{\rm tay}(s)\big)  + N_{\rm rem}(s) ] \, \mathrm ds
\\[0.2cm]
& & \qquad
  + \sigma^2 \int_0^t [ \mathcal{E}_{\Upsilon}\big(Z(s);Y_{\rm tay}(s)\big) + \mathcal{E}_{\mathcal{L}}\big(Z(s);Y_{\rm tay}(s)\big) ] \, \mathrm ds
\\[0.2cm]
& & \qquad
    + \int_0^t [ \sigma \mathcal{S}^\odot\big(Z(s);Y_{\rm tay}(s)\big) + B_{\rm rem}(s)] \, \mathrm  d W^Q_s ,
\end{array}
\end{equation}
where all second derivatives of $Z$ are now contained in the first line.

We now proceed with a time transformation to set the coefficient in front of $\mathcal{L}_{\rm tw}$ to unity.
In particular, we write
\begin{equation}
    \tau(t) = \int_0^t \kappa_\sigma\big(Z(s);Y_{\rm tay}(s) \big) \, \mathrm  ds
\end{equation}
and introduce the time-transformed functions
$(\bar{Z}, \bar{Y}_{\rm tay}, \bar{N}_{\rm rem}, \bar{B}_{\rm rem})$ that satisfy
\begin{equation}
    \bar{Z}\big(\tau(t) \big) = Z(t ), \qquad
    \bar{Y}_{\rm tay}\big(\tau(t)\big) = Y_{\rm tay}(t),
    \qquad \bar{N}_{\rm rem}\big(\tau(t)\big)
    = N_{\rm rem}(t),
    \qquad
    \bar{B}_{\rm rem}\big(\tau(t)\big)
    = B_{\rm rem}(t).
\end{equation}
In addition, we write $\bar{\tau}_{\rm tay}(\eta) = \tau( t_{\rm tay}(\eta) )$.

Applying standard time transformation rules \cite[Lem. 6.2]{hamster2019stability}
now leads to
the system
\begin{equation}
\label{eq:res:evo:bar:y:pre:mild}
\begin{array}{lcl}
    \bar{Z}(\tau) & = & \int_0^t \big[ \kappa_\sigma\big(\bar{Z}(\tau');\bar{Y}_{\rm tay}(\tau')\big)^{-1} \Delta_{x_\perp}  +   \mathcal{L}_{\rm tw}
        \big] \bar{Z}(\tau') \, \mathrm d \tau'
\\[0.2cm]
& & \qquad
    + \int_0^\tau [\mathcal{N}(\bar{Z}(\tau');\tau',\sigma,\delta)  + \bar{N}_{\rm rem}(\tau') ] \, \mathrm d \tau'
\\[0.2cm]
& & \qquad
    + \int_0^\tau
    \big[ \mathcal{M}\big(\bar{Z}(\tau');\tau',\sigma,\delta)
    + \bar{B}_{\rm rem}(\tau') \big]
    \, \mathrm d \bar{W}^Q_{\tau'} ,
\end{array}
\end{equation}
in which we have introduced the expressions
\begin{equation}
\label{eq:res:def:n:m}
\begin{array}{lcl}
\mathcal{N}\big(z;\tau,\sigma,\delta \big) & = &
\kappa_\sigma^{-1}\big(z; \bar{Y}_{\rm tay}(\tau)\big)\big[\mathcal{R}_I^\odot\big(z;\bar{Y}_{\rm tay}(\tau) \big) + \sigma^2 \mathcal{R}^\odot_{II}\big(z;\bar{Y}_{\rm tay}(\tau) \big) 
\\[0.2cm]
& & \qquad \qquad
+ \sigma^2 \mathcal{E}_\Upsilon\big(z;\bar{Y}_{\rm tay}(\tau) \big) + \sigma^2  \mathcal{E}_{\mathcal{L}}\big(z;\bar{Y}_{\rm tay}(\tau) \big) \big]
\\[0.2cm]
\mathcal{M}\big(z;\tau,\sigma,\delta \big) & = &
\kappa_\sigma^{-1/2 }\big(z; \bar{Y}_{\rm tay}(\tau)\big)
\big[ \sigma \mathcal{S}^\odot\big(z;\bar{Y}_{\rm tay}(\tau) \big) ]. 
\end{array}
\end{equation}
We note that the time-transformed process
$\bar W_\tau^Q$ is again a $Q$-cylindrical Wiener process,
but now adapted to the filtration
$\bar{\mathbb F}=(\bar{\mathcal F}_\tau)_{\tau\geq 0}$ given by
\begin{equation}\bar{\mathcal F}_\tau=\{A\in\mathcal F:A\cap \{\tau\leq \tau(t)\}\in\mathcal F_t\,  \text{ for all }t\geq 0\}.\label{eq:filtration}
\end{equation}
It has the same statistical properties as $W^Q_t$.
The full details can be found in \cite[{\S}6.2]{bosch2024multidimensional}
and \cite{hamster2019stability}.

\begin{prop}
\label{prop:res:bnds}
Suppose that \textnormal{(HNL)}, \textnormal{(HTw)}, \textnormal{(H$V_*$)} and \textnormal{(Hq)} all hold. Then there exists $K >0$
so that for all $z \in H^{k_* + 1}$, all $\tau \ge 0$, all $\sigma \ge 0$ and all $\delta \ge 0$ we have the bounds
\begin{equation}
\label{eq:res:bnd:n:m:full}
\begin{array}{lcl}
\norm{\mathcal{N}(z;\tau,\sigma,\delta)}_{H^{k_*}}
& \le & K\big[ 1 + \norm{\bar{Y}_{\rm tay}(\tau)}_{H^{k_* + 2}}^{k_*+1}
+ \norm{z}_{H^{k_*}}^{k_*} \big]
\big[\sigma^2 +  \norm{\bar{Y}_{\rm tay}(\tau)}_{H^{k_*}} +  \norm{z}_{H^{k_*}}    \big] \norm{z}_{H^{k_* + 1}} ,
\\[0.2cm]
\norm{\mathcal{M}(z;\tau,\sigma,\delta)}_{HS(L^2_Q;H^{k_*})}
& \le & K \sigma \big[1 +  \norm{\bar{Y}_{\rm tay}(\tau)}_{H^{k_* + 1}}^{k_*+1}
+ \norm{z}_{H^{k_*}}^{k_*} \big] \norm{z}_{H^{k_* + 1}} .
\end{array}
\end{equation}
In particular, there exists $K > 0$ and $\eta_0$
so that for any $0 < \eta \le \eta_0$,
any $0\le \tau \le \bar{\tau}_{\rm tay}(\eta)$,
any $\sigma \ge 0$, any $\delta \ge 0$
and any $z \in H^{k_*+1}$ with $\norm{z}_{H^{k_*}}^2  \le \eta$
we have
\begin{equation}
\label{eq:res:bnd:n:m:after:stop}
\begin{array}{lcl}
\norm{\mathcal{N}(z;\tau,\sigma,\delta)}_{H^{k_*}}
& \le & K
\big[\sigma^2 +  \sqrt{\eta}     \big] \norm{z}_{H^{k_* + 1}} ,
\\[0.2cm]
\norm{\mathcal{M}(z;\tau,\sigma,\delta)}_{HS(L^2_Q;H^{k_*})}
& \le & K \sigma \norm{z}_{H^{k_* + 1}} ,
\end{array}
\end{equation}
together with
\begin{equation}
\label{eq:res:proj:n:m}
    P^\perp \mathcal{N}(z;\tau,\sigma,\delta) =
    \mathcal{N}(z;\tau,\sigma,\delta),
    \qquad \qquad
    P^\perp \mathcal{M}(z;\tau,\sigma,\delta) =
    \mathcal{M}(z;\tau,\sigma,\delta).
\end{equation}
\end{prop}
\begin{proof} 
Note first that $\kappa_{\sigma} \ge 1$.
The bounds \eqref{eq:res:bnd:n:m:full} follow by inspecting
the definitions \eqref{eq:res:def:n:m}
and applying \eqref{eq:res:est:r:odot:ups:ii:odot}
and \eqref{eq:res:bnd:e:ups:l}.
The bounds \eqref{eq:res:bnd:n:m:after:stop}
follow by using the pathwise bound \eqref{eq:tay:bnd:w:unif}.
The identities \eqref{eq:res:proj:n:m} follow
from the properties of the cut-off $\chi_l$ and the representations
\eqref{eq:sm:reminder:r:i},
\eqref{eq:sm:repr:r:ii}, \eqref{eq:sm:def:S}
and \eqref{eq:res:id:for:e:ups:plus:l},
recalling for the latter that $\mathcal{L}_{\rm tw}^* \psi_{\rm tw} = 0$.
\end{proof}

Write $E(t,s)$ for the evolution family associated to
\eqref{eq:conv:def:evol:probl} with $\nu = \kappa_\sigma^{-1}(\bar{Z};\bar{Y}_{\rm tay})$.
We note that when $\sigma \ge 0$ is sufficiently small condition
(hE) in {\S}\ref{sec:conv} is satisfied on account of the bound \eqref{eq:nl:est:deriv:b:nut:wtkc},
which allows us to write
\begin{equation}
\label{eq:res:bnd:kappa:sigma:2}
    1 \le \kappa_\sigma(\bar Z; \bar Y_{\rm tay}) \le 1 + \sigma^2 K \le 2.
\end{equation}
This evolution family can now be used to
transform \eqref{eq:res:evo:bar:y:pre:mild}
into a mild representation
for the residual $\bar{Z}$.

\begin{prop} 
\label{prop:res:mild}
Suppose that \textnormal{(HNL)}, \textnormal{(HTw)}, \textnormal{(H$V_*$)} and \textnormal{(Hq)} hold. Fix $T > 0$ together with
a sufficiently small $\sigma \ge 0$. 
Then for any $\delta \ge 0$
there exists an increasing sequence of stopping times $(\bar{\tau}_\ell)_{\ell\geq 0}$ and a stopping time $\bar{\tau}_\infty $, with  $\bar{\tau}_\ell\to \bar{\tau}_\infty$ and $0<\bar{\tau}_\infty\leq T$ $\mathbb P$-a.s., together with a  map
\begin{equation}
    \bar{Z}:[0,T]\times \Omega\to H^{k_*}
\end{equation}
that is progressively measurable with respect to the  filtration $\bar{ \mathbb F}=(\bar{\mathcal F}_\tau)_{\tau\geq 0}$
and satisfies the following properties:
    \begin{itemize}
          \item[(i)] For almost every $\omega\in\Omega$, the map
    $
        \tau\mapsto \bar{Z}(\tau,\omega)
    $
    is of class $C([0,\bar{\tau}_\infty(\omega));H^{k_*})$;
    
    \item[(ii)] For any $\ell \ge 0$ we have the integrability condition $\bar{Z}\in L^2\big(\Omega;L^2([0,\bar{\tau}_\ell];H^{k_*+1})\big)$;
   \item[(iii)] For any $\ell \ge 0$
   we have $
            \mathcal M(\bar{Z}(\cdot);\cdot,\sigma,\delta)\in L^2\big(\Omega; L^2([0,\bar{\tau}_\ell];HS( L^2_Q,H^{k_*})\big)$,
            together with \\$\mathcal N(\bar{Z}(\cdot,\omega);\cdot,\sigma,\delta)\in L^1([0,\bar{\tau}_\ell(\omega)];H^{k_*})$ for almost every $\omega\in\Omega$;
    \item[(iv)] The $H^{k_*}$-valued identity
        \begin{equation}
        \begin{array}{lcl}
        \bar Z(\tau) & = &\int_0^\tau E(\tau,\tau')
        \big[ \mathcal N(\bar Z(\tau');\tau', \sigma,\delta)
        + \bar{N}_{\rm rem}[\sigma,\delta](\tau') \big]
        \mathrm d \tau'
        \\[0.2cm]
        & & \qquad + \int_0^\tau E(\tau,\tau')
        \big[ \mathcal M(\bar Z(\tau');\tau',\sigma,\delta)
          + \bar{B}_{\rm rem}[\sigma,\delta](\tau')
        \big]
        \mathrm d \bar{W}_{\tau'}^{Q;-}
        \label{eq:res:mild:form:Y}
        \end{array}
    \end{equation}
    holds $\mathbb P$-a.s. for all $0\leq t< \bar{\tau}_\infty$;
    \item[(v)]
       Upon writing
       \begin{equation}
           \mathcal{Z}_{\ell} = \sup_{0 \le \tau' \le \tau_{\ell}} \norm{\bar{Z}(\tau')}_{H^{k_*}}^2
           + \int_0^{\tau_{\ell}} \norm{\bar{Z}(\tau')}^2_{H^{k_*+1}} \, \mathrm  d \tau',
       \end{equation}
       we have $\mathcal{Z}_{\ell} \le \ell$ for every $\ell \ge 0$, together with the localization identity
       \begin{equation}
       \label{eq:res:prob:reverse:zero}
        \mathbb{P} \Big( \bar{\tau}_\infty < T \hbox{ and } \sup_{\ell \ge 0} \mathcal{Z}_{\ell} < \infty \Big) = 0.
       \end{equation}
    \end{itemize}
  \end{prop}
\begin{proof}
    This can be established by following the proofs of \cite[Prop 6.1, Prop 6.3, Prop 6.4]{bosch2024multidimensional}, which rely on the framework developed in \cite{agresti2024criticalNEW}.
\end{proof}

\section{Nonlinear stability}
\label{sec:stb}

In this section we set out to establish Theorem \ref{thm:main},
establishing that our decomposition of the perturbation $V$ can be maintained over timescales that are exponentially long with respect to $\sigma^{-1}$.
Our main objective is to control the size of
\begin{equation}
\label{eq:stb:def:n:eps:res}
\bar{\mathcal N}_{\rm res}(\tau)=\norm{\bar{Z}(\tau)}_{H^{k_*}}^2+\int_0^\tau e^{-\beta (\tau-\tau')} \norm{\bar{Z}(\tau')}_{H^{k_*+1}}^2\mathrm d \tau',
\end{equation}
using the time-transformed mild  representation in \eqref{eq:res:mild:form:Y}.

 In particular, we recall the shorthand \eqref{eq:mr:def:alpha}
 and introduce the stopping time
\begin{equation}
\label{eq:res:def:tau:res}
    \bar{\tau}_{\rm res}(\eta; \sigma,\delta, T)=\inf\{0 \le \tau < \min \{ \bar{\tau}_\infty, \bar{\tau}_{\rm tay}(\eta) \}:
  \bar{\mathcal N}_{\rm res}(\tau) >\eta \min\{1, \alpha^{2(r-1)}\} \},
\end{equation}
writing
\begin{equation}
\bar{\tau}_{\rm res}(\eta;\sigma,\delta,T) = \min \{ \bar{\tau}_\infty, \bar{\tau}_{\rm tay}(\eta) \}
\end{equation}
if the set is empty. We emphasize that when $\eta$ is sufficiently small,
the bounds \eqref{eq:res:bnd:n:m:after:stop} and identities
\eqref{eq:res:proj:n:m} hold for $z = \bar{Z}(\tau)$
whenever $0 \le \tau < \bar{\tau}_{\rm res}(\eta;\sigma,\delta,T)$.

Our main result here provides logarithmic growth bounds
for the expectation of the maximal value that $\bar{\mathcal N}_{\rm res}$ attains as we increase $T$. We establish this result in {\S}\ref{subsec:st:mom:bnd}
and use it to prove Theorem \ref{thm:main} in {\S}\ref{subsec:st:prf:mr}.

\begin{prop}\label{prop:E_stab}
Suppose that \textnormal{(HNL)}, \textnormal{(HTw)}, \textnormal{(H$V_*$)} and \textnormal{(Hq)} are satisfied.
Pick two sufficiently small constants $\delta_\eta>0$ and $\delta_\sigma>0$. Then there exists a constant $K>0$ so that for any $0<\eta<\delta_\eta,$ any $0\leq \sigma\leq \delta_\sigma$, any $\delta \ge 0$, any integer $T\geq 2$  and any integer $p\geq 1$, we have the bound
 \begin{equation}
 \label{eq:prp:stb:bnd:sup:res:p}
     \mathbb E\left[\sup_{0\leq \tau < \bar{\tau}_{\rm res}(\eta;\sigma,\delta, T)} \bar{\mathcal N}_{\rm res}(\tau)^p\right]\leq K^p
     \big( (\sigma^2 p)^{pr} + [\sigma^2 \ln T + \delta^2]^{pr} \big).
 \end{equation}
\end{prop}

\subsection{Proof of Proposition \ref{prop:E_stab}}
\label{subsec:st:mom:bnd}

Following the approach in \cite{hamster2019stability},
we proceed by providing separate estimates for the integrals in
\eqref{eq:res:mild:form:Y}.
To this end, we introduce the integrals
\begin{equation}
\begin{aligned}
 \mathcal{E}^d_{\rm rem}(\tau)&=\int_0^{\tau} E(\tau,\tau') P^\perp \bar{N}_{\rm rem}[\sigma,\delta](\tau') 
 \mathbf{1}_{\tau'<\bar{\tau}_{\rm res}(\eta;\sigma,\delta, T)}
 \, \mathrm d \tau', \\
 \mathcal{E}^d_{\mathcal N}(\tau)&=\int_0^{\tau} E(\tau,\tau') P^\perp \mathcal{N}(\bar{Z}(\tau');\tau', \sigma,\delta ) \mathbf{1}_{\tau'<\bar{\tau}_{\rm res}(\eta;\sigma,\delta, T)}\, \mathrm d \tau',
\\
 \mathcal{E}^s_{\rm rem}(\tau)&=\int_0^{\tau} E(\tau,\tau') P^\perp \bar{B}_{\rm rem}[\sigma,\delta](\tau') 
 \mathbf{1}_{\tau'<\bar{\tau}_{\rm res}(\eta;\sigma,\delta, T)}
 \, \mathrm d \bar{W}_{\tau'}^{Q;-}
\\
\mathcal{E}^s_{\mathcal M}(\tau)&=\int_0^{\tau} E(\tau,\tau') P^\perp \mathcal{M}(\bar{Z}(\tau');\tau', \sigma,\delta) \mathbf{1}_{\tau'<\bar{\tau}_{\rm res}(\eta;\sigma,\delta, T)} \,\mathrm d \bar{W}_{\tau'}^{Q;-}.
\end{aligned}
\end{equation}
The presence of the projection $P^\perp$ in the above is simply to emphasise \eqref{eq:res:proj:n:m}. Using these expressions, we obtain the estimate
     \begin{equation}\begin{aligned}
    \label{eq:st:a:priori:bnd:e}
        &\mathbb E\sup_{0\leq \tau < \bar{\tau}_{\rm res}(\eta;\sigma,\delta,T)} \norm{\bar Z(\tau)}_{H^{k_*}}^{2p}\\&\qquad\quad\leq 4^{2p}\mathbb E\sup_{0\leq \tau\leq T}\left[\|\mathcal E^d_{\rm rem}(\tau)\|_{H^{k_*}}^{2p}+\|\mathcal E^d_{\mathcal N}(\tau)\|_{H^{k_*}}^{2p}+\|\mathcal E^s_{ \rm rem}(\tau)\|_{H^{k_*}}^{2p}
        + \|\mathcal E^s_{\mathcal M}(\tau)\|_{H^{k_*}}^{2p}
        \right].
        \end{aligned}
    \end{equation}
Turning to the integrated $H^{k_*+1}$-bound, we introduce the integrals
\begin{align}
\label{eq:st:a:priori:bnd:i}
 \mathcal{I}^d_{ \rm rem}(\tau)&=\int_0^\tau e^{-\beta (\tau-\tau')}\| \mathcal{E}^d_{\rm rem}(\tau')\|_{H^{k_*+1}}^2 \mathrm d \tau', \\
 \mathcal{I}^d_{ \mathcal N}(\tau)&=\int_0^\tau e^{-\beta(\tau- \tau')}\|\mathcal{E}^d_{\mathcal N}(\tau')\|_{H^{k_*+1}}^2 \mathrm d  \tau', \\
   \mathcal{I}^s_{ \rm rem}(\tau)&=\int_0^\tau e^{-\beta(\tau- \tau')}\| \mathcal{E}^s_{\rm rem}(\tau')\|_{H^{k_*+1}}^2 \mathrm d  \tau', \\
   \mathcal{I}^s_{\mathcal M}(\tau)&=\int_0^\tau e^{-\beta(\tau- \tau')}\| \mathcal{E}^s_{\mathcal M}(\tau')\|_{H^{k_*+1}}^2 \mathrm d \tau'.
\end{align}
This leads directly to the estimate
  \begin{equation}
        \begin{aligned}
            &\mathbb E \sup_{0\leq \tau < \bar{\tau}_{\rm res}(\eta;\sigma,\delta,T)}\left[\int_0^t e^{-\beta(\tau-\tau')}\norm{\bar{Z}(\tau')}_{H^{k+1}}^2\mathrm d \tau'\right]^p \\&\qquad\quad\leq 4^{2p}\mathbb E \sup_{0\leq \tau\leq T}\Big[\mathcal I^d_{\rm rem}(\tau)^p
            + \mathcal I^d_{\mathcal N}(\tau)^p
+ \mathcal I^s_{\rm rem}(\tau)^p
+\mathcal I^s_{\mathcal M}(\tau)^p\Big].\\
        \end{aligned}
        \end{equation}

\begin{lem}\label{lem:stb:n:path}
Suppose that \textnormal{(HNL)}, \textnormal{(HTw)}, \textnormal{(H$V_*$)} and \textnormal{(Hq)} are satisfied, recall the constant
$\eta_0 >0$ defined in Proposition \ref{prop:res:bnds} and pick a sufficiently small $\delta_\sigma > 0$.
Then there exists a constant $K> 0$
so that for any $0<\eta<\eta_0$,
any $0 \le \sigma < \delta_\sigma$, any $\delta \ge 0$,
any $T>0$
and any $p \ge 1$,
we have the pathwise bound
\begin{equation}\begin{aligned}
             \sup_{0\leq t\leq T}\|\mathcal E^d_{\mathcal N}(t)\|_{H^k}^{2p}
             +
             \sup_{0\leq t\leq T} \mathcal I^d_{\mathcal N}(t)^p
             &\leq K^{2p} (\sigma^2 + \sqrt{\eta})^{2p} \sup_{0\leq \tau < \bar{\tau}_{\rm res}(\eta;\sigma,\delta,T)} \bar{N}_{\rm res}(\tau)^p.
    \end{aligned}\end{equation}
\end{lem}
\begin{proof}
This bound follows readily from straightforward integral estimates;
see \cite[Lem. 5.3]{hamster2020expstability}.
\end{proof}

\begin{lem}\label{lem:E0ElinEnl}
Suppose that \textnormal{(HNL)}, \textnormal{(HTw)}, \textnormal{(H$V_*$)} and \textnormal{(Hq)} are satisfied, recall the constant
$\eta_0 >0$ defined in Proposition \ref{prop:res:bnds} and pick a sufficiently small $\delta_\sigma > 0$.
Then there exists a constant $K> 0$
so that for any $0<\eta<\eta_0$,
any $0 \le \sigma <\delta_\sigma$, any $\delta \ge 0$,
any $T\ge 2$ and
any real $p \ge 1$,
we have the bound
    \begin{equation}\begin{aligned}
             \sup_{0\leq \tau \leq T}\|\mathcal E^d_{\rm rem}(\tau)\|_{H^{k_*}}^{2p}
             + \sup_{0\leq \tau\leq T} \mathcal I^d_{\rm rem}(\tau)^p
             &\leq K^{2p} \big[ (\sigma^2 p)^{pr} + [\sigma^2 \ln T + \delta^2]^{rp}].
    \end{aligned}\end{equation}
\end{lem}
\begin{proof}
Assuming without loss that $\sigma > 0$,
the bound \eqref{eq:tay:bnds:n:b:rem}
implies that (hN) is satisfied with
\begin{equation}
\Theta_1 = \sigma^{r} K,
\qquad \Theta_2 = \sigma^{-2}[ \sigma^2 \ln T + \delta^2],
\qquad n = r.
\end{equation}
The desired estimate now follows from an application of
Proposition \ref{prp:cnv:hn}.
\end{proof}

\begin{lem}\label{lem:EB}
Suppose that \textnormal{(HNL)}, \textnormal{(HTw)}, \textnormal{(H$V_*$)} and \textnormal{(Hq)} are satisfied, recall the constant
$\eta_0 >0$ defined in Proposition \ref{prop:res:bnds} and pick a sufficiently small $\delta_\sigma > 0$.
Then there exists a constant $K > 0$ so that
for any $0<\eta<\eta_0$, any $0 \le \sigma < \delta_\sigma$,
any $\delta \ge 0$, any integer $T \ge 2$ and any integer $p \ge 1$, we have the bound
    \begin{equation}
        \mathbb E\sup_{0\leq \tau \leq T}\|\mathcal E^s_{\mathcal M}(\tau)\|_{H^{k_*}}^{2p}
        + \mathbb E \sup_{0\leq \tau\leq T}\mathcal I^s_{\mathcal M}(\tau)^p
        \leq
        K^{2p} \big[ (\sigma^2p)^{pr} + [ \sigma^2 \ln T + \delta^2]^{rp} \big].
    \end{equation}
\end{lem}
\begin{proof}
We first note that \eqref{eq:res:bnd:n:m:after:stop}
and the stopping time definition \eqref{eq:res:def:tau:res}
imply the pathwise bound
\begin{equation}
    \int_0^\tau e^{-\beta(\tau-\tau')}\norm{\mathcal{M}(\bar{Z}(\tau'), \tau', \sigma, \delta)}_{HS(L^2_Q;H^{k_*})}^2 \textbf{1}_{\tau' < \bar{\tau}_{\rm res}(\eta;\sigma,\delta,T)} \, \mathrm d \tau' \leq K^2 \sigma^2  \eta \alpha^{2(r-1)}
    \end{equation}
for all $0 \le \tau \le T$. In addition, we may exploit the smoothening
properties of the semigroup $E_{\rm tw}(t) = e^{\frac{1}{4} t \Delta_{x_\perp}} e^{\mathcal{L}_{\rm tw} t} $
to show that there exists $C_1 > 0$ so that
\begin{equation}
   \norm{E_{\rm tw}(1) \mathcal{M}(\bar{Z}(\tau), \tau, \sigma, \delta)}_{HS(L^2_Q;H^{k_*})}^2
    \leq \sigma^2 C_1 \eta \alpha^{2(r-1)}
\end{equation}
for all $0 \le \tau \le \bar{\tau}_{\rm res}(\eta;\sigma,\delta,T)$,
since $\norm{\bar{Z}(\tau)}_{H^{k_*}}^2 \le \eta \alpha^{2(r-1)}$ on this interval.
We may hence conclude that condition (HB) in \cite{bosch2024multidimensional} is satisfied with $\Theta_* = C_2 \sigma \sqrt{\eta} \alpha^{r-1}$.
Applying \cite[Prop 3.18]{bosch2024multidimensional}
we find that there exist $C_3 > 0$ so that
\begin{equation}
\begin{array}{lcl}
        \mathbb E\sup_{0\leq \tau \leq T}\|\mathcal E^s_{\mathcal M}(\tau)\|_{H^{k_*}}^{2p}
        + \mathbb E \sup_{0\leq \tau\leq T}\mathcal I^s_{\mathcal M}(\tau)^p
        & \leq &
        C_3^{2p} \sigma^{2p} \eta^p
        (p^p+[\ln T]^p) \big( \sigma^2 \ln T + \delta^2 \big)^{p(r-1)}
        \\[0.2cm]
        & \leq &
        C_3^{2p} \eta^p
        \big( \sigma^2 p+[\sigma^2 \ln T] + \delta^2  \big)^{pr},
    \\[0.2cm]
\end{array}
    \end{equation}
which provides the desired bound.
\end{proof}

\begin{lem}\label{lem:stb:b:rem}
Suppose that \textnormal{(HNL)}, \textnormal{(HTw)}, \textnormal{(H$V_*$)} and \textnormal{(Hq)} are satisfied, recall the constant
$\eta_0 >0$ defined in Proposition \ref{prop:res:bnds} and pick a sufficiently small $\delta_\sigma > 0$.
Then there exists a constant $K> 0$
so that for any $0<\eta<\eta_0$,
any $0 \le \sigma < \delta_\sigma$,
any $\delta \ge 0$, any integer $T \ge 3$
and any real $p \ge 1$,
we have the bound
   \begin{equation}
        \mathbb E\sup_{0\leq \tau\leq T}\|\mathcal E_{B;\rm c}(\tau)\|_{H^{k_*}}^{2p}
        + \mathbb E \sup_{0\leq \tau\leq T}\mathcal I_{B;\rm c}(\tau)^p
        \leq
        K^{2p} \big[ (\sigma^2 p)^{pr} + [\sigma^2 \ln T + \delta^2]^{rp}].
    \end{equation}
\end{lem}
\begin{proof}
Assuming without loss that $\sigma > 0$, the bound
\eqref{eq:tay:bnds:n:b:rem}
implies that $\bar{B}_{\rm rem}/\sigma$
satisfies (hB) with
\begin{equation}
        \Theta_1 = \sigma^{r-1} K,
         \qquad
         \Theta_2 = \sigma^{-2}[ \sigma^2 \ln T + \delta^2],
         \qquad n= r.
    \end{equation}
An application of Proposition \ref{prp:cnv:hb}
now yields the desired estimate.
\end{proof}

\begin{proof}[Proof of Proposition \ref{prop:E_stab}]
Collecting the results in Lemmas \ref{lem:stb:n:path}--\ref{lem:stb:b:rem}, the estimates \eqref{eq:st:a:priori:bnd:e}
and \eqref{eq:st:a:priori:bnd:i} can be combined to yield
\begin{equation}\begin{aligned}
            &\mathbb E\left[\sup_{0\leq \tau < \bar{\tau}_{\rm res}(\eta;\sigma,\delta,T)} \bar{\mathcal N}_{\rm res}(\tau)^p\right]\\&\qquad\quad\leq K^p\Bigg(
            (\sigma^{2}p)^{pr}+[\sigma^2 \ln T + \delta^2]^{pr}
            +(\sigma^{2}+\sqrt{\eta})^{2p}\mathbb E
            \left[\sup_{0\leq \tau < \bar{\tau}_{\rm res}(\eta;\sigma,\delta,T)} \bar{\mathcal N}_{\rm res}(\tau)^p\right]
            \Bigg).
            \end{aligned}
        \end{equation}
        The result hence readily follows by
        restricting the size of $\sigma^2+\sqrt{\eta}$ .
\end{proof}

\subsection{Proof of Theorem \ref{thm:main}}
\label{subsec:st:prf:mr}

In order to establish our main result, we need to
include the growth of the expansion functions $Y_j$
and reverse the effects of the time transformation.
Starting with the former, we write
\begin{equation}
    \bar{\mathcal N}_{\rm full}(\tau;\sigma,\delta,T) =
    \alpha^{2(r-1)} \norm{\bar{Y}_1(\tau)}_{1}^2 + \ldots + \alpha^2 \norm{\bar{Y}_{r-1}(\tau)}_{r-1}^2     + \bar{\mathcal{N}}_{\rm res}(\tau)
\end{equation}
and introduce the stopping time
\begin{equation}
\label{eq:res:def:tau:full}
    \bar{\tau}_{\rm full}(\eta; \sigma,\delta, T)=\inf\{0 \le \tau <  \bar{\tau}_\infty :
   \bar{\mathcal N}_{\rm full}(\tau;\sigma,\delta,T) >\eta \min\{1,  \alpha^{2(r-1)}\} \},
\end{equation}
writing
\begin{equation}
    \bar{\tau}_{\rm full}(\eta; \sigma,\delta, T) = \bar{\tau}_\infty
\end{equation}
if the set is empty.
By construction we have
$\bar{\tau}_{\rm full}(\eta;\sigma,\delta,T)\le \bar{\tau}_{\rm tay}(\eta)$,
which implies the ordering
\begin{equation}
\label{eq:stb:ord:stopping:time}
 \bar{\tau}_{\rm full}(\eta;\sigma,\delta,T)\le
 \bar{\tau}_{\rm res}(\eta;\sigma,\delta,T) \le
 \min \{ \bar{\tau}_\infty, \bar{\tau}_{\rm tay}(\eta) \}.
\end{equation}

\begin{cor}
Suppose that \textnormal{(HNL)}, \textnormal{(HTw)}, \textnormal{(H$V_*$)} and \textnormal{(Hq)} are satisfied.
Pick two sufficiently small constants $\delta_\eta>0$ and $\delta_\sigma>0$. Then there exists a constant $K>0$ so that for any  any $0<\eta<\delta_\eta,$ any $0\leq \sigma\leq \delta_\sigma$, any integer $T\geq 2$ and any integer $p\geq 1$, we have the bound
 \begin{equation}
 \label{eq:stb:cor:stb:bnd:sup:res:p}
     \mathbb E\left[\sup_{0\leq \tau < \bar{\tau}_{\rm full}(\eta;\sigma,\delta, T)} \bar{N}_{\rm full}(\tau;\alpha)^p\right]\leq K^p
     \big( (\sigma^2 p)^{pr} + [\sigma^2 \ln T + \delta^2]^{pr} \big).
 \end{equation}
\end{cor}
\begin{proof}
Using the bounds \eqref{eq:mr:mo:bnds:w:j}
together with $\tau(t) \ge t$,
we may apply Young's inequality to compute
\begin{equation}
\label{eq:mr:bnd:alpha:times:w}
\begin{array}{lcl}
    \alpha^{2(r-j)p}  \mathbb E \sup_{0 \le \tau \le T} \norm{\bar{Y}_j(\tau)}_{j}^{2p}
    & \le &
    \alpha^{2(r-j)p}  \mathbb E \sup_{0 \le t \le T} \norm{Y_j(t)}_{j}^{2p}
    \\[0.2cm]
    & \le & K^{2p}  \alpha^{2(r-\ell)p}  \big[ (\sigma^2 p)^{p \ell}] + (\alpha^2)^{p \ell} \big]
    \\[0.2cm]
    & \le &
    K^{2p}   \big[   (\sigma^2 p)^{p r} \frac{\ell}{r} +
    (\alpha^2)^{pr} \frac{r-\ell}{r}
    +
    (\alpha^2)^{p r} \big]
\\[0.2cm]
& \le &
(2K)^{2p}   \big[   (\sigma^2 p)^{p r}     +
    (\alpha^2)^{p r} \big]
\end{array}
\end{equation}
for any $1 \le j \le r-1$. In addition, the ordering \eqref{eq:stb:ord:stopping:time} yields
\begin{equation}
    \sup_{0 \le \tau < \bar{\tau}_{\rm full}(\eta,\sigma,\delta,T) }
    \bar{\mathcal N}_{\rm res}(\tau)
    \le \sup_{0 \le \tau < \bar{\tau}_{\rm res}(\eta,\sigma,\delta,T) }
    \bar{\mathcal N}_{\rm res}(\tau).
\end{equation}
In particular, the desired bound follows from \eqref{eq:prp:stb:bnd:sup:res:p}.
\end{proof}

\begin{lem}
\label{lem:stb:a:vs:a1}
Suppose that \textnormal{(HNL)}, \textnormal{(HTw)}, \textnormal{(H$V_*$)} and \textnormal{(Hq)} are satisfied.
Pick two sufficiently small constants $\delta_\eta>0$ and $\delta_\sigma>0$.
Then for any $0 < \eta < \delta_\eta$,
any $0 \le \sigma < \delta_\sigma$ and any $0 \le \delta < \sqrt{\eta}$,
we have the bound
\begin{equation}
\label{eq:stb:bnd:prop:a:vs:a1}
    \mathbb{P} \big( \bar{\tau}_{\rm full}(\eta,\sigma,\delta, T) < T \big) \le
    \mathbb{P}
    \left( \sup_{0 \le \tau < \bar{\tau}_{\rm full}(\eta,\sigma,\delta,T) }
    \bar{\mathcal N}_{\rm full}(\tau;\alpha) \ge \eta \alpha^{2(r-1)}  \right).
\end{equation}
\end{lem}
\begin{proof}
For convenience, we define the events
\begin{equation}
\begin{array}{lcl}
    \bar{\mathcal{A}}_{1}
    & = & \{
    0< \bar{\tau}_{\rm full}(\eta,\sigma,\delta,T) < \bar{\tau}_\infty \hbox{ and }
    \sup_{0 \le \tau < \bar{\tau}_{\rm full}(\eta,\sigma,\delta,T) }
    \bar{\mathcal N}_{\rm full}(\tau) \ge \eta \alpha^{2(r-1)} \} ,
\\[0.2cm]
    \bar{\mathcal{A}}_{2}
    & = & \{  0<\bar{\tau}_\infty < T \hbox{ and } \sup_{0 \le \tau < \bar{\tau}_\infty(T) }
    \bar{\mathcal N}_{\rm full}(\tau) \le \eta \alpha^{2(r-1)} \},
\\[0.2cm]
    \bar{\mathcal{A}}_{3} & = & \{ \bar{\tau}_{\rm full}(\eta; \sigma,\delta, T) = 0 \}.
\end{array}
\end{equation}
Items (i) and (ii) of Proposition \ref{prop:res:mild} imply that
$\tau \mapsto \bar{\mathcal N}_{\rm full}(\tau)$ is continuous on $[0, \bar{\tau}_\infty)$ for almost every $\omega \in \Omega$.
This shows that we may write
\begin{equation}
\{ \bar{\tau}_{\rm full}(\eta,\sigma,\delta, T) < T \}
\subset \bar{\mathcal{A}}_0 \cup \bar{\mathcal{A}}_1 \cup \bar{\mathcal{A}}_2 \cup \bar{\mathcal{A}}_3,
\end{equation}
for some set $\bar{\mathcal{A}}_0$ with zero measure.
Note furthermore that we have $\mathbb P (\bar{ \mathcal{A}}_2) =0 $ by \eqref{eq:res:prob:reverse:zero}. Since $\bar{\mathcal N}_{\rm full}(0) = \delta^2 \alpha^{2(r-1)}$, the demand $\delta^2 < \eta$
ensures that  $\mathcal{A}_3$ is empty.  In particular, we see that
\begin{equation}
    \mathbb{P}(\bar{\tau}_{\rm full}(\eta,\sigma,\delta, T) < T ) \le \mathbb{P}( \bar{\mathcal{A}}_1 ),
\end{equation}
which implies the desired bound.
\end{proof}

\begin{prop}\label{prop:res:prob:bnd:bar}
Suppose that \textnormal{(HNL)}, \textnormal{(HTw)}, \textnormal{(H$V_*$)} and \textnormal{(Hq)} hold.
Then  there exist  constants $0<\mu<1$, $\delta_\eta>0$, and $\delta_\sigma>0$ such that, for  any $0<\eta\leq \delta_\eta$, any $0 < \sigma < \delta_\sigma$,
any $\delta^2 < \mu \eta$ and any integer $T\geq 3,$ we have
    \begin{equation}
        \mathbb P(\bar{\tau}_{\rm full}(\eta;\sigma,\delta, T)<T)\leq 3T^{1/(2e)} \exp\left(-\mu \frac{\eta^{1/r}}{\sigma^{2/r}}\right) .
        \label{eq:res:bnd:stop:time}
    \end{equation}
\end{prop}
 \begin{proof}

    Using the bounds \eqref{eq:stb:cor:stb:bnd:sup:res:p}
    together with the identification  \eqref{eq:stb:bnd:prop:a:vs:a1},
    we recall the shorthand \eqref{eq:mr:def:alpha}
    and apply the exponential Markov-type inequality in
    Corollary \ref{cor:prlm:gen:bnd:for:tail:distr}
    with
\begin{equation}
    \nu = r, \qquad \Theta_1 = K^{1/2} \sigma^{r} , \qquad \Theta_2 =  [ \alpha^{2 } / \sigma^{2}],
    \qquad \vartheta = \eta \alpha^{2(r-1)}
\end{equation}
to obtain
\begin{equation}
\begin{array}{lcl}
    P (\bar{\tau}_{\rm full}(\eta;\sigma,\delta, T)<T)
    & \le &
    3 \mathrm{exp}\left( \dfrac{ \alpha^2 }{2e \sigma^2} \right)
    \mathrm{exp}\left( \dfrac{- \eta^{1/r} \alpha^{2(1 - 1/r)} }{ 2e K^{1/r} \sigma^2} \right)
    \\[0.3cm]
    & = &
    3 T^{1/(2e)}
    \mathrm{exp}\left( \dfrac{K^{1/r} \delta^2 - \eta^{1/r} \alpha^{2(1 - 1/r)} }{ 2e K^{1/r} \sigma^2} \right).
\end{array}
\end{equation}
We now use the bound
\begin{equation}
    \alpha^{2(1-1/r)} \ge \frac{1}{2} (\sigma^2 \ln T)^{1-1/r} + \frac{1}{2} \delta^{2(1 - 1/r)}
\end{equation}
to find
\begin{equation}
\begin{array}{lcl}
    P (\bar{\tau}_{\rm full}(\eta;\sigma,\delta, T)<T)
    & \le &
    3 T^{1/(2e)}
    \mathrm{exp}\left( \dfrac{K^{1/r} \delta^2 - \frac{1}{2}\eta^{1/r} \delta^{2(1-1/r)}}{ 2e K^{1/r} \sigma^2} \right) \mathrm{exp}\left(    -\dfrac{ \eta^{1/r} [\ln T]^{1-1/r} }{ 4e K^{1/r} \sigma^{2/r}} \right).
\end{array}
\end{equation}
Upon choosing
\begin{equation}
    \mu = \min\{ (2^r K)^{-1},  (4eK^{1/r})^{-1} \},
\end{equation}
we see that
\begin{equation}
\frac{1}{2} \eta^{1/r} \delta^{2(1-1/r)}
\ge \frac{1}{2} \eta^{1/r} \delta^2 [ \mu \eta]^{-1/r} = \frac{1}{2}\delta^2 \mu^{-1/r} \ge K^{1/r} \delta^2
\end{equation}
and hence
\begin{equation}
\begin{array}{lcl}
    P (\bar{\tau}_{\rm full}(\eta;\sigma,\delta, T)<T)
    & \le &
    3 T^{1/(2e)} \mathrm{exp}\left(    - \mu \dfrac{ \eta^{1/r} [\ln T]^{1-1/r} }{ \sigma^{2/r}} \right).
\end{array}
\end{equation}
The desired bound now follows by noting that $\ln T \ge 1$ holds for $T \ge 3$.
\end{proof}

We are now ready to provide the proof of our main result. The key ingredient
that allows the time transform to be removed is that the exponential weight
in \eqref{eq:stb:def:n:eps:res} decays slower than its counterpart in
\eqref{eq:mr:def:n:res}. This enables us to show that stability loss in the original problem implies stability loss in a suitably chosen transformed problem, providing an ordering on the associated probabilities.

\begin{proof}
[Proof of Theorem \ref{thm:main}]
Note first that the bound
\eqref{eq:res:bnd:kappa:sigma:2} implies that $\tau(t) - \tau(t') \le 2 (t - t')$ for any pair $0 \le t'\le t$. This allows us to compute
\begin{equation}
\begin{array}{lcl}
    \int_0^{\tau(t)}  e^{-\beta(\tau(t) - \tau' )}\norm{\bar{Z}(\tau')}_{H^{k_* + 1}}^2 \, \mathrm d \tau'
    & = & \int_0^t  e^{- \beta(\tau(t) - \tau(t')) } \norm{\bar{Z}(\tau(t'))}_{H^{k_* + 1}}^2 \, \partial_t \tau(t') \mathrm d t'
\\[0.2cm]
    & \ge &
    \int_0^t  e^{- 2 \beta(t - t')  } \norm{Z(t')}_{H^{k_* + 1}}^2 \, \mathrm dt',
\end{array}
\end{equation}
which using $\tau(T) \le 2 T$ shows that
\begin{equation}
 \mathcal{N}_{\rm res}(t) \le \bar{\mathcal N}_{\rm res}\big(\tau(t) \big)   ,
 \qquad \qquad
 \mathcal{N}_{\rm full}(t;\sigma,\delta,T) \le \bar{\mathcal N}_{\rm full}\big(\tau(t) ; \sigma, \delta, 2T\big)
\end{equation}
for $0 \le t < \tau_\infty$. In particular,
if
\begin{equation}
    \sup_{0 \le t < t_{\rm st}(\eta;\sigma,\delta, T)} \mathcal{N}_{\rm full}(t;\sigma,\delta,T) \ge
    \eta [\delta^2 + \sigma^2 \ln T]^{r-1}
\end{equation}
holds, then we may use $T \ge 3$ to conclude
\begin{equation}
    \sup_{0 \le \tau < \tau( t_{\rm st}(\eta;\sigma,\delta, T) )} \bar{\mathcal N}_{\rm full}(\tau;\sigma,\delta, 2T) \ge
    \eta [\delta^2 + \sigma^2 \ln T]^{r-1}
    > \frac{\eta}{4^{r-1}}  [\delta^2 + \sigma^2 \ln (2T)]^{r-1},
\end{equation}
which implies
\begin{equation}
    \bar{\tau}_{\rm full}(4^{1-r}\eta; \sigma, \delta, 2T) <
    2T.
\end{equation}

Arguing as in Lemma \ref{lem:stb:a:vs:a1}
we obtain
\begin{equation}
\label{eq:stb:bnd:prop:a:vs:a1:org}
    \mathbb{P} \big( t_{\rm st}(\eta;\sigma,\delta, T) < T \big) \le
    \mathbb{P}
    \left( \sup_{0 \le t < t_{\rm st}(\eta;\sigma,\delta, T) }
    \mathcal{N}_{\rm full}(t;\alpha) \ge \eta \alpha^{2(r-1)}  \right)
\end{equation}
which in view of the discussion above
and the bound \eqref{eq:res:bnd:stop:time}
implies
\begin{equation}
\begin{array}{lcl}
    \mathbb{P} ( t_{\mathrm{st}}(\eta; \sigma, \delta, T) < T)
    & \le & \mathbb{P} \big( \bar{\tau}_{\rm full}(4^{1-r}\eta; \sigma, \delta, 2T) < 2 T \big)
    \\[0.2cm]
    & \le &
    3(2T)^{1/(2e)}\exp\left(-\mu 4^{1/r -1} \dfrac{ \eta^{1/r}}{\sigma^{2/r}}\right).
\end{array}
\end{equation}
The desired bound follows by adjusting the constant $\mu$ and noting that $3(2T)^{1/2e}\leq 2T$ for $T \ge 3$.
\end{proof}

\appendix

\section{List of main functions}
\label{app:def}

In this appendix we provide an overview of the main functions used in this paper,
building upon the presentation in \cite[App. A]{bosch2024multidimensional} and the naming conventions used in \cite{hamster2019stability,hamster2020transinv}.
Throughout this section, we assume that (HNL), (HTw), (H$V_*$) and (Hq) hold.
In addition, we take $\gamma \in \mathbb R$ together with $\xi \in L^2_Q$ and consider functions $u$ for which $u - \Phi_0 \in L^2(\mathcal D;\mathbb R^n)$.
For the purposes of clarity, we will be fully explicit with respect to the domains and codomains of our function spaces throughout this appendix.

We start by
choosing  a smooth non-decreasing cut-off function
\begin{equation}
    \chi_{\rm low}:\mathbb R\to [\tfrac14,\infty)
\end{equation}
that satisfies the properties
\begin{equation}
    \chi_{\rm low}(\theta)=\frac14 |\mathbb T|^{d-1},\quad \theta\leq \frac14 |\mathbb T|^{d-1},\quad \chi_{\rm low}(\theta)=\theta,\quad \theta\geq \frac12 |\mathbb T|^{d-1},
\end{equation}
together with a smooth non-increasing cut-off function
\begin{equation}
    \chi_{\rm high}:\mathbb R_+
    \to [0,1]
\end{equation}
that satisfies the properties
\begin{equation}
    \chi_{\rm high}(\theta)=1,\quad \theta\leq 2  ,\quad  \chi_{\rm high}(\theta)=0,\quad \theta\geq 3.
\end{equation}

These cut-offs can be used to define
\begin{equation}\label{eq:list:def:chi:h:l}
\chi_h(u,\gamma)=\chi_{\rm high}(\|u-T_\gamma\Phi_{0}\|_{L^2(\mathcal D;\mathbb R^n)})\quad\text{and}\quad
    \chi_l(u,\gamma)=\big[ \chi_{\rm low}\big(-\langle  u,T_\gamma\psi_{\rm tw}'\rangle_{L^2(\mathcal D;\mathbb R^n)} \big) \big]^{-1}.
\end{equation}
Upon taking
$u = T_{\gamma}[\Phi_0 + v]$ with
\begin{equation}
\label{eq:app:bnd:on:v:l2}
    \|v\|_{L^2(\mathcal D;\mathbb R^n)}
    \le \min \{ 1,  |\mathbb T|^{\frac{d-1}{2}} [4 \|\psi_{\rm tw}\|_{H^1(\mathbb R;\mathbb R^n)}]^{-1} \},
\end{equation}
we note that
\begin{equation}
    \chi_h(u, \gamma) = 1
    \quad\text{and}\quad
    \chi_l(u, \gamma) = - \big[\langle u,  T_{\gamma} \psi_{\rm tw}' \rangle_{L^2(\mathcal D;\mathbb R^n)} \big]^{-1}.
\end{equation}

We now introduce the scalar expressions
\begin{equation}
\label{eq:list:def:b:nu:tilde}
\begin{array}{lcl}
    b(u,\gamma)[\xi]
    & = &-\chi_h(u,\gamma)^2\chi_l(u,\gamma)\langle g(u)[\xi],T_{\gamma}\psi_{\rm tw}\rangle_{L^2(\mathcal D;\mathbb R^n)},
\\[0.2cm]
    \widetilde{\nu}(u,\gamma)
    & = &
    \frac{1}{2} \chi_h(u,\gamma)^4\chi_l(u,\gamma)^2\langle Qg^{T}(u) T_\gamma\psi_{\rm tw}, g^T(u) T_\gamma\psi_{\rm tw}\rangle_{L^2(\mathcal D;\mathbb R^n)},
\\[0.2cm]
\end{array}
\end{equation}
noting that $\widetilde{\nu} = \frac{1}{2}\|b(u,\gamma)\|^2_{HS( L^2_Q;\mathbb R)} $.
In addition, we introduce the functions
\begin{equation}
\label{eq:list:def:wt:k:c}
\begin{array}{lclcl}
    \widetilde{\mathcal{K}}_C(u, \gamma) &=&
       \chi_l(u,\gamma)
         \chi_h(u,\gamma) Q g(u)^\top T_{\gamma} \psi_{\rm tw}
         & \in & L^2_Q,
\\[0.2cm]
\mathcal{K}_C(u,\gamma)
&=& - \chi_h(u,\gamma) g(u) \widetilde{\mathcal{K}}_C(u, \gamma)
& \in & L^2(\mathcal D;\mathbb R^n),
\end{array}
\end{equation}
together with the scalar expression
\begin{equation}
\label{eq:list:def:a}
\begin{array}{lcl}
    a_\sigma(u,\gamma) &=&-\chi_l(u,\gamma)
    \Big[
      \langle f(u) , T_{\gamma} \psi_{\rm tw} \rangle_{L^2(\mathcal D;\mathbb R^n)}
      - \langle c_0 u + \sigma^2 \mathcal{K}_C(u,\gamma), \partial_x \psi_{\rm tw} \rangle_{L^2(\mathcal D;\mathbb R^n)}
\\[0.2cm]
& & \qquad \qquad\qquad
      + (1 + \sigma^2 \widetilde{\nu}(u,\gamma) ) \langle u, T_{\gamma} \psi''_{\rm tw} \rangle_{L^2(\mathcal D;\mathbb R^n)}
    \Big].
\end{array}
\end{equation}
We note that $b$ and $a_\sigma$ determine the evolution of the phase $\Gamma$ via
\eqref{eq:mr:def:Gamma}.

Exploiting the translational invariance of our nonlinearities and the noise,
we obtain the commutation relations
\begin{equation}
\label{eq:app:list:comm:rels:a:b}
     a_\sigma(u,\gamma)=a_\sigma(T_{-\gamma}u,0) ,
     \qquad
     b(u,\gamma)[\xi]=b(T_{-\gamma}u,0)[T_{-\gamma}\xi]
\end{equation}
and note that similar identities hold for $\widetilde{\nu}_\sigma$
and the cut-off functions
\eqref{eq:list:def:chi:h:l}. This allows
the dependence on the phase $\Gamma$ to be eliminated
from the evolution \eqref{eq:mr:weak:form:evol:v} for the frozen perturbation $V$.
This evolution is defined in terms of the nonlinearities
\begin{equation}
\label{eq:app:def:R:sigma:S}
\begin{array}{lcl}
    \mathcal R_\sigma(v) &= &
    \Delta_{x_\perp} v + (1 + \sigma^2\tilde{\nu}(\Phi_0 + v, 0) ) \partial_x^2 (\Phi_0 + v)
    + f(\Phi_0 + v) + c_0 \partial_x (\Phi_0 + v)
\\[0.2cm]
& & \qquad
    + \sigma^2 \partial_x \mathcal{K}_C(\Phi_0 + v, 0)
    +a_\sigma(\Phi_0 +v, 0)\partial_x(\Phi_0+v),
\\[0.2cm]
\mathcal S(v)[\xi]
&=&g(\Phi_0+v)[\xi]+\partial_x(\Phi_0+v)  b(\Phi_0+v,0)[\xi],
\end{array}
\end{equation}
where we take $v \in H^{2}(\mathcal D;\mathbb R^n)$; see \cite[Eq. (A.27)]{bosch2024multidimensional}.

Introducing the notation
\begin{equation}
    \mathcal{N}_f(v) = f(\Phi_0 + v) - f(\Phi_0) - Df(\Phi_0)v
\end{equation}
and noting that $-c_0 \Phi_0' = \Phi_0''  + f( \Phi_0)$
together with
$\mathcal{L}^*_{\rm tw} \psi_{\rm tw} = 0$,
we may write
\begin{equation}
\label{eq:list:expr:a:smooth}
\begin{array}{lcl}
    a_\sigma(\Phi_0 + v, 0) & = & -\chi_l(\Phi_0 + v,0)
    \Big[ \langle
    \mathcal{N}_f(v) , \psi_{\rm tw} \rangle_{L^2(\mathcal D;\mathbb R^n)}
    - \sigma^2 \langle  \mathcal{K}_C(\Phi_0 + v, 0)
    ,  \psi_{\rm tw}' \rangle_{L^2(\mathcal D;\mathbb R^n)}
    \\[0.2cm]
    & & \qquad \qquad \qquad
    + \sigma^2 \tilde{\nu}(\Phi_0 + v, 0)\langle  \Phi_0 + v,
    \psi_{\rm tw}'' \rangle_{L^2(\mathcal D;\mathbb R^n)}
    \Big] .
 \end{array}
\end{equation}
In a similar fashion, we may write
\begin{equation}
\label{eq:list:r:sigma:rewrite:ii}
\begin{array}{lcl}
\mathcal{R}_\sigma(v)
& = & \Delta_{x_\perp} v
+ \mathcal{L}_{\rm tw} v
+ \sigma^2 \widetilde{\nu}(\Phi_0 + v, 0) \partial_{xx} (\Phi_0 + v)
+ \mathcal{N}_f(v)
+ \sigma^2 \partial_x \mathcal{K}_C(\Phi_0 + v)
\\[0.2cm]
& & \qquad
+ a_\sigma(\Phi_0 + v;0, c_0) \partial_x(\Phi_0 + v).
\end{array}
\end{equation}
Upon introducing the notation
\begin{equation}
\label{eq:list:def:ups:i:ii}
\begin{array}{lcl}
        \Upsilon_I(v)
        &=& \widetilde{\nu}(\Phi_0 + v, 0) \partial^2_x [\Phi_0 +v],
        \\[0.2cm]
        \Upsilon_{II}(v)
        &=& - \chi_l(\Phi_0 + v,0)
        \widetilde{\nu}(\Phi_0 + v, 0)
        \langle \Phi_0 + v, \psi_{\rm tw}'' \rangle
        \partial_x[\Phi_0 + v],
\end{array}
\end{equation}
together with the combination
\begin{equation}
\label{eq:list:def:ups}
\Upsilon(v)
 =  \Upsilon_I(v) + \Upsilon_{II}(v)
\end{equation}
and the expressions
\begin{equation}
\label{eq:list:def:r:i:i:ups}
\begin{array}{lcl}
    \mathcal{R}_{I}(v)
    &=&
     \mathcal{N}_f(v)
    - \chi_l (\Phi_0 + v) \langle \mathcal{N}_f(v), \psi_{\mathrm{tw}} \rangle \partial_x [\Phi_0 + v],
\\[0.2cm]
    \mathcal{R}_{II}(v)
    & = & \partial_x \mathcal{K}_C(\Phi_0 +v , 0)
    - \chi_l (\Phi_0 + v) \langle \partial_x \mathcal{K}_C(\Phi_0 +v , 0), \psi_{\mathrm{tw}} \rangle \partial_x [\Phi_0 + v] ,
\\[0.2cm]
\end{array}
\end{equation}
we may recast \eqref{eq:list:r:sigma:rewrite:ii} in the convenient form
\begin{equation}
\label{eq:list:r:sigma:final}
\begin{array}{lcl}
\mathcal{R}_\sigma(v)
& = & \Delta_{x_\perp} v
+ \mathcal{L}_{\rm tw} v
+ \mathcal{R}_I(v) + \sigma^2 \mathcal{R}_{II}(v) + \sigma^2 \Upsilon(v)
\end{array}
\end{equation}
that appears in \eqref{eq:mr:decomp:r:sigma}.

\section{Moment bounds and tail probabilities}
Here we formulate several useful results concerning tail bounds
and maximal expectations. We note that related estimates can be found
in \cite{talagrand2005generic,veraar2011note}. The novel feature compared to \cite{bosch2024multidimensional} is that we allow general exponents $\nu \ge 1$
rather than fixing $\nu = 1$.

\begin{lem}\label{lem:prob:moment:to:tail}
   Consider a nonnegative random variable $X$. Suppose that  there exists two constants $\Theta_1>0$ and $\Theta_2>0$ so that the moment bound
\begin{equation}
\mathbb E X^{p} \leq \big[ p^p + \Theta_2^{p}] \Theta_1^{ 2p}
\end{equation}
holds for all integers $p \geq 1$. Then for every $\vartheta>0$ we have the estimate
\begin{equation}
\mathbb P(X\geq\vartheta) \leq 3\exp[\Theta_2/(2e)] \mathrm{exp} \left[-\frac{\vartheta}{2e \Theta_1^2} \right].
\end{equation}
\end{lem}
\begin{proof}
This is a slight restatement of \cite[Lem. B.1]{bosch2024multidimensional}.
\end{proof}

\begin{cor}
\label{cor:prlm:gen:bnd:for:tail:distr}
   Consider a nonnegative random variable $X$ and pick $\nu \ge 1$. Suppose that  there exist two constants $\Theta_1>0$ and $\Theta_2>0$ so that the moment bound
\begin{equation}
    \mathbb E X^{ p} \le \big[ p^{ \nu p }  + \Theta_2^{\nu p} \big] \Theta_1^{2p}
\end{equation}
holds for all integers $p \ge 1$.
Then for every $\vartheta > 0$ we have the bound
\begin{equation}
\mathbb P(X > \vartheta) \leq 3\exp[\Theta_2/(2e)] \left[-\frac{\vartheta^{1/\nu}}{2e \Theta_1^{2/\nu}} \right].
\end{equation}
\end{cor}
\begin{proof}
Write $Y = \sqrt[\nu]{X}$
and pick an integer $p' \ge 1$.
Applying H{\"o}lder's inequality
together with
the estimate
$\sqrt[\nu]{a + b} \le \sqrt[\nu]{a} + \sqrt[\nu]{b}$ for $a \ge 0$ and $b \ge 0$,
we obtain
\begin{equation}
 \mathbb E Y^{p'} =
 \mathbb E X^{p'/\nu}   \le
 \big[\mathbb E X^{p'} \big]^{1/\nu}
 \le \big[ (p')^{\nu p'}  + \Theta_2^{\nu p'} ]^{1/\nu} \Theta_1^{2 p' / \nu}
 \le \big[ (p')^{p'} + \Theta_2^{p'} \big] \Theta_1^{2 p' / \nu} .
\end{equation}
In view of Lemma \ref{lem:prob:moment:to:tail} this yields
\begin{equation}
   \mathbb  P\big(\sqrt[\nu]{X} > \sqrt[\nu]{\vartheta}\,\big) \le
    3\exp[ \Theta_2/(2e)] \mathrm{exp}\left[-\frac{\vartheta^{1/\nu}}{2e \Theta_1^{2/\nu}} \right],
\end{equation}
which leads to the stated bound.
\end{proof}

\begin{lem}[Lemma 2.3 of {\cite{hamster2020expstability}}]
\label{lem:tail:prob:to:mom}
Fix two constants $A \ge 2$ and $B > 0$ and consider a random variable $Z \ge 0$ that satisfies the bound
\begin{equation}
    P(Z > \vartheta) \le
    2 A \mathrm{exp}\left[ - \frac{\vartheta^2}{2e \Theta_1^2} \right]
\end{equation}
for all $\vartheta > 0$. Then for all
(reals) $p \ge 1$ we have
\begin{equation}
    \mathbb E Z^{2p}
    \le (p^p + \ln^p A) (8e\Theta_1^{2})^p .
\end{equation}
\end{lem}

\begin{cor}
\label{cor:prlm:tail:bnd:to:exp:bnd:gen}
Fix $\nu \ge 1$ together with two constants $A \ge 2$ and $B > 0$ and consider a random variable $Z \ge 0$ that satisfies the bound
\begin{equation}
    P(Z > \vartheta) \le
    2 A \mathrm{exp}\left[ -  \frac{\vartheta^{2/\nu}}{2e \Theta_1^{2/\nu}} \right]
\end{equation}
for all $\vartheta > 0$. Then  for all (real)
$p \ge 1$ we have
\begin{equation}
    \mathbb E Z^{2p} \le \big(p^{\nu p} + [\ln A]^{\nu p}  \big)
    \big( (8e \nu)^\nu \Theta_1^{2} \big)^p.
\end{equation}
\end{cor}
\begin{proof}
Upon writing $Y = \sqrt[\nu]{Z}$,
we observe that
\begin{equation}
  \mathbb  P(Y > \vartheta) \le 2 A \exp[-\vartheta^2/(2e \Theta_1^{2/\nu})].
\end{equation}
Applying Lemma \ref{lem:tail:prob:to:mom}
with $p' = \nu p \ge 1$, we obtain
\begin{equation}
\mathbb E Z^{2p} = \mathbb E Y^{2p'}   \le
( p'^{p'} +  \ln[A]^{p'} ) (8e \Theta_1^{2/\nu})^{p'}
= \big( ( \nu p )^{\nu p} +    [\ln A]^{\nu p}  \big) (8e)^{\nu p} \Theta_1^{2p},
\end{equation}
which can be absorbed by the stated bound since $\nu \ge 1$.
\end{proof}

\begin{cor}
\label{cor:prlm:exp:bnd:to:exp:bnd:with:n}
Pick $\nu \ge 1$ and consider
a collection of non-negative random variables
$\{Z_i\}_{i=1}^{N}$ for some integer $N \ge 2$.
Suppose that there exist two constants $\Theta_1 > 0$
and $\Theta_2 > 0$ so that the moment bound
\begin{equation}
 \mathbb   E Z_i^{2 p} \le \big[ p^{ \nu p } + \Theta_2^{\nu p} \big] \Theta_1^{2p}
\end{equation}
holds for all integers $p \ge 1$ and all $i \in \{1, \ldots, N \}$.
Then for all (real) $p \ge 1$  we have
\begin{equation}
\mathbb  E \max_{i \in \{1, \ldots, N\} }
    Z_i^{2p} \le
     \big(p^{\nu p} +  [\ln N +   \Theta_2]^{\nu p}\big) ( (16e\nu)^\nu \Theta_1^2)^{p}.
\end{equation}
\end{cor}
\begin{proof}
We first use Corollary \ref{cor:prlm:gen:bnd:for:tail:distr} to observe that
\begin{equation}
\mathbb  P \big(    \max_{i \in \{1, \ldots, N\} }
    Z_i > \vartheta)
    \le \mathbb P( Z_1 > \vartheta)
    + \ldots + \mathbb P( Z_N > \vartheta)
    \le 3 N  \mathrm{exp}[ \Theta_2 / (2e)] \mathrm{exp} \left[ - \frac{\vartheta^{2/\nu}}{2e \Theta_1^{2/\nu}}
    \right].
\end{equation}
Applying Corollary \ref{cor:prlm:tail:bnd:to:exp:bnd:gen}
we hence obtain
\begin{equation}
  \mathbb E \max_{i \in \{1, \ldots,  N\} }
    Z_i^{2p} \le
     \big(p^{\nu p} +  [\ln N + \ln[3/2] + \Theta_2/(2e)]^{\nu p}\big) ( (8e\nu)^\nu \Theta_1^2)^{p},
\end{equation}
which can be absorbed in the stated estimate.
\end{proof}

\bibliographystyle{klunumHJ}
\bibliography{ref}

\begin{thebibliography}{000}

\bibitem{adams2003sobolev}
R.~A. Adams and J.~J. Fournier (2003), {\em Sobolev spaces}.
\newblock Elsevier.

\bibitem{adams2024quasi}
Z.~P. Adams (2024), Quasi-Ergodicity of transient patterns in stochastic
  reaction-diffusion equations.
\newblock {\em Electronic Journal of Probability} {\bf 29}, 1--29.

\bibitem{adams2022isochronal}
Z.~P. Adams and J. MacLaurin (2025), The isochronal phase of stochastic pde and
  integral equations: Metastability and other properties.
\newblock {\em Journal of Differential Equations} {\bf 414}, 773--816.

\bibitem{agresti2023reaction}
A. Agresti and M. Veraar (2023), Reaction-diffusion equations with transport
  noise and critical superlinear diffusion: local well-posedness and
  positivity.
\newblock {\em Journal of Differential Equations} {\bf 368}, 247--300.

\bibitem{agresti2024criticalNEW}
A. Agresti and M. Veraar (2024), The critical variational setting for
  stochastic evolution equations.
\newblock {\em Probability Theory and Related Fields} {\bf 188}(3), 957--1015.

\bibitem{agresti2023b_reaction}
A. Agresti and M. Veraar (2024), Reaction-diffusion equations with transport
  noise and critical superlinear diffusion: Global well-posedness of weakly
  dissipative systems.
\newblock {\em SIAM Journal on Mathematical Analysis} {\bf 56}(4), 4870--4927.

\bibitem{Armero1996}
J. Armero, J. Sancho, J. Casademunt, A. Lacasta, L. Ramirez-Piscina, and F.
  Sagu{\'e}s (1996), External fluctuations in front propagation.
\newblock {\em Physical review letters} {\bf 76}(17), 3045.

\bibitem{beyn2004freezing}
W.-J. Beyn and V. Th{\"u}mmler (2004), Freezing solutions of equivariant
  evolution equations.
\newblock {\em SIAM Journal on Applied Dynamical Systems} {\bf 3}(2), 85--116.

\bibitem{Bloemker}
L.~A. Bianchi, D. Bl\"omker and P. Wacker (2017), Pattern size in Gaussian
  fields from spinodal decomposition.
\newblock {\em SIAM Journal on Applied Mathematics} {\bf 77}(4), 1292--1319.

\bibitem{birzu2018fluctuations}
G. Birzu, O. Hallatschek and K.~S. Korolev (2018), Fluctuations uncover a
  distinct class of traveling waves.
\newblock {\em Proceedings of the National Academy of Sciences} {\bf 115}(16),
  E3645--E3654.

\bibitem{bosch2024multidimensional}
M.~v.~d. Bosch and H.~J. Hupkes (2024), Multidimensional Stability of Planar
  Travelling Waves for Stochastically Perturbed Reaction-Diffusion Systems.
\newblock {\em arXiv preprint arXiv:2406.04232}.

\bibitem{bressloff2015nonlinear}
P.~C. Bressloff and Z.~P. Kilpatrick (2015), Nonlinear Langevin equations for
  wandering patterns in stochastic neural fields.
\newblock {\em SIAM Journal on Applied Dynamical Systems} {\bf 14}(1),
  305--334.

\bibitem{Bressloff}
P.~C. Bressloff and M.~A. Webber (2012), Front propagation in stochastic neural
  fields.
\newblock {\em SIAM Journal on Applied Dynamical Systems} {\bf 11}(2),
  708--740.

\bibitem{Chow}
P.-L. Chow (2014), {\em Stochastic partial differential equations}.
\newblock CRC Press.

\bibitem{cox2024sharp}
S. Cox and J. van Winden (2024), Sharp supremum and H$\backslash$" older bounds
  for stochastic integrals indexed by a parameter.
\newblock {\em arXiv preprint arXiv:2409.13615}.

\bibitem{dapratomild}
G. Da~Prato, A. Jentzen and M. R{\"o}ckner (2019), A mild It{\^o} formula for
  SPDEs.
\newblock {\em Transactions of the American Mathematical Society}.

\bibitem{da2014stochastic}
G. Da~Prato and J. Zabczyk (2014), {\em Stochastic equations in infinite
  dimensions}, Vol. 152.
\newblock Cambridge university press.

\bibitem{day1990large}
M.~V. Day (1990), Large deviations results for the exit problem with
  characteristic boundary.
\newblock {\em Journal of mathematical analysis and applications} {\bf 147}(1),
  134--153.

\bibitem{NunnoAdvMathFinance2011}
G. di~Nunno and B. Oksendal (eds.) (2011), {\em Advanced Mathematical Methods
  for Finance}.
\newblock Springer.

\bibitem{faris1982large}
W.~G. Faris and G. Jona-Lasinio (1982), Large fluctuations for a nonlinear heat
  equation with noise.
\newblock {\em Journal of Physics A: Mathematical and General} {\bf 15}(10),
  3025.

\bibitem{Climate}
C.~L.~E. Franzke, T.~J. O'Kane, J. Berner, P.~D. Williams and V. Lucarini
  (2015), Stochastic climate theory and modeling.
\newblock {\em Wiley Interdisciplinary Reviews: Climate Change} {\bf 6}(1),
  63--78.

\bibitem{freidlin1998random}
M.~I. Freidlin and A.~D. Wentzell (1998), Random perturbations.
\newblock In: {\em Random perturbations of dynamical systems}.
\newblock Springer, pp. 15--43.

\bibitem{GarciaSpatiallyExtended}
J. Garc{\'\i}a-Ojalvo and J. Sancho (2012), {\em Noise in spatially extended
  systems}.
\newblock Springer Science \& Business Media.

\bibitem{Gawarecki}
L. Gawarecki and V. Mandrekar (2010), {\em Stochastic differential equations in
  infinite dimensions: with applications to stochastic partial differential
  equations}.
\newblock Springer Science \& Business Media.

\bibitem{goldys2003transition}
B. Goldys and J. Van~Neerven (2003), Transition semigroups of Banach
  space-valued Ornstein--Uhlenbeck processes.
\newblock {\em Acta Applicandae Mathematica} {\bf 76}, 283--330.

\bibitem{gowda2015early}
K. Gowda and C. Kuehn (2015), Early-warning signs for pattern-formation in
  stochastic partial differential equations.
\newblock {\em Communications in Nonlinear Science and Numerical Simulation}
  {\bf 22}(1), 55--69.

\bibitem{hairer2009introduction}
M. Hairer (2009), An Introduction to Stochastic PDEs.
\newblock \url{http://www.hairer.org/notes/SPDEs.pdf}.

\bibitem{hamster2019stability}
C.~H.~S. Hamster and H.~J. Hupkes (2019), Stability of Traveling Waves for
  Reaction-Diffusion Equations with Multiplicative Noise.
\newblock {\em SIAM Journal on Applied Dynamical Systems} {\bf 18}(1),
  205--278.

\bibitem{hamster2020diag}
C.~H.~S. Hamster and H.~J. Hupkes (2020), Stability of traveling waves for
  systems of reaction-diffusion equations with multiplicative noise.
\newblock {\em SIAM Journal on Mathematical Analysis} {\bf 52}(2), 1386--1426.

\bibitem{hamster2020expstability}
C.~H.~S. Hamster and H.~J. Hupkes (2020), Stability of traveling waves on
  exponentially long timescales in stochastic reaction-diffusion equations.
\newblock {\em SIAM Journal on Applied Dynamical Systems} {\bf 19}(4),
  2469--2499.

\bibitem{hamster2020transinv}
C.~H.~S. Hamster and H.~J. Hupkes (2020), Travelling waves for
  reaction--diffusion equations forced by translation invariant noise.
\newblock {\em Physica D: Nonlinear Phenomena} {\bf 401}, 132233.

\bibitem{Inglis}
J. Inglis and J. MacLaurin (2016), A general framework for stochastic traveling
  waves and patterns, with application to neural field equations.
\newblock {\em SIAM Journal on Applied Dynamical Systems} {\bf 15}(1),
  195--234.

\bibitem{karczewska2005stochastic}
A. Karczewska (2005), Stochastic integral with respect to cylindrical Wiener
  process.
\newblock {\em arXiv preprint math/0511512}.

\bibitem{KUEHN2013FRAME}
C. Kuehn (2013), A {M}athematical {F}ramework for {C}ritical {T}ransitions:
  {N}ormal {F}orms, {V}ariance and {A}pplications.
\newblock {\em Journal of nonlinear science} {\bf 23}(3), 457--510.

\bibitem{KUEHN2013KPP}
C. Kuehn (2013), Warning {S}igns for {W}ave {S}peed {T}ransitions of {N}oisy
  {F}isher--{K}{P}{P} invasion fronts.
\newblock {\em Theoretical Ecology} {\bf 6}(3), 295--308.

\bibitem{KuehnReview}
C. Kuehn (2020), Travelling waves in monostable and bistable stochastic partial
  differential equations.
\newblock {\em Jahresbericht der Deutschen Mathematiker-Vereinigung} {\bf 122},
  73--107.

\bibitem{Kuske2017}
R. Kuske, C. Lee and V. Rottsch{\"a}fer (2017), Patterns and coherence
  resonance in the stochastic Swift-Hohenberg equation with Pyragas control:
  The Turing bifurcation case.
\newblock {\em Physica D: Nonlinear Phenomena}.

\bibitem{maclaurinnew}
J. MacLaurin (2023), Phase Reduction of Waves, Patterns, and Oscillations
  Subject to Spatially Extended Noise.
\newblock {\em SIAM Journal on Applied Mathematics} {\bf 83}(3), 1215--1244.

\bibitem{DaPratoZab}
G. Prato and J. Zabczyk (1992), {\em Stochastic equations in infinite
  dimensions}.
\newblock Cambridge University Press, Cambridge New York.

\bibitem{prevot2007concise}
C. Pr{\'e}v{\^o}t and M. R{\"o}ckner (2007), {\em A concise course on
  stochastic partial differential equations}, Vol. 1905.
\newblock Springer.

\bibitem{russo1991integrales}
F. Russo and P. Vallois (1991), Int{\'e}grales progressive, r{\'e}trograde et
  sym{\'e}trique de processus non adapt{\'e}s.
\newblock {\em Comptes rendus de l'Acad{\'e}mie des sciences. S{\'e}rie 1,
  Math{\'e}matique} {\bf 312}(8), 615--618.

\bibitem{salins2019}
M. Salins and K. Spiliopoulos (2021), Metastability and exit problems for
  systems of stochastic reaction--diffusion equations.
\newblock {\em The Annals of Probability} {\bf 49}(5), 2317--2370.

\bibitem{schimansky1991}
L. Schimansky-Geier and C. Z{\"u}licke (1991), Kink propagation induced by
  multiplicative noise.
\newblock {\em Zeitschrift f{\"u}r Physik B Condensed Matter} {\bf 82}(1),
  157--162.

\bibitem{shardlow}
T. Shardlow (2005), Numerical simulation of stochastic PDEs for excitable
  media.
\newblock {\em Journal of computational and applied mathematics} {\bf 175}(2),
  429--446.

\bibitem{talagrand2005generic}
M. Talagrand (2005), {\em The generic chaining: upper and lower bounds of
  stochastic processes}.
\newblock Springer Science \& Business Media.

\bibitem{veraar2011note}
M. Veraar and L. Weis (2011), A note on maximal estimates for stochastic
  convolutions.
\newblock {\em Czechoslovak mathematical journal} {\bf 61}(3), 743.

\bibitem{vinals1991numerical}
J. Vi{\~n}als, E. Hern{\'a}ndez-Garc{\'\i}a, M. San~Miguel and R. Toral (1991),
  Numerical study of the dynamical aspects of pattern selection in the
  stochastic Swift-Hohenberg equation in one dimension.
\newblock {\em Physical Review A} {\bf 44}(2), 1123.

\bibitem{Zhang}
J. Zhang, A. Holden, O. Monfredi, M. Boyett and H. Zhang (2009), Stochastic
  vagal modulation of cardiac pacemaking may lead to erroneous identification
  of cardiac chaos.
\newblock {\em Chaos: An Interdisciplinary Journal of Nonlinear Science} {\bf
  19}(2), 028509.

\bibitem{ZUMHOW1998}
K. Zumbrun and P. Howard (1998), Pointwise {S}emigroup {M}ethods and
  {S}tability of {V}iscous {S}hock {W}aves.
\newblock {\em Indiana Univ. Math. J.} {\bf 47}(3), 741--871.

\end{thebibliography}

\end{document}